\documentclass[11pt,reqno]{amsart}
\usepackage{xifthen,setspace,bm}

\usepackage[utf8]{inputenc}
\usepackage[T1]{fontenc} 
\usepackage{bm}
\usepackage{amsmath,amsthm,amsfonts,amssymb,mathrsfs,stmaryrd,bigints,graphicx}
\usepackage{dsfont}
\newtheorem{maintheorem}{Theorem}
\usepackage{mathtools}
\usepackage[usenames]{color}
\usepackage{thmtools, thm-restate}

\usepackage[colorlinks=true,linkcolor=blue]{hyperref}
\DeclareFontFamily{U}{matha}{\hyphenchar\font45}
\DeclareFontShape{U}{matha}{m}{n}{
      <5> <6> <7> <8> <9> <10> gen * matha
      <10.95> matha10 <12> <14.4> <17.28> <20.74> <24.88> matha12
      }{}
\DeclareSymbolFont{matha}{U}{matha}{m}{n}
\DeclareMathSymbol{\varleftrightarrow}{3}{matha}{"D8}
\DeclareMathSymbol{\nvarleftrightarrow}{3}{matha}{"DC}

\newcommand{\sett}{\textup{\textsf{Set}}}
\usepackage{epstopdf} 
\usepackage{amssymb}
\usepackage{setspace}
\usepackage{enumerate}
\usepackage{bigstrut}
\usepackage{esvect}
\usepackage{multirow}
\usepackage{mathtools,xparse}
\usepackage{mathrsfs}
\usepackage{hyperref}
\usepackage{xcolor}
\newtheorem*{lemma*}{Lemma}
\newtheorem{theorem}{Theorem}[section]

\newtheorem{lemma}[theorem]{Lemma}
\newtheorem{proposition}[theorem]{Proposition}

\newtheorem{remark}[theorem]{Remark}
\newtheorem{assumption}{Assumption}

\theoremstyle{definition}
\newtheorem{definition}[theorem]{Definition}

\allowdisplaybreaks
\usepackage{chngcntr}

\renewcommand{\P}{\mathbb{P}}

\newcommand{\E}{{\mathbb{E}}}
\newcommand{\1}{\mathds{1}}

\newcommand{\Z}{\mathbb{Z}}

\newcommand{\e}{\varepsilon}

\newcommand{\extr}{\textup{ext}}
\newcommand{\trace}{\textup{\textsf{VRange}}}

\newcommand{\tiloop}{\widetilde{\Gamma}}
\newcommand{\cable}{\widetilde{\Z}^d}

\newcommand{\B}{\mathbb{B}}

\newcommand{\N}{\mathbb{N}}

	\renewcommand{\P}{\mathbb{P}}

\newcommand{\ble}{\textup{\textsf{BLE}}}
\newcommand{\ete}{\textup{\textsf{ETE}}}
\newcommand{\ite}{\textup{\textsf{ITE}}}

\newcommand{\cA}{\mathcal{A}}

\newcommand{\cC}{\mathcal{C}}

\newcommand{\cG}{\mathcal{G}}

\newcommand{\cL}{\mathcal{L}}

\newcommand{\wt}{\widetilde}

\newcommand{\ar}{\leftrightarrow}
\newcommand{\arr}{\overset{r}{\leftrightarrow}}

\newcommand{\bij}{\leftrightarrow}
\newcommand{\lbij}{\longleftrightarrow}

\renewcommand{\emptyset}{\varnothing}

\newcommand{\set}[1]{\{#1\}}

\renewcommand{\setminus}{\backslash}

\newcommand{\IIC}{\textup{IIC}}

\newcommand{\bwt}{\widetilde}

\def\ba{\begin{align}}
\def\ea{\end{align}}
\def\bs{\begin{split}}
\def\es{\end{split}}

\begin{document}

\title[Ant on loops]{The ant on loops: Alexander-Orbach conjecture for the critical level set of the Gaussian free field}
\author{Shirshendu Ganguly and Kyeongsik Nam}

\begin{abstract} Alexander and Orbach (AO) in 1982 \cite{AO82} conjectured that the simple random walk on critical percolation clusters (also known as the ant in the labyrinth) in Euclidean lattices exhibit mean field behavior; for instance, its spectral dimension is $4/3$, its speed exponent is $1/3$ and so on.  While known to be false in low dimensions, this is expected to be true above the upper critical dimension of six. First rigorous results in this direction go back to Kesten \cite{kesten} who verified the mean field behavior for the critical branching process. Subsequently, after a series of developments, in a breakthrough \cite{kn2}, Kozma and Nachmias established the universality of the critical exponents for bond percolation in dimensions bigger than $19$ for the usual lattice and bigger than $6$ for the spread out lattice. We investigate the validity of the AO conjecture for the critical level set of the Gaussian Free Field (GFF), a canonical dependent percolation model of central importance. In an influential work, Lupu \cite{lupu} proved that for the cable graph of $\Z^d$ (which is obtained by also including the edges as segments), the signed clusters of the associated GFF  are given by the corresponding clusters induced by a Poisson loop soup introduced in \cite{le1}. This reduces the study of the critical level set on the cable graph to the analysis of critical percolation in the loop model. 

In 2021, Werner in \cite{werner} put forth an evocative picture for the critical behavior in this setting drawing an analogy with the usual bond case in high dimensions, with insightful heuristics and conjectures and some results. Subsequently, building on this, in an impressive recent development, Cai and Ding in \cite{cd} established the universality of the extrinsic one arm  exponent for all $d > 6.$  In this article, we carry this program further, and consider the random walk on sub-sequential limits of the cluster of the origin conditioned to contain far away points. These form candidates for the Incipient Infinite Cluster (IIC), first introduced by Kesten in the planar case as the critical percolation cluster conditioned to percolate. Inspired by the program in \cite{kn2}, introducing several novel ideas to tackle the long range nature of this model and its effect on the intrinsic geometry of the percolation cluster, we establish that the AO conjecture indeed holds for any sub-sequential IIC for all large enough dimensions.  
\end{abstract}
  

\address{Department of Statistics, UC Berkeley, USA}
\email{sganguly@berkeley.edu}
\address{Department of Mathematical Sciences, KAIST, South Korea}
\email{ksnam@kaist.ac.kr}

\maketitle


\tableofcontents

\section{Introduction}
Percolation was introduced by 
 Broadbent and Hammersley in 1957 to study how 
 the random properties of a medium influence the passage of a fluid through it. The simplest
version is bond percolation on the Euclidean 
lattice $\Z^d$ where every edge appears independently with probability $p \in (0,1).$
The percolation cluster containing a point $x$, 
denoted
$\cC(x)$, is simply the connected component of $x$ of the retained bonds.
There is a critical value $p_c \in (0, 1)$ such that if $p < p_c$ then all clusters are finite, while for
$p>p_c$ there is an infinite cluster.
Random walks on percolation clusters were introduced by De Gennes in 1976 who introduced 
the evocative terminology
 ‘the ant in the labyrinth’. 
 {Let ${p}_n(x, y)$ for $n\ge 0$} and $x,y \in \Z^d$ be the associated heat kernel.
  For the supercritical phase ($p > p_c$) it is well known that on the event that the origin in the infinite component, $p_{n}(0, \cdot)$ converges to a Gaussian distribution 
  as $n\to \infty $ with PDE techniques introduced by
Nash in the 1950s, playing an important role in the 
arguments \cite{gaussian}.
The critical case $p = p_c$ is much more subtle exhibiting intricate fractal structure usually making it significantly more challenging to analyze.
Alexander and Orbach (AO) conjectured in 1982 \cite{AO82} that for critical percolation clusters, the spectral dimension $d_s$ is equal to $4/3$ in all dimensions (the relation to the heat kernel is given by $p_n(0,0) \sim n^{
-d_s/2}$). 
However, one only expects this to hold in high enough dimensions, i.e. above the so called upper critical dimension which is expected to be $6,$ with results of \cite{ch} refuting such mean-field behavior for $d<6.$  

The AO conjecture was first verified in 1986 for the critical percolation on a tree leading to a critical branching process by Kesten in his pioneering work in \cite{kesten} who also concurrently proved non-trivial bounds in the planar case in the extrinsic metric (this was recently extended to the case of the intrinsic metric in \cite{ganguly}). 
Subsequently, building on breakthroughs by Barlow \cite{barlow1}, there was a slew of developments aimed towards proving analogous results in high dimensional lattices. This started with the work \cite{barlow} who considered oriented critical percolation and culminated with the influential work of Kozma and Nachmias \cite{kn2} on the standard bond percolation.
The underlying graphs considered include usual Euclidean lattices with dimensions greater than nineteen, as well as spread out lattices in dimensions greater than six.
A particularly interesting object introduced by Kesten is the infinite labyrinth, i.e., the critical percolation cluster \emph{conditioned to be infinite} in the \emph{planar case}. This was defined by showing that the sequence of measures $$\P(\cdot \subset \cC(0) \mid 0\lbij \partial [-n,n]^2)$$ has a weak limit as $n\rightarrow \infty$, which Kesten termed as the \textit{Incipient infinite cluster} (IIC).
In high dimensions, the proof of the existence of the IIC relies on the powerful lace expansion techniques introduced and developed in \cite{hara1,hara2,hara3}.
Beyond considering the IIC and its spectral dimension, in \cite{kn1}, Kozma and Nachmias also established the Euclidean arm exponent to be $2$ in the above setting i.e. 
\begin{equation}\label{onearmextrinsic}
\P(0\lbij \partial [-n,n]^d) \approx \frac{1}{n^2}.
\end{equation}

More recently, over the last two decades, there has been a significant interest in analyzing percolation models
 that exhibit some dependence, motivated
by field theory and random geometry considerations emerging from disordered systems with long-range interactions.
 A common feature of these models is the 
 strength of the
correlations between local observables, which exhibit power law decay like $|x-y|^{-a}$
 as $|x-y| \to \infty$
for a certain exponent $a > 0$.

Examples include random interlacements describing the local limit of a random walk trace on the Euclidean lattice $\Z^d$, 
percolation induced by the voter model, 
loop soup percolation \cite{fkg,loop1,lupu}, as well level-set percolation of random fields. The last example will be the object of investigation in this article. 
While random fields considered in the literature include  Gaussian ensembles relating to various classes of functions such as
randomized spherical harmonics (Laplace eigenfunctions) \cite{lap1,lap2,lap3} at high frequencies,
 we will focus on the canonical case of the massless Gaussian free field.  \\

\emph{In fact, as will be made clear shortly, both the GFF level set percolation and the loop soup percolation are intimately connected and will feature centrally in this paper which should be considered as a part of this general program aimed at understanding such dependent percolation models.}\\

 The level set percolation for the GFF, originally investigated by Lebowitz and Saleur in \cite{ls}    
 as a canonical percolation model with slow, algebraic decay of correlations,  has witnessed some particularly intense activity recently.
  In a recent breakthrough,  Duminil-Copin, Goswami,
Rodriguez and Severo 
 \cite{gff5}  established the
  equality of parameters dictating the onset of various forms of sub-critical to super-critical transitions. 
  A detailed understanding of both the super-critical and sub-critical regimes have been accomplished across the works of Drewitz,   
R\'ath and Sapozhnikov  \cite{drs},   
Popov and R\'ath  \cite{pr}, Popov and Teixeira  \cite{pt}, Goswami,
Rodriguez and Severo  \cite{grs}. 
Nonetheless, despite these extensive works, the critical regime has resisted analysis for the most part. 

The starting point of this paper is the beautiful article of Lupu \cite{lupu}
 where he discovered an important connection between the signed clusters of the
  Gaussian free field on $\widetilde \Z^d$ (the metric graph of $\Z^d$ obtained 
  by including the edges as one dimensional segments, often termed as the {cable graph} in the literature)
   and a loop soup model (a Poisson process on the space of loops) introduced earlier in \cite{le1}. 
   In particular, he showed that {$0$ is the critical level for percolation} and, 
   the percolation cluster of a given point, say the origin, agrees with that in the corresponding loop soup. 
   Owing to this connection, we will simply consider the loop soup model for the majority of this article.

Based on Lupu's coupling, in an evocative paper \cite{werner}, Werner had initiated the study of high dimensional level set
 percolation for the GFF via the analysis of the corresponding loop soup drawing parallel 
 with the mean field bond percolation picture established in high enough dimensions, 
 putting forth a variety of insightful arguments and conjectures. Werner asserted that in high dimensions (i.e. $d > 6$), the GFF on
$\Z^d$ becomes 
asymptotically independent and thus shares similar behavior with
bond percolation.
 A key step toward establishing this general phenomenon was accomplished in the recent work of Cai and Ding \cite{cd} 
 where building on the aforementioned earlier work of  Kozma and Nachmias \cite{kn1} on the extrinsic 
 one arm exponent, the universality of the same was established for this model in $d>6$ by showing that indeed \eqref{onearmextrinsic} continues to hold. 

In this article we carry this program further establishing universal mean field behavior in high enough dimensions of several critical exponents governing the chemical one arm probability, volume growth, spectral dimension, speed of random walk, and resistance growth, albeit some of the exponents being related to each other by what has come to be known as the Einstein relations \cite{lee}.  \\
 
 \emph{In particular, we take up the question of establishing the AO conjecture  for the critical percolation cluster for the loop soup.}\\

In light of the discussion above indicating the connection to the loop soup, we will take our underlying geometry
 to be the cable graph $\widetilde \Z^d$ instead of the usual lattice $\Z^d.$
Owing to previous results by Barlow \cite{barlow1}, the task reduces to establishing certain geometric properties of the IIC.
 As in \cite{kn2},
 we will not focus on the proof of the existence of the IIC, which remains an interesting problem, 
 but rather prove the requisite properties for any sub-sequential limit of the sequences of measures obtained 
 by conditioning on either $\{0\lbij x\}$ or $\{0\lbij \partial [-n,n]^d\}$ as either $|x|$ or $n$ goes to infinity. 
 
While we follow their framework from \cite{kn2}, several new ideas are needed to handle the long range nature of the loop soup. We provide a detailed overview of the new ingredients in Section \ref{iop} but we first turn to setting up the stage formally and stating our main results.

\subsection{Models}
Although our main object of interest is the GFF on the cable graph $\widetilde\Z^d$, 
we first introduce the more standard discrete GFF living on the lattice $\Z^d$ ($d\ge 3$). 
The  discrete GFF $\Phi=\{ {\phi}_v\}_{v\in  {\Z}^d}$ is a mean-zero Gaussian field, whose covariance
 is given by the Green's function on $\Z^d$ which we denote by $G(\cdot,\cdot)$. Thus 
\begin{align*}
    \text{Cov}(\phi_x,\phi_y) := G(x,y),\qquad \forall x,y\in \Z^d.
\end{align*}
It is well known that there exist $C_1,C_2>0$ such that (see for instance \cite{rw})
\begin{align*}
C_1|x-y|^{2-d}\le G(x,y)\le C_2|x-y|^{2-d}.
\end{align*}
Next, we define the cable graph $\cable$. 
 For each edge $e \in E(\Z^d),$ say with endpoints $x$ and 
 $y$, let $I_e$ be the line segment joining $x$ and $y$, 
 scaled by $d,$ thus a compact interval of length $d$ which
  will often also be denoted by $[x,y]$ or interchangeably $[e]$. Then  $\cable:= \cup_{e\in E(\Z^d)}I_e.$
Next, the 
 GFF $\widetilde \Phi:=\{\tilde{\phi}_v\}_{v\in \cable}$ on  $\cable$ is defined as follows.  
 Given the discrete GFF  $\{{\phi}_v\}_{v\in {\Z}^d}$, set $\tilde{\phi}_v:= \phi_v$  for all lattice points $v\in \Z^d$.  We next define the interpolating scheme specifying the Gaussian values on $I_e$. 
 Given $\Phi,$ for
each interval $I_e$ with $e=(x,y)$ (i.e. $x$ and $y$ are endpoints of the edge $e$), 
$\{\tilde{\phi}_v\}_{v\in I_e}$
is given by an independent Brownian
bridge of length $d$ with the further specification that the variance at time $1$ is $2$,
 conditioned on $\tilde{\phi}_x = \phi_x$ and  $\tilde{\phi}_y = \phi_y$.

The corresponding super level set is given by 
\begin{align*}
    \tilde{E}^{\ge h} := \{ x\in \cable: \tilde{\phi}_x \ge h\}.
\end{align*}
Note that this is a percolation model, albeit a rather dependent one. 
{In \cite{lupu}, Lupu proved that the critical level for percolation exactly equals zero.} 
In other words, $\tilde{E}^{\ge h}$  almost surely percolates (i.e. contains
an infinite connected component) when $h<0$, while $\tilde{E}^{\ge 0}$ does not percolate.
 The focus of this article will indeed be this critical regime. Lupu also established an isomorphism theorem for this model connecting it to an already alluded to loop percolation model which we will define shortly.
  Relying on this, in \cite{lupu} he had further established the following all important two-point estimate: There exist $C_1,C_2>0$ such that for any $x,y\in \Z^d$,
\begin{align} \label{two point0}
   C_1 |x-y|^{2-d} \le  \P(x \ar y) \le C_2 |x-y|^{2-d}.
\end{align}
where $x\overset{\tilde{E}^{\ge 0}}{\lbij} y$,  or simply $x \ar y$,
 will be the notation to denote that $x,y$ belong to the same connected component of {$\tilde{E}^{\ge 0}$}. 
This in particular implies the triangle condition 
\begin{align} \label{triangle condition}
\sum_{x,y\in \Z^d}\P(0\bij x)\P(x\bij y)\P(y\bij 0) < \infty. 
\end{align}
As has been indicated in \cite{triangle1, triangle2}, this alone should suffice to deduce various mean field behavior. 
An observable of particular interest is the arm probability $$\pi_1(n):= \P\big(0 \lbij \partial\B_n \big),$$ i.e., the probability that the percolation cluster of the origin (henceforth denoted by $\widetilde\cC(0):=\{x\in \widetilde \Z^{d}: 0\overset{\tilde{E}^{\ge 0}}{\lbij}x  \}$) extends to the boundary of {$\B_n$, the Euclidean box $[-n,n]^d,$}  which was already shown in \cite{lupu} to converge to $0$ as $n\to \infty,$ i.e. there is no percolation at criticality. Towards obtaining quantitative 
estimates, 
Ding and Wirth \cite{dw} employed
a martingale argument and proved various polynomial bounds for the one-arm probability in various dimensions. 
More recently sharp bounds were obtained in \cite{cd} who showed $\pi_1(n) \sim n^{-2}$. Several other estimates were established in this work, many of which we will rely on in our arguments as well. 

In this paper we will primarily focus on the intrinsic geometry of $\widetilde\cC(0)$. Our results will establish universality of the mean field exponents for this model in all large enough dimensions. As already briefly alluded to,  a natural framework to study such questions is through the analysis of the incipient infinite cluster (IIC) for the GFF, provided it can be made sense of. The IIC was first introduced by Kesten in \cite{kesten1} for planar critical bond percolation.
In our setting, the IIC should be viewed as $\widetilde\cC(0)$ conditioned to percolate \emph{to infinity}. Two particularly natural attempts to make sense of this involves considering the measures
\begin{align} \label{limitiic}
\nu_n:= \P(\cdot \subset \widetilde\cC(0)\mid 0 \lbij \partial \B_n) \text{ or}\,\,  \nu^x:=\P(\cdot \subset \widetilde\cC(0)\mid 0 \lbij x),
\end{align}
where  $n\in \N$ and $x\in \Z^d.$ One then aims to show that these well defined measures have a weak limit as $n$ or $|x|$ goes to infinity. 
{Note that by local compactness}, both $\{\nu_n\}_n$ and $\{\nu^x\}_x$ form a tight sequence of measures and we will denote by $\widetilde\P_{\IIC}$  any sub-sequential limit.  {There are some subtle measure theoretic issues around the weak convergence that one needs to address to indeed make this precise. This will be done in Section \ref{section 2.2}.}

For the usual bond percolation on the discrete lattice, for $d=2$, the convergence of $\nu_n$ as $n\to \infty$ was shown by Kesten \cite{kesten1}. The convergence of $\nu^x$ to the same limit as $|x|\to \infty$ was later established in \cite{iic}. The higher dimensional construction however turned out to be significantly more complicated.  This was carried out for 
bond  percolation on the sufficiently spread-out lattice
($d > 6$) (we refer the reader to \cite{iic,iic2} for details).

Note that in our setting, i.e., of the cable graph, under $\widetilde\P_{\IIC},$ the underlying graph $\widetilde \cC(0)$  is a subset of the cable graph and will hence be a connected union of edges and partial edges. 
However since $\Z^d$ {has degree $2d,$} for our purposes, i.e. to study its volume growth, define and study the random walk on it, we will mostly consider $\cC(0)$ obtained by removing the partial edges. We will denote this measure as $\P_{\IIC}$ (a formal treatment of these notions will be presented in the next section).

With the above preparation, we are now in a position to state our main result, establishing the AO conjecture for any subsequential $\P_{\IIC}.$ We also establish exponents governing the speed of the random walk, as well as its trace.

\subsection{Main result}

For a graph $G$ and $n\in \N$, 
let $p^G_n(\cdot, \cdot)$ be the associated heat kernel. In particular, $p^G_n(0, 0)$ is the return probability of the simple random walk starting from the origin
after $n$ steps.

\begin{maintheorem} \label{main}
Let $d>20$ and consider any sub-sequential limit 
$\P_{\textup{IIC}}$ of the measures $\nu^x$ defined above in \eqref{limitiic} and a graph $\cG$ sampled from it. {Let $\{X_i\}_{i\ge 1}$  be the discrete  time}
simple random walk on $\cG$. Then $\P_{\textup{IIC}}$-a.s.,
    \begin{align*}
        \lim_{n\rightarrow \infty }  \frac{\log p^\cG_{2n}(0,0)}{\log n} = - \frac{2}{3},\quad   \lim_{r\rightarrow \infty }  \frac{\log \E \tau_r}{\log r} = 3,\quad \lim_{n\rightarrow \infty }  \frac{\log |\{X_1,X_2,\ldots,X_n\}|}{\log n} =\frac{2}{3}  \ \textup{a.s.}
    \end{align*}
Above,   $\tau_r$ is the hitting time of {(intrinsic)} {distance} $r$ from the origin (the expectation $\E$ is only over
the randomness of the walk). Further, $\{X_1,X_2,\ldots, X_n\}$ denotes the trace of the random walk up to time $n,$ and hence the last conclusion is about the growth rate of the range of the random walk.  
\end{maintheorem}

Note that while the upper critical dimension is supposed to be $6$, the above theorem is only stated for dimensions bigger than $20.$ The reason behind this is of a technical nature and will be discussed shortly. Finally, while the above result is stated in terms of sub-sequential limits of the sequence of measures $\nu^x,$ in accordance with how it appears in \cite{iic}, the same conclusion holds for any subsequential limit of the measures $\nu_n$. This will be elaborated on at the very end after the proof of Theorem \ref{main} (see Remark \ref{difflimit}).

Moving to the proof, it involves several key ingredients and we provide a broad overview of them now. 
As already briefly alluded to above, we will rely on the isomorphism theorem of Lupu and analyze the corresponding loop percolation model instead. The broad steps in the analysis are inspired from \cite{kn2} and so the main novelty lies in developing new ideas and methods to handle the several long range effects in the loop percolation model.
\subsection{Idea  of proofs} \label{iop}
We start with a toy discrete loop soup model  which will suffice for the purposes of this discussion. The actual model has continuous loops which leads to technical complications, primarily topological in nature, which we will overlook for the moment. 
Consider the equivalence class $\gamma$ of a discrete loop $\Gamma=(x_0,x_1,\ldots,x_k)$ with $x_i\sim x_{i+1}$ (denoting adjacency), where two loops are equivalent if one is obtained from the other by a cyclic rotation. Let the set of all such equivalence classes be $\Omega$. We say the length of $\gamma$ denoted by $|\gamma|$ is $k.$
Consider now a Poisson process with intensity measure $\mu$ on $\Omega$, with 
\begin{equation}\label{intensity1}
\mu(\gamma)=\frac{1}{(2d)^{k}},
\end{equation}
 i.e., the probability of a simple random walk tracing out $\Gamma$ started from $x_0.$
Given a realization of this Poisson loop soup denoted by $\cL$, let {$E$ be the subset} of edges of $\Z^d$ obtained by taking the union of all the loops $\gamma$ in the loop soup and let $\cC(0)$ be the connected component of the origin {in $E$}.  The goal is to study the random walk on $\cC(0)$ and towards this we will first study the geometry of $\cC(0)$ conditioned on $\{0\lbij x\}$ as $|x|\to \infty$ (in fact the same arguments work even if we consider the conditioning on $\{0\lbij \partial \B_n\}$).

Following \cite{kn2},  letting $B(0,r)$ be the intrinsic ball of radius $r$, the following estimates will be proven:
\begin{enumerate}
\item $\E | B (0,r)|\approx r,$ (expected volume growth)
\item $\P(\partial B(0,r) \neq \emptyset)\approx \frac{1}{r},$  (chemical one-arm exponent)
\item $\P(R_{\text{eff}}(0 \bij \partial B(0,r))\le \e r )\le  \frac{c(\e)}{r},$ where $R_{\text{eff}}$ is the effective resistance and $c(\e)$ is an explicit function of $\e$ converging to zero {as $\e \rightarrow 0$}.  
\end{enumerate}
It might be instructive to point out that the above behavior matches exactly with the geometry of the critical branching process (see for instance \cite{kesten}).
Finally, these inputs will be used to prove conditional statements of the form 

\begin{itemize}
\item[(1*)] $\P(|B (0,r)| \notin [\e r^2, \frac{1}{\e}r^2] \mid 0\bij x) \le c(\e),$
\item [(2*)] $\P(R_{\text{eff}}(0 \bij \partial B(0,r))\notin [\e r, \frac{1}{\e}r]\mid 0\bij x ) \le c(\e),$ for all large enough $|x|$ depending on $r.$
\end{itemize}
Since the above random variables are functions of finitely many edges (namely the ones in the extrinsic box $\B_r$), the  estimates hold for any sub-sequential limit $\P_{\IIC},$ provided we can show that indeed the probability of the presence of any given finite subset of edges converges while taking sub-sequential limits. There is a bit of subtlety considering that on $\cable$, partial edges may become full in the limit (this will be ruled out formally in the next section). Ignoring this issue for the moment, the above estimates then allows us to appeal to general results developed in \cite{barlow} (the precise statement is recorded as Proposition \ref{barlow}).
We now indicate some of the methods developed in \cite{kn2} to prove (1)-(3) and how the long range nature of the loop model breaks the arguments calling for new ideas. For expository reasons, we will only focus on two illustrative instances.  

First, consider the sum $$\E[|B(0,r)| \mathbf{1}_{0\bij x}]=\sum_{z \in \Z^d}\P(0\overset{r}\bij z, 0\bij x),$$ where $0\overset{r}\bij z$ denotes that the chemical distance between $0$ and $z$ is at most $r.$ This appears, for instance, while establishing (1*).

To bound this, in the bond case one can consider a path of length at most $r$ connecting $0$ and $z$ as well as a path connecting $0$ and $x.$ Let $w$ be the last point where these paths intersect (see {Figure} \ref{iop1}). Then one witnesses the event $\{0\overset{r}\bij w \circ w\overset{r}\bij z \circ {w\bij x}\},$ where $\circ$ denotes that the connections occur disjointly. This allows one to apply the van den Berg-Kesten (BK) inequality which proves that the probability of disjoint occurrences is at most the product of the individual probabilities.
\begin{figure}[h]
\centering
\includegraphics[scale=.8]{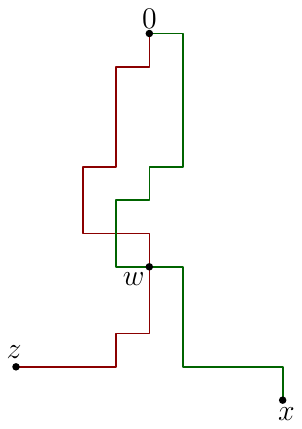}
\caption{The two paths $0\ar z$ and $0 \ar x$ last meet at $w$ leading to $0\overset{r}\bij w \circ w\overset{r}\bij z \circ {w\bij x}.$
}
\label{iop1}
\end{figure}

The above reasoning yields 
\begin{align}\label{intrinsic1}
\sum_{z \in \Z^d}\P(0\overset{r}\bij z, 0\bij x)&\le \sum_{w,z \in \Z^d} \P(0\overset{r}\bij w \circ w\overset{r}\bij z \circ {w\bij x}) \nonumber  \\
& \le \sum_{w,z \in \Z^d} \P(0\overset{r}\bij w) \P(w\overset{r}\bij z) \P({w\bij x}) \nonumber \\
&\lesssim \Big( \E | B(0,r) | \Big)^2  \P({0\bij x}).
\end{align}

If one replaces $0\overset{r}\bij z$ by the extrinsic constraint that $z\in \B(0,r),$ i.e. $|z|\le r,$ a way to adapt the above argument to the loop percolation model is by the tree expansion approach introduced first in \cite{triangle1} and adapted to the setting of loops by Werner in \cite{werner}. We review this now. The first order of business is to review the the counterpart of BKR (van den Berg-Kesten-Reimer) inequality \cite{bk1,bk2,bk3} in the setting of loops.
Towards this let us say that if $A_1, A_2, \ldots, A_m$ are connection events, then $A_1\circ A_2 \ldots \circ A_m$ denotes that the events are certified by disjoint loops (note that this does not prevent the different $A_i$s from using the same edges as long as they come from different loops). The actual BKR  inequality in this setting in fact requires the connections to be certified not by disjoint loops but by disjoint ``glued'' loops, an additional complication which we will overlook in this section. The formal treatment appears later as Lemma \ref{BKR inequality}).\\

\noindent
\textbf{Lemma} (Informal) The following BKR inequality holds. 
$$\P(A_1\circ A_2 \ldots \circ A_m) \le \prod_{i=1}^m \P(A_i).$$

\begin{figure}[h]
\centering
\includegraphics[scale=.8]{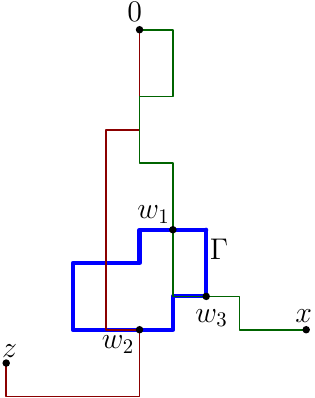}
\caption{The chain of loops $0\ar z$ and $0\ar x$ last intersect at $\Gamma$ from which emanate three disjoint arms none of which use $\Gamma$ leading to $0\bij w_1 \circ w_2 \bij z \circ w_3\bij x \circ \{\Gamma \in \cL\}$
}
\label{iop2}
\end{figure}
Given the above, let us return to bounding $\sum_{z \in \Z^d: |z|\le r}\P(0\bij z, 0\bij x).$
Note that since we have dropped the constraint on the intrinsic length of the path from $0$ to $z,$ we can take both the two paths {from  $0$ to  $x$ and from $0$ to $z,$} to be  consecutive chains of {\emph{distinct}} loops (this fact is proven formally later in the article, see \eqref{simple}; also see Figure \ref{3.11}). Note that while the individual chains are made of distinct loops, the same loop might appear in both the chains. Then, letting $\Gamma$ be the last common loop in the paths (as illustrated in Figure \ref{iop2}), notice that $\{0\bij z, 0\bij x\}$ leads to an event of the form $\{0\bij w_1\circ w_2 \bij z \circ w_3 \bij x\}$ with $w_1, w_2, w_3$ being points on $\Gamma$ with none of the connections using $\Gamma.$ This is then bounded by the following. 
{\begin{align}\label{extrinsic1}
\sum_{\Gamma}\sum_{\substack{w_1,w_2, w_3 \in \Gamma\\ z\in \B(0,r)}}&\P(0\bij w_1 \circ w_2 \bij z \circ w_3\bij x \circ \Gamma \in \cL)\nonumber \\
& \le \sum_{\Gamma} \sum_{\substack{w_1,w_2, w_3 \in \Gamma\\ z\in \B(0,r)}}\P(0\bij w_1)\P(w_2 \bij z) \P(w_3\bij x) \P(\Gamma \in \cL).
\end{align}}
This along with the two point function in \eqref{two point0} allows us to bound the RHS.  We will term arguments of this form as \emph{extrinsic tree expansion} or $\ete.$
\begin{figure}[h]
\centering
\includegraphics[scale=.8]{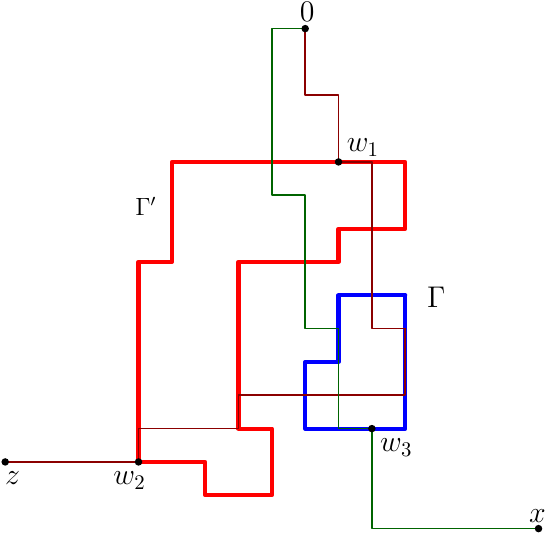}
\caption{While $\Gamma$ is the last loop on the path $0\ar x$ which intersects the path $0\ar z,$ and hence the arm $w_3 \ar x$ is disjoint from everything else, including $\Gamma,$ owing to the intrinsic constraint of the type $0\overset{r}\bij z$, the path certifying the latter might use edges from another loop $\Gamma'$ both before and after intersecting $\Gamma$ making the corresponding arms not disjoint.}
\label{iop3}
\end{figure}

However, for our analysis of the IIC, we will indeed need to deal with intrinsic constraints as in \eqref{intrinsic1} and cannot get away with extrinsic ones as above. But at this point note that the above tree expansion breaks since one cannot simply consider a chain of loops and hence the same loop might appear multiple times along the path from $0$ to $z,$ both before and after $\Gamma$, as illustrated in Figure \ref{iop3}.\\

\noindent
\textbf{Big-loop ensembles.} This leads us to introduce a key object in our arguments: a \emph{big loop ensemble} or, in short, a $\ble.$ Ignoring certain degenerate cases, generically it denotes a triple of loops and points $(\Gamma_1, \Gamma_2, \Gamma_3, u_1,u_2,u_3)$ with $u_i\in \Gamma_i$ for $i=1,2,3$, satisfying the following conditions:
     \begin{enumerate}
         \item $d^\extr({\Gamma}_i,u_i) \le 1$ for $i=1,2,3.$
         \item  $|{\Gamma}_2| +2\ge |u_2-u_3|$ and $ |{\Gamma}_3| +2\ge|u_1-u_3|.$
     \end{enumerate}
     Note that the size of $\Gamma_1$ does not feature into the definition and should be thought of as being of $O(1)$ size. So what this definition stipulates is that the sizes of the loops $\Gamma_2$ and $\Gamma_3$ are large compared to the distances $|u_2-u_3|$ and $|u_1-u_3|$ respectively. Without going into too much details, we only comment that the definition of the $\ble$ is motivated by the observation that for repeated occurrences of a loop along an intrinsic geodesic, it must be the case that the length of the portion of the geodesic between the two occurrences of a loop is smaller than the length of the loop, since otherwise shortcutting via the loop would decrease the length of the geodesic leading to a contradiction. We further use that the length of the portion of the geodesic is at least the $\ell^1$ distance between its endpoints.
     
     Relying on this, and taking the path from $0$ to $z$ to be a geodesic, one may perform a tree expansion akin to the sketch above in \eqref{extrinsic1}, where the loop $\Gamma$ will now be replaced by the $\ble$ with the condition that $\Gamma_1$ is connected to $\Gamma_2$ and $\Gamma_2$ connected to $\Gamma_3$ and these two connections while not disjoint from each other are disjoint from the remaining three connections appearing in the expression in \eqref{extrinsic1}.
     
     \begin{figure}[h]
\centering
\includegraphics[scale=1]{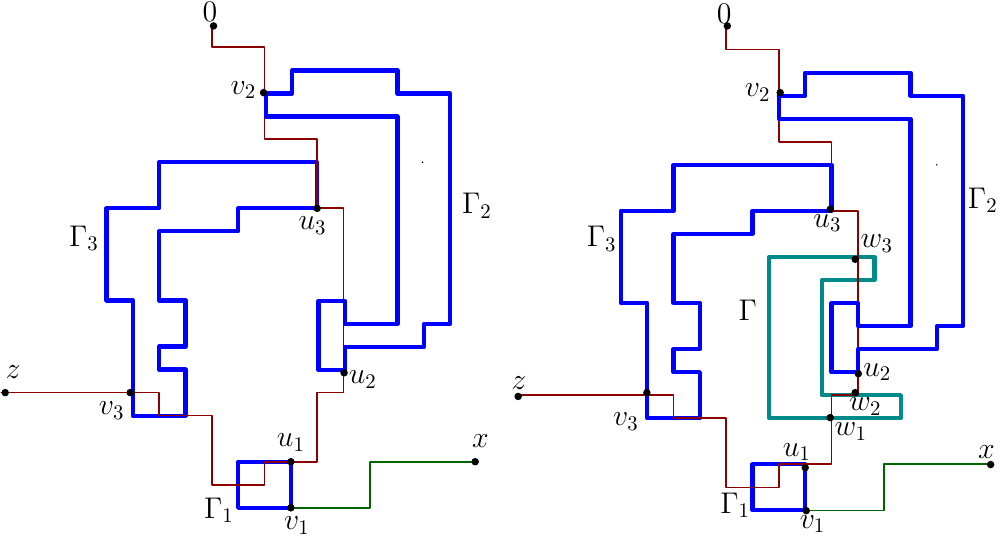}
\caption{On the left, the figure shows the $\ble$ $(\Gamma_1, \Gamma_2, \Gamma_3, u_1,u_2,u_3)$ along with the points $v_1,v_2,v_3$ from which emanate disjoint arms to $x,0,z$ respectively. On the right, the connections between the $\ble$ loops are tree expanded extrinsically around $\Gamma$ leading to the points $w_1,w_2, w_3$ wich are connected to $u_1,u_2, u_3$ disjointly without using any of $\Gamma_1,\Gamma_2,\Gamma_3$ or $\Gamma.$}
\label{iop4}
\end{figure}
     One can further tree expand the connections inside the $\ble$ $(\Gamma_1, \Gamma_2, \Gamma_3, u_1,u_2,u_3)$ to obtain an event of the following form (see Figure \ref{iop4} for an illustration of this):

\noindent     
There exist points $v_1,v_2,v_3,w_1, w_2,w_3\in \Z^d$ along with a loop $\Gamma$ such that,\\
 \noindent 
$\bullet$  $d^\extr({\Gamma}_i,v_i) \le 1$ for $i=1,2,3$,\\
\noindent
$\bullet$ 
$d^\extr(\Gamma,w_i) \le 1$ for $i=1,2,3$, \\
$\bullet$ {The following disjoint connection holds:}
 \begin{align*}
&v_1\overset{{\mathcal{L}} \setminus \{{\Gamma},{\Gamma}_1,{\Gamma}_2,{\Gamma}_3  \}}{\longleftrightarrow} x \circ v_2 \overset{{\mathcal{L}} \setminus \{{\Gamma},{\Gamma}_1,{\Gamma}_2,{\Gamma}_3  \},r}{\longleftrightarrow}  0 \circ  v_{3} \overset{{\mathcal{L}} \setminus \{{\Gamma},{\Gamma}_1,{\Gamma}_2,{\Gamma}_3  \},r}{\longleftrightarrow}  z \\
&\circ u_1 \overset{{\mathcal{L}} \setminus \{{\Gamma},{\Gamma}_1,{\Gamma}_2,{\Gamma}_3 \}}{\longleftrightarrow}  w_1 \circ u_2 \overset{{\mathcal{L}} \setminus \{{\Gamma},{\Gamma}_1,{\Gamma}_2,{\Gamma}_3 \}}{\longleftrightarrow}  w_2 \circ u_3 \overset{{\mathcal{L}} \setminus \{{\Gamma},{\Gamma}_1,{\Gamma}_2,{\Gamma}_3 \}}{\longleftrightarrow}  w_3. 
 \end{align*}     
 We will term such arguments involving $\ble$s as an \emph{intrinsic tree expansion} (\ite).
We now briefly indicate how the usage of $\ble$ already forces  $d$ to be large in our arguments. The above analysis leads us to bounding moments of the ``size'' of a \ble. A first moment computation involves an expression of the form (note the $\ble$ consideration manifests in the constraint $|\Gamma|\ge |u|$)
 \begin{align*} 
 \sum_{u\in \Z^d}\P(0\bij  u) \sum_{\Gamma \ni u:|\Gamma|\ge |u|}|\Gamma|\P(\Gamma \in \cL)\le {C\sum_{r \in \N} r^{d-1} \cdot r^{2-d}\cdot r \cdot \frac{1}{r^{d/2-1}}.}
 \end{align*}
 where the last summand is obtained by taking $r=|u|.$
The RHS is summable only when $d>8$ and not when $d>6.$ The actual implementation of arguments of this type turns out to involve control on higher moments as well as deal with multiple $\ble$s simultaneously, pushing the threshold further up.  

While $\ble$s will be ubiquitous in our arguments, another central tool which will feature prominently is an averaging method which we describe next via a bond percolation argument which breaks down in the presence of loops. 
While this issue appears throughout the proofs, let us for illustrative purposes, take the case of the analysis of intrinsic one arm exponent, i.e. $p_r:=\P(\partial B(0,r) \neq \emptyset)$. As mentioned above, a key input for us will be the estimate $p_r\lesssim \frac{1}{r}.$
In the bond case, an argument of the following style was implemented in \cite{kn2}. Using an a priori bound of {Aizenman-Barsky} \cite{triangle2} (the formal expression is recorded in Proposition \ref{cluster upper} {in the case of loop model}), one has $\P(|\cC(0)|\ge \e r^2)\lesssim \frac{1}{r}$ (where $\e>0$ is a small constant). Now if $|\cC(0)|\le \e r^2,$ by the pigeonhole principle there is some $i\in [r/2,r],$ such that the surface measure, i.e., points at distance \emph{exactly} $i$ (which we will term as a sphere) is less than $2\e r.$ One can now reveal the cluster of the origin till the first such $i$, say $\tau.$ This conditioning only has the effect of removal of some edges and on the remainder one can apply an inductive argument to obtain 
\begin{align}\label{surface1}
\P(\partial B(0,2r) \neq \emptyset \mid B(0,\tau))&\le 2\e r p_r, \\
\nonumber
\text{ which leads to a recursion of the form \,\, }
 p_{2r}&\lesssim 2 \e r p_r^2+\frac{1}{r}.
\end{align}
\begin{figure}[h]
\centering
\includegraphics[scale=.7]{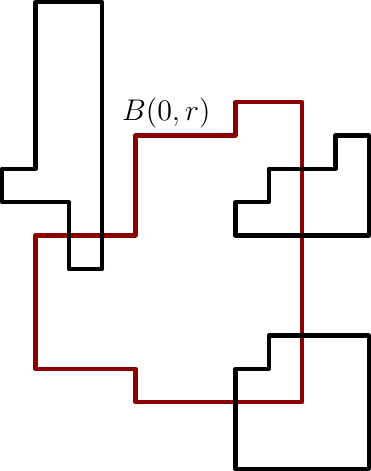}
\caption{The brown curve denotes the boundary of $B(0,r),$ and there could be large loops (denoted by black) which extends from within the interior of $B(0,r)$ to much further.
}
\label{iop5}
\end{figure}
Such a strategy leads to several complications and falls apart in the loop model. First, conditioning on all the loops intersecting the metric ball $B(0,\tau)$ reveals long range connections potentially going much beyond the ball (see Figure \ref{iop5}). Further, the total volume of the loops intersecting a given sphere, might be much larger than the number of points on the sphere, necessitating a new approach. Towards this, we simply remark that instead of working with a stopping domain, we introduce an averaging argument. 
The intuition being that, it is unlikely for a majority of the spheres to intersect large loops. Thus we consider two cases.\\
$\bullet$ If all the loops intersecting a sphere has small length, then the loop surface area is comparable to that of the sphere area and hence an argument of the type of \eqref{surface1} can be employed. \\
$\bullet$ We separately bound the probability that at least a significant fraction of the spheres intersect large loops by an {$\ite$} argument.

We end this discussion by commenting that while the above two examples illustrate some of the new strategies needed to address the long range nature of the model, many technical complications and subtleties arise in their implementation which we will not comment on here, to maintain an ease of readability. 

We now move towards the main body of the article, and end this section with a brief account of the organization of the remaining article. 

\subsection{Organization of the article}

{In Section \ref{section 2}, we introduce the Poissonian loop soup model, and state formally the coupling with the critical level set of GFF on $\cable$, and state its basic properties. Section \ref{section 3} is devoted to formally introducing the central notion of big-loop ensembles and proving various results stating that $\ble$s appear in various events of our interest. We establish an expected volume bound on intrinsic balls in Section \ref{section 4}. The intrinsic one-arm exponent is obtained  Section \ref{section 5}. We deduce  effective resistance bounds in Section \ref{section 6} relying on a generalization of the well known Nash-Williams inequality.
Finally in Section \ref{section 7}, we translate the unconditional results into ones for the conditional measures \eqref{limitiic} which pass to the limit and hence in terms of any subsequential IIC measure and then establish our main result. Some technical estimates are finally proven in the appendix, Section \ref{appendix}.}\\

While we will introduce several notations throughout the paper, we begin by defining the most basic ones.

\subsection{Notations}
For $A_1,A_2\subseteq \Z^d$, define $d^\extr(A_1,A_2):= \min \{|x_1-x_2| : x_1\in A_1,x_2\in A_2\},$ where $|\cdot|$ denotes the Euclidean $\ell^1$-norm.  For $r\in \N$, let $\B(x,r):=B^\extr(x,r) := \{y \in \Z^d:  |x-y| \le r\}$ be a box of center $x$ and radius $r$. For $A\subseteq \cable$, let $|A|$ be the number of lattice points in $A$.
 In addition, for any sequence $L$, define $\textsf{Set}(L)$ be the set of elements that appear in $L.$ Throughout the paper, $C$ is a constant which may change from line to line in proofs.

\subsection{Acknowledgement}S.G. was partially supported by NSF grants DMS-1855688, DMS-1945172,
and a Sloan Fellowship.
K.N. is supported by the National Research Foundation of Korea (NRF-2019R1A5A1028324, NRF-2019R1A6A1A10073887).

\section{Preliminaries} \label{section 2}
This section is devoted to providing the formal underpinning to carry out our analysis.  
We start with the all important notion of the loop soup and how that can be used to encode the level set of GFF.
\subsection{Loops on ${\Z}^d$ and  on $\cable$}
We start with some definitions of loops in both discrete and continuous settings, primarily importing notations from \cite[Section 2]{cd}.
\subsubsection{Discrete-time loops on $\Z^d$} For $x,y\in \Z^d,$ we say that $x\sim y$ if $x$ and $y$ are adjacent.
A discrete-time path on $\Z^d$
is a function $\eta : \{0,1,\cdots,k\} \rightarrow \Z^d $ ($k$ is a non-negative integer)  such that $\eta(i) \sim \eta(i+1)$ for any $i=0,\cdots,k-1.$ If $\eta(0) = \eta(k)$, 
then we say that  $\eta$ is a rooted discrete-time path {rooted at $\eta(0)$}.
Two rooted discrete-time loops are called 
equivalent if they equal to each other after a time-shift {i.e., $\eta \sim \eta'$ if 
there exists $0\le i \le k$ such that $\eta(j)=\eta'(i+j)$ for all $j$, where indices are considered as  modulo $k.$}
Each equivalent class is called a \emph{discrete loop} on $\Z^d.$ {Define $\mathscr{S}_0$ to be the collection of all  discrete loops on $\Z^d$}.

From now on, we write a  discrete loop as ${\Gamma} = (x_0,x_1,\cdots,x_k)$, where $x_i$s denote the consecutive 
lattice points on ${\Gamma}$, and define its length to be $|{\Gamma}| := k$.
Its multiplicity $J=J({\Gamma}  )$ is defined to be the
maximum integer such that subsequences $
(x_{(j-1)kJ^{-1}},x_{(j-1)kJ^{-1}+1}, \cdots, x_{jkJ^{-1}})$  are identical for $ j=1,2\cdots, J$, i.e., {the period of $\Gamma.$}
\subsubsection{Continuous-time loops on $\Z^d$}
A continuous-time path on $\Z^d$
is a function $\eta : [0,T) \rightarrow \Z^d $  
for which there exist a discrete-time path $\eta' : \{0,1,\cdots,k\} \rightarrow \Z^d$  and $0 = t_0<t_1<\cdots<t_k = T$ such that
 for all $i=0,1,\cdots,k-1$,
\begin{align*}
    \eta(t) = \eta'(i) ,\quad t_i\le t<t_{i+1}.
\end{align*}
Its length   is defined to be  $|\eta| := k$.   For each $i=0,1,\cdots,|\eta|$, 
$\eta^{(i)}:= \eta'(i) $ is called {the $i$-th lattice point of $\eta$} and 
  $t_{i+1}-t_i$ is the $i$-th holding time. We term $\eta$ as a rooted  continuous-time loop if the $0$-th position and the $|\eta|$-th 
    lattice point are the same, {in which case the root is defined to be $\eta(0).$}

The intensity measure of the loops will be dictated by the heat kernel of the random walk on $\Z^d.$
Let $\{S^x_t\}_{t\ge 0}$ be a continuous-time simple random walk starting from $x$. 
Define $p_t(x, y) := \P (S^x_t = y)$ to be the continuous-time
 heat kernel i.e., of the random walk which walks on $\Z^d$ with the holding 
   time at any vertex being independent i.i.d.
  standard Exponentials.  Let $\P^t_{x,x}(\cdot)$ be the conditional distribution
of $\{S^x_{t'}\}_{0\le t'\le t}$   given $S_t=x$.
Then we define a loop measure $\mu$  on the space of rooted {continuous-time} loops on $\Z^d$:
\begin{align}  \label{mumu}
    \mu(\cdot) = \sum_{x\in \Z^d} \int_0^\infty t^{-1}\P_{x,x}^t(\cdot) p_t(x,x)dt.
\end{align} 
Similarly as before, we say that two rooted {continuous-time} loops on $\Z^d$ are equivalent if they 
  equal to each other after a time-shift.
Each equivalent class of such rooted loops is called a {continuous-time}  loop on $\Z^d.$ 
{As $\mu$  is invariant under the time-shift, it induces
a measure on the space of {continuous-time} loops on $\Z^d$,
 which is also denoted by $\mu$, i.e. the measure of a given equivalence class is the measure of any given representative}. 

Then, {for any discrete loop  ${\Gamma} = (x_0,x_1,\cdots,x_k)$, 
\begin{align} \label{basic}
    \mu &( \{  \text{continuous-time loop $\gamma$ on $\Z^d$}: \exists  \text{ rooted continuous-time loop $\eta$ on $\Z^d$} \text{ such that } \nonumber \\
    &\eta \in \gamma \text{ and } \eta^{(i)} = x_i \text{ for $i=0,1,\cdots,k$}\}) = (2d)^{-k} / {J((x_0,\cdots,x_k))}.
\end{align}}
{Here, the term $(2d)^{-k}$ comes from the probability of the corresponding trajectory of a simple random walk and the term
 $J^{-1}$ is related to the factor $t^{-1}$ in \eqref{mumu}. 
 We refer to \cite[Section 2.6]{cd}  and \cite[Sections 2.1 and 2.3]{le1}   for details.}

{For $\alpha>0$, define the  loop soup  $ \cL_\alpha$ to be the Poisson point process in the
space of continuous-time loops on $ \Z^d$  with  the intensity measure $\alpha  {\mu}$.}

\subsubsection{Continuous-time loops on $\cable$} \label{section 2.1.3}
{A rooted continuous-time loop on $\cable$  
is a path $\tilde{\rho} : [0,T] \rightarrow \cable$
such that $\tilde{\rho} (0) = \tilde{\rho}(T)$. Similarly as before,
a continuous-time  loop  on  $\cable$ 
is an equivalent class of rooted continuous-time loops on   $\cable$ such that one can be
transformed into another by a time-shift.} Throughout the paper, we denote by $\tiloop$ a continuous-time loop on $\cable.$ \\

\textit{In fact, we adopt the convention that all continuous objects will be denoted using a tilde notation whereas their discrete counterparts will be denoted without a tilde. }\\

Following \cite[Section 2.6.2]{cd}, any continuous-time loop on $\cable$ will be of one the following three types:
\begin{enumerate}
    \item fundamental loop: a loop which visits 
      at least two lattice points;
    \item  point loop: a loop which visits
       exactly one lattice point;
    \item edge loop: a loop contained in a single interval and does not visit any lattice points.
\end{enumerate}
For any continuous-time loop $\tiloop$ on $\cable$ which is either of  fundamental or point type,
define the corresponding discrete loop as
\begin{align} \label{trace}
   \textsf{Trace}(\tiloop):= (x_0,x_1,\cdots,x_k),
\end{align}
where $x_i$s denote consecutive adjacent lattice points that  $\tiloop$ passes through.
Note that the map $\textsf{Trace}$ is well-defined, i.e. for any two rooted continuous-time loops on $\cable$ in the same equivalence class,  
their outputs also belong to the same equivalence class of discrete loops on $\Z^d$. 
From now on, we denote by 
 $|\tiloop|$ the length of  the corresponding discrete loop $\textsf{Trace}(\tiloop)$ of $\tiloop$.
 
While we will not entirely describe the structure of the continuous loops, the reader should think of them as being obtained from the continuous time discrete loops by adding Brownian excursions which don't hit the neighboring lattice points  and hence does not add any full edge of $\Z^{d}$ but only partial edges. Before providing a bit more elaboration let us review what Brownian motion on $\cable$ is.
Referring the reader to \cite{lupu} for a formal treatment, we provide a sketch to help form a visual picture for the reader. Brownian motion on $\cable$ behaves as a standard Brownian motion in the interior of edges in $E(\Z^d)$ and when  it hits a lattice point $x$, it behaves as a Brownian excursion from $x$ in an uniformly chosen edge incident on $x$ (a detailed description appears in \cite{lupu}) (another good way to think of it is in terms of the limit of the usual random walk when the edges are subdivided into many vertices).
 Then by the framework of \cite{markovian}, there is an associated measure $ \tilde{\mu}$ on the space of continuous-time loops  on $\cable$.

\subsection{Loop soup} \label{section 2.2} For $\alpha>0$, define the  loop soup  $ \widetilde \cL_\alpha$ to be the Poisson point process in the
space of continuous-time loops on $ \cable$  with  the intensity measure $\alpha  \tilde{\mu}$. 
{Let $\widetilde \cL^\text{f}_\alpha$, $\widetilde \cL^\text{p}_\alpha$ and $\widetilde \cL^\text{e}_\alpha$ be the point processes consisting of fundamental loops, 
point loops and edge loops  in $\widetilde \cL_\alpha$ respectively. 
  By the thinning   
  property of the Poisson point process, these point processes are independent.  

Throughout the paper, we focus on the case $\alpha=1/2$.  Referring the reader to
 \cite[Section 2]{lupu} and \cite[Section 2.6]{cd} for more details, we briefly explain how one obtains the loop
  soup $ \widetilde \cL_{1/2}$ from $ \cL_{1/2}$.
  For any $\gamma \in  \cL^\text{f}_{1/2}$ (i.e.  a continuous-time loop on $\Z^d$ of fundamental type),
the range $\textsf{Range}(\gamma)$ of the corresponding loop in  $ \widetilde \cL^\text{f}_{1/2}$  
 is defined as follows: Taking any $\eta$ in the equivalence class $\gamma$, $\textsf{Range}(\gamma)$
  is the union of edges 
    traversed by $\eta$ and additional 
     Brownian excursions at each 
   $\eta^{(i)}$, where these excursions are 
      conditioned on returning 
        to $\eta^{(i)}$ before visiting
    its neighbors and the total local 
       time at   $\eta^{(i)}$ is the $i$-th holding time of $\eta$. 
{Thus, in particular, $\textsf{Range}(\gamma)$ is contained in the one-neighborhood of $\gamma$ (i.e. the union of edges incident on some vertex in $\gamma$).}
A similar construction works for the loops in  $ \widetilde   \cL_{1/2}^\text{p}$. 
Finally, for an edge $e$, the union of the ranges of the loops  in $ \widetilde \cL_{1/2}^\text{e}$, 
whose range is contained in  $I_e$, 
has the same law as the union of non-zero points of a standard Brownian bridge on $I_e$.} \\

The measure $\tilde \mu$ projected to the discrete underlying loop yields the measure $\mu.$\\

We now state the key isomorphism theorem that allows us to pass from the GFF percolation to loop percolation.
\subsubsection{Isomorphism theorem}
Lupu \cite[Proposition 2.1]{lupu} established a coupling between the GFF $\widetilde \Phi$ on the cable graph $\cable$ and the loop soup  $ \widetilde \cL_{1/2}$.

\begin{proposition}[Proposition 2.1 in \cite{lupu}] \label{lupu}
{    There is a coupling between the loop soup $\widetilde \cL_{1/2}$
 and $\widetilde \Phi$ such that the clusters composed of loops in $\widetilde \cL_{1/2}$ are the same as the sign clusters of GFF $\widetilde \Phi$.}
 \end{proposition}
{Here, a sign cluster denotes a maximal connected subgraph on which $\tilde{\phi}$ has the same sign. Further any $v\in \cable$ such that $\tilde{\phi}_v=0$  does not belong to any sign cluster.}

{Note that while we are considering GFF percolation induced by  the level set $\tilde{E}^{\ge 0}$, the above proposition 
provides only a description for the level set $\tilde{E}^{> 0}$. 
However, the difference between the connected components of the origin in $\tilde{E}^{\ge 0}$ and $\tilde{E}^{> 0}$ 
respectively is almost surely only finitely many points. This is because $\P(\tilde{\phi}_v = 0  ) = 0$ for all $v\in \Z^d,$ and for 
any  edge $e$, conditioned on the values of $\tilde{\phi}$ at two  endpoints of $e$ to be non-negative, $\{\tilde{\phi}_v\}_{v\in I_e}$ (i.e. Brownian bridge) 
does not have an extreme value 0 a.s. and hence if the edge is
not contained in $\tilde{E}^{> 0}$ it will not be contained in $\tilde{E}^{\ge 0}$ as well.}
 Since boundary of the connected component of the origin in $\tilde{E}^{\ge 0}$ is precisely the set of points in the component which are missing
  in the connected component in $\tilde{E}^{> 0},$ the difference is simply the boundaries (which is a single point) of finitely many partially covered edges.

~

To ease the notation, throughout the paper we will drop the parameter $1/2$ in the loop soup $\widetilde \cL_{1/2}$, and just write $\widetilde \cL$. 
{By Proposition \ref{lupu} along with the above observation and the
 symmetry of GFF (under the sign flip), in order to  prove Theorem \ref{main}, it suffices to 
prove the counterpart result for the loop soup $\widetilde \cL$ and this is the result we record next.}
 Let  $\widetilde \P^{\textup{Loop}}_{\textup{IIC}}$ be the conditional distribution of the connected component of the origin in the loop soup $\widetilde \cL,$ again denoted by $\widetilde \cC(0)$ to avoid introducing new notation, given that $0$ is connected to the infinity. More precisely, as in \eqref{limitiic}, it is any sub-sequential limit of 
    \begin{align*}
 \P(\cdot \subset \widetilde\cC(0)\mid 0 \lbij \partial \B_n) \textup{ or}\,\,   \P(\cdot \subset \widetilde\cC(0)\mid 0 \lbij x)
           \end{align*}
           as $n\rightarrow \infty$ or $|x| \rightarrow \infty$, where the underlying metric on closed subsets of $\widetilde \Z^d$ is taken to be the Hausdorff distance (we elaborate more on the measure theoretic aspects shortly). 
 Since $\widetilde \cC(0)$ is a closed connected subset of $\widetilde \Z^d,$ for every $z\in \Z^d \cap \widetilde \cC(0)$,  and an edge of $\Z^d$ incident on $z$, say $[z,w]$ for some $w \in \Z^d$ with $w\sim z$, the intersection $[z,w]\cap \widetilde \cC(0)$ is either $[z,w]$ or a union of two segments $[z,z_1]$ and $[w_1,w]$ with $z_1\neq w_1.$ We will term the edge $[z,w]$ as a fully covered or a partially covered edge in the two cases respectively. Let $\cC(0)\subset \widetilde \cC(0)$ be the union of all fully covered edges. Let the corresponding measure {of $\widetilde \P^{\textup{Loop}}_{\textup{IIC}}$}  be $\P^{\textup{Loop}}_{\textup{IIC}}.$
 
{\begin{theorem} \label{main2}
Suppose that $d>20$ and let $\{X_i\}_{i\ge 1}$ be the discrete 
simple random walk on $\cG\sim \P^{\textup{Loop}}_{\textup{IIC}}.$ Then $\P^{\textup{Loop}}_{\textup{IIC}}$-a.s.,
    \begin{align*}
        \lim_{n\rightarrow \infty }  \frac{\log p^{\cG}_{2n}(0,0)}{\log n} = - \frac{2}{3},\quad   \lim_{r\rightarrow \infty }  \frac{\log \E \tau_r}{\log r} = 3,\quad \lim_{n\rightarrow \infty }  \frac{\log |\{X_1,X_2,\ldots, X_n\}|}{\log n} =\frac{2}{3}  \ \textup{a.s.},
    \end{align*}
where  $\tau_r$ is as defined in the statement of Theorem \ref{main}. 
 \end{theorem}}

Proving the above will be the goal of the remainder of the paper. However we first get some measure theoretic details out of the way. From now on we say that a sequence of random subsets $\{K_n\}_{n\ge 1}$ of $\cable$ \textsf{cont}-converges if it weakly converges with respect to the Hausdorff distance.
While \textsf{cont}-converges involves partial edges, 
for our arguments, it will be convenient to consider the notion of convergence restricted to the fully covered edges.  For $K\subseteq \cable$, define  $K^{\text{dis}}$ to be the subset of $\cable$, obtained by retaining only the fully covered edges in $K$. Thus $K^{\text{dis}}$ can be regarded as an element in the configuration space $\{0,1\}^{E(\Z^d)}$. We say that a sequence of random subsets $\{K_n\}_{n\ge 1}$ of $\cable$ \textsf{dis}-converges if $\{K_n^{\text{dis}}\}_{n\ge 1}$ weakly converges with respect to finite dimensional distributions.

While \textsf{cont}-convergence of $\{K_i\}_{i\ge 1}$ to $K,$ does not in general imply \textsf{dis}-convergence of $\{K^{\text{dis}}_i\}_{i\ge 1}$ to $K^{\text{dis}}$, in the next result we show that in our model it indeed does. This in particular implies that, taking the weak sub-sequential limit of $\widetilde \cC(0)$ as a subset of $\cable$, and then restricting to full edges yields the same graph as taking the full edges of $\widetilde \cC(0)$ and then taking a weak sub-sequential limit as a subset of edges of $\Z^d.$

\begin{lemma}\label{disconv}
Let $\{x_n\}_{n\ge 1}$ be a sequence of lattice points such that  $|x_n| \rightarrow \infty$.
Suppose that $\{\bwt{\mathcal{C}}(0)\}_{n\ge 1}$, conditioned on $0\ar x_n$,  \textup{\textsf{conti}}-converges to  $\wt\P^{\textup{Loop}}_{\textup{IIC}}$ as  $n \rightarrow \infty$. Then $\{\mathcal{C}(0)\}_{n\ge 1}:=\{\bwt{\mathcal{C}}(0)^{\textup{dis}}\}_{n\ge 1}$ \textup{\textsf{dis}}-converges to  $\P^{\textup{Loop}}_{\textup{IIC}}$ as well.
 \end{lemma}
{It will be apparent that the same conclusion continues to hold even in the case the conditioning event is $\{ 0 \ar \partial \B_n\}$.}
 
\begin{proof}
For an edge $e=(u_1,u_2)$ in $E(\Z^d)$, let  $\bwt{\mathcal{C}}(0)[e]$ be the projection of $\mathcal{C}(0)$ onto the edge $e$ thought of as a line segment denoted by $[e]$, i.e. $\bwt{\mathcal{C}}(0)\cap [e]$ which as already mentioned is either $[e]$ itself or a disjoint union of two intervals $[e_1]$ and $[e_2]$. We claim that for any $\delta>0$, there exists $\e>0$ such that for any large enough $n\in \N$,
\begin{align}\label{unifcond}  
\P(1-\e<\text{Leb}(\bwt\cC(0)[e]) <1 \mid 0\ar x_n) < \delta,
\end{align}
where $\text{Leb}$ denotes the Lebesgue measure. Observe that under the event $\text{Leb}(\mathcal{C}(0)[e]) <1$, the connection $0\ar x_n$ cannot pass through  the edge $e$.
Recall from Section \ref{section 2.1.3}, that given $e,$  the set of loops can be divided in two categories, the edge loops that are subsets of the edge $e$, whose union was termed $\overline{\gamma}_e$ {(this object will be termed as a ``glued loop'' which we  systematically  define below)} and everything else, say, $\bwt\cL_{\bar e}.$ {Recall that $\overline{\gamma}_e$ is distributed as the non-zero set of a standard Brownian bridge on $I_e$.}   By the thinning property of the Poisson process, $\overline{\gamma}_{e}$ and $\bwt\cL_{\bar e}$ are independent. Hence instead of proving \eqref{unifcond}, we will in fact prove  
$$ \P(1-\e<\text{Leb}(\bwt\cC(0)[e]) <1 \mid  \bwt\cL_{\bar e},  \   0\ar x_n) < \delta,$$ which on averaging over $\bwt\cL_{\bar e}$ yields \eqref{unifcond}. Now note that given $\bwt\cL_{\bar e}$, if the edge $e$ is pivotal for the event $\{0\ar x_n\}$, i.e. the absence or presence of the edge $e$ dictates the occurrence of $\{0\ar x_n\}$, then conditional on {$\bwt\cL_{\bar e}$ and the event $ \{0 \ar x_n\}$}, with probability one, the edge $e$ occurs, i.e. $\bwt \cC(0)[e]=[e].$   Thus it suffices to consider the case where $e$ is not pivotal and hence  the distribution of $\overline{\gamma}_e$ conditioned on {$\bwt\cL_{\bar e}$ as well as the event $ \{0 \ar x_n\}$} is the same as its unconditional distribution.
However note that $\{0 \ar x_n \text{ in } \bwt\cL_{\bar e}\}$ is an increasing conditioning on $\bwt\cL_{\bar e}$ and denoting $\widehat \cL_{\bar e}$ to be the conditional distribution, by the FKG inequality for Poisson processes, there exists a coupling such that  $\bwt \cL_{\bar e}\subset \widehat \cL_{\bar e}$. 
Now if $\widehat \cL_{\bar e} \cap [e]=[e]$ then there is nothing to prove. Let us assume that $\widehat \cL_{\bar e} \cap [e]$ is 
the union of two disjoint intervals $[e_1]=[u_1,v_1]$ and $[e_2]=[u_2,v _2]$ adjacent to $u_1$ and $u_2$ respectively. 
By FKG inequality,  for any $\delta_1>0$ there exists $\e_1>0$ such that 
\begin{align} \label{brownian}
\P(\min \{ |u_1-v_1|,|u_2-v_2| \} \le \e_1)\le \delta_1.
\end{align}
This follows by observing that the above is true in the unconditional distribution $\bwt \cL_{\bar e}$ by considering the point loops $\overline{\gamma}_{u_1}$ and $\overline{\gamma}_{u_2}$.
\begin{figure}[h]
\centering
\includegraphics[scale=.9]{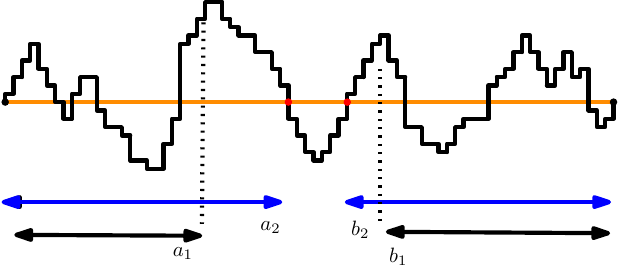}
\caption{The figure illustrates the quantities $a_1,b_1,a_2,b_2.$ If $a_1$ and $b_1$ are not very small, then it is unlikely that the gap between the closest zeros to the points $a_1$ and $1-b_1$ is small.  
}
\label{2.1}
\end{figure}
The problem now reduces to the following Brownian statement. Consider a random vector $(a_1, b_1)$ which has the distribution $(|u_1-v_1|,|u_2-v_2|)$  and  a standard Brownian bridge $B_t$ on $[0,1]$ starting and ending at $0$ independent of $(a_1,b_1)$. {See Figure \ref{2.1} for an illustration.} 
Let 
\begin{align} 
a_2=\inf \{s>a_1: B_s  = 0\} \ \text{ and }  \ 1-b_2=\sup \{ s <1-b_1: B_s = 0 \}.
\end{align}
 Then for any $\delta_2>0,$ there exists $\e_2>0$ such that {$$\P(0<1-(a_2+b_2)\le \e_2) \le \delta_2.$$}
 This follows from straightforward Brownian computations {together with \eqref{brownian} (which uses a lower bound on $a_1$ and $b_1$),}  which finishes the proof. 
\end{proof}

Next, we state the two-point estimate for the loop soup $\widetilde \cL$.
{With the aid of Proposition \ref{lupu}} along with the well known estimates for the Green's function, Lupu \cite[Proposition 5.2]{lupu} established the following sharp two-point bound. For $x,y\in \Z^d$, let $x\ar y$ be the event that $x$ and $y$ are connected via loops in $\widetilde \cL$.
\begin{proposition}[\cite{lupu}]
    There exist $C_1,C_2>0$ such that for any $x,y\in \Z^d$,
\begin{align} \label{two point}
   C_1 |x-y|^{2-d} \le  \P(x  \ar  y) \le C_2 |x-y|^{2-d}.
\end{align}
\end{proposition}

{To develop combinatorial arguments, we will also find it particularly convenient to consider the following projected loop soup consisting of discrete loops. Recalling the definition \eqref{trace}, define
  \begin{align} \label{dl}
     {\mathcal{L}}:
     = \set{\textsf{Trace}(\tiloop): \tiloop \in \widetilde \cL \text{ and is of fundamental or point type}}  .
 \end{align}}
Then for any discrete loop ${\Gamma}$,
{\begin{align} \label{intensity1}
    \P({\Gamma} \in {\mathcal{L}}) &= \P(\exists \tiloop \in \widetilde \cL: \textsf{Trace} (\tiloop) = {\Gamma} ) \nonumber  \\
    &\le  \frac{1}{2}\tilde{\mu}( \{ \tiloop \text{ loop on  $ \cable$}:\textsf{Trace} (\tiloop ) =  {\Gamma} \} )   \nonumber \\
    &= \frac{1}{2}J({\Gamma}  )^{-1}(2d)^{-|{\Gamma} |} \le  (2d)^{-|{\Gamma} |} ,
\end{align}}
where the first inequality follows from the fact that $\bwt{\mathcal{L}}$ is a Poisson point process with the intensity measure $\frac{1}{2}\tilde{\mu}$ and {the second identity follows from \eqref{basic} and the fact that $\mu$ is a push-forward measure of $\tilde{\mu}$ under the projection map.} 

~

{
Next, we state the FKG inequality for Poisson point processes \cite[Lemma 2.1]{fkg2} which immediately implies the FKG inequality for $\widetilde \cL$.
\begin{lemma}[\cite{fkg2}]
Let $f$ and $g$ be two bounded increasing (or decreasing) measurable functions of a Poisson process $\Xi$. Then
\begin{align}  
\E [f(\Xi)g(\Xi)] \ge \E [f(\Xi)]\E [g(\Xi)].
\end{align}
\end{lemma}
}

The next order of business is to introduce the formal statement  of the  van den Berg-Kesten-Reimer (BKR) inequality, which as already evident from the discussion in Section \ref{iop} will play a central role in our arguments. This inequality as a conjecture goes back to 
den Berg and Kesten \cite{bk1} and was subsequently proved by van den Berg and Fiebig \cite{bk2}
and Reimer \cite{bk3}. The interested reader is referred to the exposition by Borgs, Chayes and Randall \cite{bk4}. 

 As briefly alluded to there, the building blocks for the statement will not be loops but rather objects called \emph{glued loops} which we proceed to introducing next.
We will primarily follow the treatment in \cite[Section 3.3]{cd}.
For any {connected set $A \subseteq \Z^d$} {(i.e. connected as a subset of vertices)} with $|A| \ge 2,$ 
 let $\overline{\gamma}_A$ be the union of the range {(i.e. a subset of $\cable$)}  of loops in $\widetilde \cL^\text{f} $  
that visit every point in
$A$ and do not visit any 
 other lattice points. Next, for $x\in \Z^d,$
let  
$\overline{\gamma}_x$ be the union of 
 ranges of loops in $\widetilde \cL^\text{p} $ passing through $x$.
 Finally, for an edge $e\in E(\Z^d)$, let   $\overline{\gamma}_{e}$ be 
the union of ranges 
 of loops in $\widetilde \cL^\text{e} $ whose range is a subset of $I_e$.
 We call every element in $\overline{\gamma}_A$
  (for all connected $A \subseteq \Z^d$ with $|A| \ge 2$), $
 \overline{\gamma}_x$ (for all $x\in \Z^d$) 
 and  $\overline{\gamma}_{e}$ (for all edges $e$)  as a
\emph{glued loop}. 
{Note that glued loops are themselves not continuous-time loops on  $\cable$ but rather a superposition
 of a bunch of them and hence  a random subset of $\cable$. Throughout the paper, we use the notation $\overline{\Gamma}$ for a glued loop, and 
 denote by  $\overline{\mathcal{L}}$ the glued loop soup induced by $\widetilde{\mathcal{L}}$.} 
  {Also any collection of glued loops, corresponding to distinct index sets, behave independently, by the thinning property of Poisson point processes.}

~
 
{We end this section by introducing some language to relate glued loops and discrete loops. 
For a discrete loop $\Gamma$ and a glued loop $\overline{\Gamma}$, we define its vertex projection onto  the lattice in the following way.
\begin{align}
    \trace({\Gamma}):=
    \{v\in \Z^d: v \text{ is contained in } {\Gamma}\},\\
    \trace(\overline{\Gamma}):=
    \{v\in \Z^d: v \text{ is contained in } \overline{\Gamma}\}.
\end{align} 
For a discrete loop $\Gamma$ in $\mathcal{L},$ define the corresponding glued loop
 \begin{align} \label{trace}
 \overline{\Gamma}:=\overline{\gamma}_{\trace ( \Gamma)} \in \overline{\mathcal{L}}.
 \end{align}
Note that $ \overline{\Gamma}$ is either of  fundamental or point type.
{Conversely for a glued loop $\overline{\Gamma}$ of fundamental or point type in $\overline{\mathcal{L}}$,  we denote by 
{$\textsf{Dis}(\overline{\Gamma})$} the collection of discrete loops $\Gamma \in \mathcal{L}$  such that $\trace(\Gamma) = \trace(\overline{\Gamma}).$}\\

With the above preparation we are now in a position to formally state the BKR inequality.
\subsubsection{BKR inequality}
Following the terminology in \cite[Section 3.3]{cd}, a collection of glued loops is said to \emph{certify}
an event $\mathcal{A}$ if on the realization of this 
  collection of glued loops, $\mathcal{A}$ 
    occurs regardless 
    of the realization 
      of all other glued loops.
For two events $\mathcal{A}$ and $\mathcal{B}$, define  $\mathcal{A}\circ  \mathcal{B}$ to be the event that there are two disjoint
collections of glued loops such that one collection certifies $\mathcal{A}$ and the other collection certifies
$\mathcal{B}$. {Note that disjoint collections of glued loops simply mean that the collections do not share a common glued loop, 
but a glued loop in one collection can intersect (as a subset of $\cable$)
a glued loop in the other collection.} 
A first statement only considers events depending on finitely many glued loops. 
\begin{lemma}[BKR inequality] \label{BKR inequality}
    For any two events $\mathcal{A}$ and $\mathcal{B}$ depending on finitely many glued loops, 
\begin{align}
    \P(\mathcal{A}\circ \mathcal{B}) \le \P(\mathcal{A}) \P(\mathcal{B}).
\end{align}
\end{lemma}
One can pass to the limit to remove the finitary condition and this is the statement we will be relying on. 
We quote this from \cite{cd}. However unlike there, where the statement is made restricted to 
``connecting events", we will also need to include events involving certain chemical distance constraints. 
We start with some general notation we will be using to denote  these events.

Let $A_1,A_2 \subseteq {\Z}^d.$
For a collection $\overline{\mathscr{S}}$ of glued loops  on $\cable$, define 
\begin{align}\label{abb12}
    A_1 \overset{\overline{\mathscr{S}} }{\longleftrightarrow} A_2
\end{align}
to be the event that there exist points $x \in A_1$ and $y\in A_2$ which can be connected using only glued  loops in $\overline{\mathscr{S}}$. Let's call an event of the above kind, a connecting event.
Also for $r>0$, define
\begin{align} \label{abb34}
    A_1 \overset{\overline{\mathscr{S}},r }{\longleftrightarrow} A_2
\end{align}
to be the event that there exists a lattice path $\ell$ from some point $x\in A_1$ to some point $y\in  A_2$, only using glued loops in $\overline{\mathscr{S}}$, 
of length {(i.e. the number of edges in $\ell$)} at most $r$. We term these events as connecting events with chemical constraints.
The same proof as \cite[Corollary 3.4]{cd} implies,
\begin{lemma}\label{BKRCD}
    For any sequence of events  $\mathcal{A}_1,\cA_2,\ldots, \cA_k$ where every event is of one of the above two types, then
\begin{align}
    \P(\mathcal{A}_1\circ \mathcal{A}_2\circ \ldots \mathcal{A}_k) \le \prod_{i=1}^k\P(\mathcal{A}_i).
\end{align}
\end{lemma}
Also for $A_1:=\{x\}$ and $A_2 := \{y\}$ with $x,y\in \Z^d$, we simply write
\begin{align*}
    x  \overset{\overline{\mathscr{S}} }{\longleftrightarrow}  
  y:= \{x\}  \overset{\overline{\mathscr{S}}}{\longleftrightarrow}  \{y\},\qquad   x  \overset{\overline{\mathscr{S}},r }{\longleftrightarrow}  
  y:= \{x\}  \overset{\overline{\mathscr{S}},r }{\longleftrightarrow}  \{y\}.
\end{align*}
{In addition, we abbreviate
\begin{align*}
    A_1\ar A_2:= A_1 \overset{\widetilde \cL}{\longleftrightarrow} A_2:= A_1 \overset{\overline{\mathcal{L}}}{\longleftrightarrow} A_2,\qquad   A_1 \arr A_2:= A_1 \overset{\widetilde \cL,r }{\longleftrightarrow} A_2:=  A_1 \overset{\overline{\mathcal{L}},r }{\longleftrightarrow} A_2.
\end{align*}}

\subsubsection{Cluster size}
The next result we record provides a crucial bound on the cluster size. 
This bound for bond percolation was proved in the seminal work \cite{triangle2}. Here we record the version for $\bwt \cL$ proved in \cite{cd}.
\begin{proposition}[Proposition 6.6 in \cite{cd}]  \label{cluster upper}
There exists $C>0$ such that for any $M \ge 1,$
    \begin{align*}
        \P( |C(0)| \ge M ) \le \frac{C}{M^{1/2}}.
    \end{align*}
  {Recall that the notation $|C(0)|$ denotes the number of lattice points in $C(0)$. }
\end{proposition}

We now move on to a series of discrete loop estimates that will feature quite heavily in our arguments. 
\subsection{Discrete loop estimates}
Recall that a discrete loop denotes the equivalence class of rooted discrete-time loops on $\Z^d$ equivalent to each other via a cyclic rotation. Although many estimates in this section hold for arbitrary dimension $d$, we assume that $d>6$. We denote by $p_\cdot (\cdot,\cdot)$ the heat kernel of the discrete-time simple random walk on $\Z^d$.

{Recall that $\mathscr{S}_0$ denotes the collection of all  discrete loops on $\Z^d$}. Throughout the paper, for brevity, a summation over discrete loops satisfying a certain condition $\mathcal{A}$ will be shorthanded as
\begin{align*}
    \sum_{{\Gamma} \text{ satisfying $\mathcal{A}$}} :=\sum_{\substack{{\Gamma}\in \mathscr{S}_0 \\ {\Gamma}   \text{ satisfying $\mathcal{A}$}}},
\end{align*}
The summation sign will only be used for discrete loops and hence there is no scope for confusion.
Also, the summation over all discrete loops on $\Z^d$ is just written as
{\begin{align*}
    \sum_{{\Gamma}} := \sum_{{\Gamma} \in \mathscr{S}_0 }.
\end{align*}}

\subsubsection{One-point estimates}
The following is a probability bound for discrete loops passing through a given point.

\begin{lemma}\label{second moment}
There exists $C>0$ such that for any $L \ge 1$ and $i \ge 0$ such that $i+1<d/2$,  
    \begin{align} \label{21200}
      \sum_{ \substack{ {\Gamma}\ni 0 \\ |{\Gamma}| 
 \ge L}} |{\Gamma}|^i  \P({\Gamma} \in {\mathcal{L}}) \le CL^{i+1-d/2}.
    \end{align}
    In particular,   for $i=0,1,2,$
    \begin{align} \label{211}
    \sum_{{\Gamma}\ni 0}|{\Gamma}|^i   \P({\Gamma} \in {\mathcal{L}})\le  C.
    \end{align}
\end{lemma}      
\begin{proof}
Note that for any $\ell \in \N,$
\begin{align*}
 (2d)^{-\ell} \cdot |\{\text{discrete loop } \Gamma \text{ on } \Z^d :  {\Gamma}\ni 0 , |{\Gamma} | = \ell\}| \le   p_\ell (0,0) 
  \le C\ell^{-d/2} .
\end{align*}
Thus \eqref{21200} is bounded by
    \begin{align*}
        \sum_{ \substack{ {\Gamma}\ni 0 \\ |{\Gamma}|  \ge L}} |{\Gamma}|^i  \P({\Gamma} \in {\mathcal{L}}) \overset{ \eqref{intensity1}}{\le}  \sum_{\ell=L}^\infty  \ell^i  \sum_{ \substack{ {\Gamma}\ni 0 \\ |{\Gamma}|  = \ell}}  (2d)^{-\ell}   \le C \sum_{\ell=L}^\infty  \ell^i   \cdot \ell^{-d/2} \le CL^{i+1-d/2},
    \end{align*}
where we used the condition  $i+1<d/2$ in the above inequality. 
\end{proof}

\subsubsection{Two-point and three-point estimates}
We now move on to probability bounds for the existence of loops passing through two or three given points.
 \begin{lemma} \label{key0}
There exists $C>0$ such that    for any $u,v,w\in \Z^d$, 
 \begin{align} \label{25200}
     \sum_{\substack{ {\Gamma} \ni  u,v,w\\ |{\Gamma}| \ge 1}}  \P({\Gamma}  \in {\mathcal{L}}) \le  C  |u-v|^{2-d}|v-w|^{2-d}|w-u|^{2-d}.
 \end{align}
 
\end{lemma}

\begin{proof}

Let $\{X_t\}_{t=0,1,\cdots}$ be a (discrete-time) simple random walk on $\Z^d$. 
Then for any $\ell,$
$$\sum_{\substack{ {\Gamma} \ni  u,v,w\\ |{\Gamma}| =\ell}} \P({\Gamma}  \in {\mathcal{L}}) \overset{\eqref{intensity1}}{\le} \sum_{\substack{ {\Gamma} \ni  u,v,w\\ |{\Gamma}| =\ell}} (2d)^{-|\Gamma|}=\sum_{0\le t_1 \le t_2 \le \ell} \P( X_{0} = u, X_{t_1} =v, X_{t_2} =w, {X_{\ell} =u}).$$
Thus summing over the above, we get
\begin{align*} 
  \sum_{\substack{ {\Gamma} \ni  u,v,w\\ |{\Gamma}| \ge 1}}\P({\Gamma}  \in {\mathcal{L}}) &\le
{\sum_{0\le t_1 \le t_2\le t_3} \P( X_{0} = u, X_{t_1} =v, X_{t_2} =w, X_{t_3} =u)} \\
&=\sum_{0\le t_1 \le t_2\le t_3} p_{t_1}(u,v)p_{t_2-t_1}(v,w) p_{t_3-t_2}(w,u)  \\
&\le  C\Big[\sum_{t_1} t_1^{-d/2} e^{-|v-u|^2/ 2t_1} \Big] \Big[ \sum_{t_2'} (t_2')^{-d/2} e^{-|w-v|^2/ 2t_2'}\Big]\Big[\sum_{t_3'}  (t_3')^{-d/2} e^{-|u-w|^2/ 2t_3'}  \Big]\\ 
& \le  C  |u-v|^{2-d}|v-w|^{2-d}|w-u|^{2-d},
\end{align*}
where we used a change of variables $t_2'  = t_2-t_1$ and $t_3'  = t_3-t_2$.  The local CLT for the heat kernel is well known and for instance can {be found in \cite{rw}}.

\end{proof}
  As a corollary, we obtain the following two-point estimate which will be an important input.
{\begin{lemma}   \label{key6}
There exists $C>0$ such that    for any $u,v\in \Z^d$ and $i\geq 1$, 
   \begin{align} \label{2510}
  \sum_{\substack{ \Gamma \ni  u,v }}  |\Gamma|^i  \P({\Gamma}\in {\mathcal{L}})  \le C|u-v|^{4+2i-2d }.
 \end{align}
 In addition, there exists $C>0$ such that for any  $L\ge 1$ 
 \begin{align} \label{251}
 \sum_{\substack{ \Gamma \ni  u,v \\ |{\Gamma}| \ge L}}  \P({\Gamma}\in {\mathcal{L}})  \le C L^{1-d/2} |u-v|^{2-d} .
 \end{align}
 In particular, we have
  \begin{align} \label{25}
 \sum_{\substack{\Gamma \ni  u,v\\ |{\Gamma}| \ge L}}\P({\Gamma}\in {\mathcal{L}})  \le C L^{1/2-d/4} |u-v|^{3-3d/2} .
 \end{align}
\end{lemma}
}
{Note that the estimate \eqref{25} is obtained by simply taking the geometric average of \eqref{2510} (with $i=0$)  and \eqref{251}. This bound will be convenient to use in some cases.}

\begin{proof}
Note that in the case $i=0$, \eqref{2510}  is a special case of Lemma \ref{key0}, by taking $w=v$. {Let us prove \eqref{2510}  for general $i\ge 1$.  Similarly as in the proof of Lemma \ref{key0}, 
$$\sum_{\substack{ {\Gamma} \ni  u,v\\ |{\Gamma}| =\ell}} \P({\Gamma}  \in {\mathcal{L}}) \overset{\eqref{intensity1}}{\le} \sum_{\substack{ {\Gamma} \ni  u,v\\ |{\Gamma}| =\ell}} (2d)^{-|\Gamma|}=\sum_{0\le t \le \ell} \P( X_{0} = u, X_{t} =v, X_{\ell} =u) = \sum_{0\le t \le \ell}  p_t(u,v) p_{\ell-t}(v,u).$$
Thus  LHS of \eqref{2510} is bounded by
  \begin{align*}
\sum_{\ell=1}^\infty \ell^i \sum_{\substack{ {\Gamma} \ni  u,v\\ |{\Gamma}| =\ell}} \P({\Gamma}  \in {\mathcal{L}}) &\le \sum_{t_1,t_2} (t_1+t_2)^i p_{t_1}(u,v)p_{t_2}(v,u)  \\
& \le   C \Big[ \sum_{t_1} t_1^i p_{t_1}(u,v) \Big] \Big[  \sum_{t_2} p_{t_2}(v,u) \Big]  + C\Big[ \sum_{t_1} p_{t_1}(u,v) \Big] \Big[  \sum_{t_2}t_2^i p_{t_2}(v,u) \Big]  \\ 
&\le C|u-v|^{2+2i-d } \cdot C|u-v|^{2-d } =  C|u-v|^{4+2i-2d }.
 \end{align*}} 

Next,  similarly as above,  LHS of \eqref{251} is bounded by
     \begin{align*}
  \sum_{ \substack{ 0\le t\le \ell \\ \ell \ge L} } p_{t}(u,v)p_{\ell-t}(v,u)  & \le  \Big[ \sum_{t_1 \ge L/2}  p_{t_1}(u,v) \Big] \Big[  \sum_{t_2} p_{t_2}(v,u) \Big]  + \Big[ \sum_{t_1} p_{t_1}(u,v) \Big] \Big[  \sum_{t_2 \ge L/2} p_{t_2}(v,u) \Big]\\
   &\le   CL^{1-d/2}|u-v|^{2-d}  .
 \end{align*}
{Finally as mentioned above, \eqref{25} is obtained by taking the geometric average of \eqref{2510} (with $i=0$) and \eqref{251}.}
\end{proof}

\begin{remark}\label{adjacent1}
In our applications, it will be useful to have a version of the above estimates where the lattice points are not on the loops but  \emph{adjacent} to them.
For a discrete loop $\Gamma$ and $x\in \Z^d$, we say that $\Gamma \sim x $ if there exists a lattice point $y$ on $\Gamma$ such that $d^\extr(y,x)\le 1$ (note that $x$ can lie on the discrete loop $\Gamma$). 
Then by a union bound and translation invariance, all the estimates continue to hold even when the condition $\Gamma \ni u_1,\cdots,u_k$ is replaced by $\Gamma \sim u_1,\cdots,u_k$. 
\end{remark}

Our final lemma is simply a consequence of the central limit theorem, stating that the diameter of a loop of length $\ell$ is around $\sqrt{\ell}.$ 
 \begin{lemma} \label{two clt} 
For any $\kappa>0$, there exists $C>0$ such that for any $u,v\in \Z^d$,
     \begin{align} 
    \sum_{\substack{ \Gamma \ni  u,v
 \\ |\Gamma| \le  |u-v|^{2(1-\kappa)}  } } \P(\Gamma \in {\mathcal{L}}) \le   C e^{-|u-v|^{2\kappa}/4}.
\end{align}
 \end{lemma}
 \begin{proof}
{Similarly as in the proof of Lemma \ref{key6}, the above quantity is bounded by 
      \begin{align*} 
      \sum_{t=1}^{|u-v|^{2(1-\kappa)}}   \sum_{0\le t_1 \le t}  p_{t_1}(u,v)  p_{t-t_1} (v,u) &\le \sum_{t=1}^{|u-v|^{2(1-\kappa)}}   \sum_{0\le t_1 \le t}  p_{t_1}(u,v)  \\
&\le  C |u-v|^{2(1-\kappa)}  \sum_{t=1}^{|u-v|^{2(1-\kappa)}}p_{t_1}(u,v)  \le  C e^{-|u-v|^{2\kappa}/4}.
 \end{align*}}
 \end{proof}

\subsubsection{Applications}
In this section, we record some crucial corollaries of the estimates developed so far.  Recall that for a discrete loop $\Gamma$ and $x\in \Z^d$, we say $\Gamma \sim x $ if there exists a lattice point $y$ on $\Gamma$ such that $d^\extr(y,x)\le 1$. 
{For later purposes (e.g. in Lemma \ref{lemma 3.4} later),  we will encounter expressions of the following type
    \begin{align*}
        \sum_{{\Gamma}}\sum_{ \substack{w_1\in \Z^d \\ \Gamma \sim  w_1 }}\sum_{ \substack{w_2\in \Z^d \\ \Gamma \sim  w_2}} |{\Gamma}| \P(z_1 \ar w_1 \circ z_2 \ar w_2) \P({\Gamma} \in {\mathcal{L}}) .
            \end{align*}
      By the BKR inequality, it suffices to bound the quantity in the following lemma.}
    
\begin{lemma} \label{lemma2.9}
There exists $C>0$ such that  for any $z_1,z_2\in \Z^d,$ 
    \begin{align*}
        \sum_{{\Gamma}}\sum_{ \substack{w_1\in \Z^d \\ \Gamma \sim  w_1 }}\sum_{ \substack{w_2\in \Z^d \\ \Gamma \sim  w_2}} |{\Gamma}| \P(z_1 \ar w_1 )\P(z_2 \ar w_2) \P({\Gamma} \in {\mathcal{L}})  \le C |z_1-z_2|^{4-d}. 
    \end{align*}
\end{lemma}
\begin{proof}
Interchanging the sums and {applying Lemma \ref{key6}, we bound the above quantity by}
\begin{align} \label{ubi}
    & \sum_{w_1,w_2 \in \Z^d } \sum_{\Gamma \sim  w_1,w_2} |{\Gamma}|\P(z_1 \ar w_1)\P(z_2 \ar w_2) \P({\Gamma} \in {\mathcal{L}}) \nonumber \\
    &\le C\sum_{w_1,w_2 \in \Z^d } |z_1-w_1|^{2-d} |z_2-w_2|^{2-d} 
\sum_{\Gamma \sim  w_1,w_2}|{\Gamma}|\P({\Gamma} \in {\mathcal{L}}) \nonumber  \\
 &\le  C\sum_{w_1,w_2 \in \Z^d } |z_1-w_1|^{2-d} |z_2-w_2|^{2-d} |w_1-w_2|^{6-2d}. 
\end{align}
{In order to bound this quantity, we need the following straightforward by technical estimate  which will be stated and proved formally later in the appendix (see Lemma \ref{lemma basic}):  For $u,v\in \Z^d$ and  $\alpha_1,\alpha_2<0
$ such that $\alpha_1+\alpha_2<-d
$,
\begin{align}    \label{boundbound}
 \sum_{z\in \Z^d} |u-z|^{\alpha_1} |v-z|^{\alpha_2}\le \begin{cases}
   C |u-v|^{\alpha_1+\alpha_2+d} &\quad \text{if } \min \{\alpha_1,\alpha_2\}>-d , \\
     C |u-v|^{\max\{\alpha_1, \alpha_2\}}  &\quad  \text{if } \min \{\alpha_1,\alpha_2\}< -d.
     \end{cases}
     \end{align}
}
       Using this recursively for the summation over $w_2$ and then over $w_1$,  \eqref{ubi} is bounded by
\begin{align*}
C\sum_{w_1 \in \Z^d } |z_1-w_1|^{2-d} |z_2-w_1|^{2-d} \le C |z_1-z_2|^{4-d}.
\end{align*}
\end{proof}

The final lemma in this subsection bounds  the probability of the existence of a discrete loop, passing through the origin, connected to the given point $x\in \Z^d.$ {This will be used later when controlling the volume of balls, conditioned on the origin being connected to far away points (see Section \ref{section 7}).}

\begin{lemma} \label{lemma 3.1}
There exists $C>0$ such that  for any $x\in \Z^d$ and $L\in \N$ {with $L \le |x|$},
    \begin{align} \label{3333}
  \sum_{\substack{|{\Gamma}| \ge  L \\ {\Gamma} \ni 0
    } } 
 \P ( {\Gamma} \ar x)\P({\Gamma} \in {\mathcal{L}})  \le   \sum_{\substack{|{\Gamma}| \ge  L \\ {\Gamma} \ni 0
    } } \sum_{w\in {\Gamma}} 
 \P (  w  \ar x)\P({\Gamma} \in {\mathcal{L}})  \le C {|x|^{3-d}   L ^{2-d/2}}.
\end{align}
\end{lemma}
{While the above result may be sharpened, we present the current version for the sake of simplicity since this will suffice for our applications.}

\begin{proof}
For each non-negative integer $k$, we aim to bound the quantity
\begin{align*}  
    \sum_{\substack{|{\Gamma}| \ge  L \\ {\Gamma} \ni 0
    } }  |{\Gamma} \cap  \partial  B^\extr(x,k)|\P({\Gamma} \in {\mathcal{L}}) .
\end{align*}
{First, observe that for any $k\ge 0,$  we have the naive bound}
\begin{align}   \label{215}
  \sum_{\substack{|{\Gamma}| \ge  L \\ {\Gamma} \ni 0
    } }  |{\Gamma} \cap  \partial  B^\extr(x,k)|\P({\Gamma} \in {\mathcal{L}})  \le  \sum_{\substack{|{\Gamma}| \ge L\\ {\Gamma} \ni 0
    } }|{\Gamma}|\P({\Gamma} \in {\mathcal{L}})  \le C  L^{2-d/2}. 
\end{align}
We obtain a non-trivial bound in the case  $k\le |x|/2$, where the condition $|\Gamma| \ge L$ becomes vacuous. For such $k$, we have  $|y| \ge |x|/2$ for any $y\in \partial  B^\extr(x,k).$ Thus by Lemma \ref{key6} (in particular \eqref{251}),  \begin{align*}
  \sum_{\substack{|{\Gamma} |  \ge  |x|^{1.8}   \\ {\Gamma} \ni 0
    } }|{\Gamma} \cap  \partial  B^\extr(x,k)| \P({\Gamma} \in {\mathcal{L}}) = \sum_{y\in \partial  B^\extr(x,k)} \sum_{\substack{|{\Gamma}| \ge   |x|^{1.8} \\ {\Gamma} \ni 0,y
    } } \P({\Gamma} \in {\mathcal{L}}) \le Ck^{d-1} \cdot ( |x|^{1.8})^{1-d/2}|x|^{2-d}.
\end{align*}
In addition by Lemma \ref{two clt},
\begin{align*}
  \sum_{\substack{|{\Gamma}| \le |x|^{1.8} \\ {\Gamma} \ni 0
    } }|{\Gamma} \cap  \partial  B^\extr(x,k)| \P({\Gamma} \in {\mathcal{L}})= \sum_{y\in \partial  B^\extr(x,k)} \sum_{\substack{|{\Gamma}|\le |x|^{1.8} \\ {\Gamma} \ni 0,y
    } }   \P({\Gamma} \in {\mathcal{L}})\le Ck^{d-1} \cdot e^{-|x|^{0.2}/4}.
\end{align*}
Combining the above two estimates, we deduce that for $k\le |x|/2,$
\begin{align*} 
      \sum_{\substack{ {\Gamma} \ni 0
    } }  |{\Gamma} \cap  \partial  B^\extr(x,k)|\P({\Gamma} \in {\mathcal{L}}) &\le     C k^{d-1} (|x|^{1.8})^{1-d/2} |x|^{2-d}  +   Ck^{d-1}e^{-|x|^{0.2}/4} \le    C k^{d-1} |x|^{3.8-1.9d}.
\end{align*}
Therefore combining this with \eqref{215}, using the two-point estimate \eqref{two point},
\begin{align*}
       \sum_{\substack{|{\Gamma}| \ge  L \\ {\Gamma} \ni 0
    } } \sum_{w\in {\Gamma}} 
 \P (  w \ar x) \P({\Gamma} \in {\mathcal{L}}) & \le  \sum_{\substack{|{\Gamma}| \ge  L \\ {\Gamma} \ni 0
    } } \sum_{k=0}^\infty \sum_{y\in  {\Gamma} \cap  \partial  B^\extr(x,k)} \P(y \ar x)  \P({\Gamma} \in {\mathcal{L}}) \\
       & \le   C  \sum_{k=0}^\infty k^{2-d}  \sum_{\substack{|{\Gamma}| \ge  L \\ {\Gamma} \ni 0
    } }  |{\Gamma} \cap  \partial  B^\extr(x,k)|\P({\Gamma} \in {\mathcal{L}})  \\
       &\le C \Big(
 \sum_{k=0}^{|x|/2} k^{2-d} \cdot k^{d-1} |x|^{3.8-1.9d}+ \sum_{k=|x|/2}^\infty k^{2-d} \cdot L ^{2-d/2} \Big )  \\
 &\le C ( |x|^{5.8-1.9d} +  |x|^{3-d}   L ^{2-d/2})  \le C  |x|^{3-d}   L ^{2-d/2},
 \end{align*}
 where we used $L \le |x|$ in the last inequality. Therefore, we
  conclude the proof.
\end{proof}

\subsection{A result by Barlow-J\`arai-Kumagai-Slade}
We conclude this section by stating the crucial general result by Barlow-J\`arai-Kumagai-Slade \cite{barlow}, which has already been alluded to multiple times in Section \ref{iop}. This reduces the task of obtaining sharp heat kernel asymptotics to establishing effective resistance and volume estimates. To state the result, {let  $(\Omega, \mathcal{F}, \P)$ be the probability space carrying  random graphs $\mathcal{G}(w)$ with $w\in \Omega$. We assume that for every $w\in \Omega$, $\mathcal{G}(w)$ is infinite, locally finite, connected and contains a marked vertex $0\in \mathcal{G}$ which we call origin. For $r\in \N$, let $B(0,r) $  be the subgraph of $\mathcal{G}$ induced by vertices whose distance from the origin is at most $r$, regarded as an
electric network where the resistance (or conductance) associated to each edge is given by 1.}
Then, define $ R_{\textup{eff}}(0, \partial   B     (0,r) )$  to be the effective resistance between 0 and $ \partial   B     (0,r) $.

For $r\in \N,$
let $U(\lambda)$ be the collection of $r$ such that the following conditions hold:
\begin{enumerate}
    \item $\lambda^{-1}r \le |B(0,r)| \le \lambda r$,
    \item $ R_{\textup{eff}}(0, \partial  B     (0,r) ) \ge \lambda^{-1}r.$
\end{enumerate}
\begin{proposition}[{Theorems 1.5 and 1.6 in \cite{barlow}}] \label{barlow}
{Assume that there exists $c_0>0$ such that all degrees of $\mathcal{G}(w)$ are at most $c_0$ for every $w\in \Omega$. 
Suppose that there exist $C,c,r_0>0$ such that for any $\lambda>1$ and $r\in \N$ with $r >r_0$,
    \begin{align*}
        \P( r \in U(\lambda)) \ge 1-C\lambda^{-c}.
    \end{align*}
Then  $\P$-a.s.,  the conclusions of Theorem \ref{main} hold for the simple random walk on the graph $\mathcal{G}$.}
\end{proposition}

In the perspective of this theorem, we aim to estimate the volume of balls and the effective resistance.  

\section{Big-loop ensembles} \label{section 3}

In this section, we introduce the notion of \emph{big-loop ensembles} or $\ble$s which as indicated in Section \ref{iop} will feature centrally in the majority of our proofs.  
We first start with introducing some language involving geodesics.
\subsection{Geodesics} We start with a related definition.

\begin{definition} \label{def1}
Let $\ell : [0,T] \rightarrow \cable$ be a path on $\cable.$ 
A sequence of glued loops $\overline{L}= ( \overline{\Gamma}_1,\overline{\Gamma}_2,\cdots, \overline{\Gamma}_k)$ is called a {\emph{glued loop sequence of $\ell$}}  if there exist $0=t_0<t_1<\cdots<t_k=T$ such that  for any $i=1,2,\cdots,k$,
\begin{align*}
    \ell([t_{i-1},t_i]) \text{ is a subset of }  \overline{\Gamma}_i.
\end{align*}
In addition, for $A,B \subseteq  \Z^d$,  a sequence of glued loops $\overline{L}$ is  called a \emph{glued loop sequence from $A$ to $B$}  if there exists a path $\ell$ from   $A$ to  $
B$ such that $\overline{L}$ is a glued loop sequence for $\ell$.
\end{definition}

To prevent confusion, let us remark that all paths dealt with in the paper will in fact be lattice paths, i.e. a sequence of adjacent lattice edges, since given any path (not necessarily lattice) in the loop soup between two lattice points, one can ignore the partial edges, if any, and form a lattice path without increasing the length. Thus, this implicit assumption will always stand without us explicitly repeating it. \\

Throughout the following discussion we will assume that there is a collection of loops $\widetilde{\mathscr{S}}$ (in our eventual applications it will be taken to be $\bwt\cL$) in $\cable$ and consider the subset of the latter obtained by the union of the ranges of the loops. This will be taken to be our intrinsic metric space on which we will be defining the notions of geodesics and related objects.
 
Beginning with the notion of geodesics,
{for $A,B\subseteq \Z^d$, we say that a lattice path with starting and ending points in $A$ and $B$ respectively (which we will often term as a path from $A$ to $B$), is a geodesic if it has the {smallest length} (number of edges) among all paths from  $A$ to $B$.}

\begin{definition}[Path-geodesic]  \label{pathgeo}
Let  $A,B\subseteq \Z^d$ and $\ell$ be a path on $\cable$. Let $\overline{\mathscr{S}}$ be a collection of  glued loops in $\overline{\mathcal{L}}$.

{1. $\ell$ is called an \emph{$\overline{\mathscr{S}}$-path  from $A$ to $B$} if there exists a glued loop sequence  of $\ell$ from $A$ to $B$, which only  uses glued loops in $\overline{\mathscr{S}}$.}

{2. An $\overline{\mathscr{S}}$-path $\ell$ is called  an \emph{$\overline{\mathscr{S}}$-geodesic  from $A$ to $B$}   if it is the {shortest} path among all $\overline{\mathscr{S}}$-paths from $A$ to $B$.}
\end{definition}
{Note that if an $\overline{\mathscr{S}}$-path is a geodesic from $A$ to $B$, then it is also an $\overline{\mathscr{S}}$-geodesic from $A$ to $B$. However,  clearly, an $\overline{\mathscr{S}}$-geodesic is not necessarily a geodesic.}

~

Next, we define the notion of geodesics for a glued loop sequence.

\begin{definition}[Loop-geodesic]
Let $A,B\subseteq \Z^d$  and $\overline{L}$ be a glued loop sequence {from $A$ to $B$}. 

1. $\overline{L}$ is called   a \emph{loop-geodesic} if there exists a path $\ell$, which is a geodesic  from $A$ to $B$, such that $\overline{L}$ is a glued loop sequence of $\ell$.

    2. {$\overline{L}$ is called a \emph{local loop-geodesic} if there exists a path $\ell$, which is a $\textsf{Set}(\overline{L})$-geodesic  from $A$ to $B$, such that $\overline{L}$ is a glued loop sequence of $\ell$.}
    \end{definition}
    In other words, 2. says that a glued loop sequence $\overline L$ is a local loop-geodesic if there exists a path with glued loop sequence $\overline L$ which has the shortest length \emph{among all paths with a loop sequence using elements of $\overline L$} while the definition in 1. pertains to the case when the local geodesic is in fact the global geodesic.

We next record the following simple lemma which allows us to extract a local loop-geodesic, given a collection  of glued loops.

\begin{lemma} \label{geometry4}
Let $A,B\subseteq \Z^d$  and $\overline{\mathscr{S}}$ be a collection of {glued} loops for which there exists an $\overline{\mathscr{S}}$-path from $A$ to $B$. Then, there exists a local loop-geodesic $\overline{L}$ from $A$ to $B$ such that $\textup{\textsf{Set}}(\overline{L}) \subseteq \overline{\mathscr{S}}.$
\end{lemma}

\begin{proof}
{Let $\ell$ be an  $\overline{\mathscr{S}}$-geodesic  from $A$ to $B$} (which exists because we consider only lattice paths)
     and $\overline{L}$ be any glued loop sequence of $\ell$, consisting  only of the glued loops in $\overline{\mathscr{S}}$. Then, $\overline{L}$ is a  local loop-geodesic satisfying 
     $\textsf{Set}(\overline{L}) \subseteq \overline{\mathscr{S}}.$
 \end{proof}

With the above preparation we now proceed to introducing $\ble$s and several key statements involving them. 
\subsection{Big-loop ensembles }
As alluded to in Section \ref{iop}, a big-loop ensemble is a collection of three discrete loops of large sizes which are close to each other. 
{From now on, in order to simplify  explanations, we also regard every lattice point $x\in \Z^d$ as a discrete loop. 
In other words, we include all lattice points in the discrete loop soup $\mathcal{L}$ {defined in \eqref{dl}.}
We call that such a discrete loop $\Gamma = \{x\}$  is of \emph{singleton} type.  Note that this is different from a discrete loop $\Gamma = (x)$ induced by a point loop passing through a lattice point $x\in \Z^d$. These will be used as discrete proxies for glued edge loops supported on an edge adjacent to the corresponding point, which do not have a natural discrete projection.

\begin{definition}  \label{nearby}
Let ${\Gamma}_1,{\Gamma}_2,{\Gamma}_3$ be discrete loops and $u_1,u_2,u_3$ be points in $\Z^d$. We say that a tuple (${\Gamma}_1,{\Gamma}_2,{\Gamma}_3,u_1,u_2,u_3$) is a \emph{big-loop ensemble} (\ble) if 
     \begin{enumerate}
         \item $d^\extr({\Gamma}_i,u_i) \le 1$ for $i=1,2,3.$
         \item  $|{\Gamma}_2| +2\ge |u_2-u_3|$ and $ |{\Gamma}_3| +2\ge|u_1-u_3|.$
     \end{enumerate}
 The above definition does not preclude  $\Gamma_i$s from being identical to one or all of the members of the tuple. Though degenerate, such cases would indeed have to be considered in our proofs. 
 
While the nature of the definition might indicate that it would have been more natural to switch the labels of $\Gamma_2$ and $\Gamma_3$, we stick to this for a geometric reason which will become evident soon.  
 
Finally, to avoid introducing new terminology, we will also use  ${\ble}$ to denote the collection of all such tuples $(\{{\Gamma}_i\}_{i=1}^3,\{u_i\}_{i=1}^3).$

\end{definition}

We now make the discussion in Section \ref{iop} formal by showing how $\ble$s arise in the analysis of events of the form  $\{A \arr B_1, A\ar B_2\}$, where $A, B_1,B_2 \subset \Z^d$. Before stating the  main proposition, we recall some notations. 
For a discrete loop $\Gamma$, as defined in \eqref{trace}, $\overline{\Gamma}$  denotes the corresponding glued loop with the {same $\trace$}.  {Note that different discrete loops can give rise to the same glued loop.} {Also for a discrete loop $\Gamma = \{x\}$ of singleton type, we set $\overline{\Gamma}:=\emptyset$.} Conversely for a glued loop $\overline{\Gamma}$,  we use {$\textup{\textsf{Dis}} (\overline{\Gamma})$} to denote the collection of corresponding discrete loops $\Gamma$ in $ \mathcal{L}$.

\begin{proposition} \label{grand} 
Let $A,B_1,B_2 \subseteq \Z^d$ and $r\in \N.$
Then, under the event $\{A \arr B_1, A\ar B_2\}$, there exist a   
$\textup{\ble}$ $(\{{\Gamma}_i\}_{i=1}^3,\{u_i\}_{i=1}^3)$ with ${\Gamma}_1,{\Gamma}_2,{\Gamma}_3\in {\mathcal{L}}$, a discrete loop $\Gamma\in   {\mathcal{L}}$
and $v_1,v_2,v_3,w_1,w_2,w_3\in \Z^d$  such that the following holds. 
\begin{enumerate}
    \item $d^\extr({\Gamma}_i,v_i) \le 1$ for $i=1,2,3$.
        \item $d^\extr(\Gamma,w_i) \le 1$ for $i=1,2,3$.
    \item The following connections occur disjointly:  
    \begin{align*}
&v_1\overset{\overline{\mathcal{L}} \setminus \{\overline{\Gamma},\overline{\Gamma}_1,\overline{\Gamma}_2,\overline{\Gamma}_3  \}}{\longleftrightarrow} B_2 \circ v_2 \overset{\overline{\mathcal{L}} \setminus \{\overline{\Gamma},\overline{\Gamma}_1,\overline{\Gamma}_2,\overline{\Gamma}_3  \},r}{\longleftrightarrow}  A \circ  v_{3} \overset{\overline{\mathcal{L}} \setminus \{\overline{\Gamma},\overline{\Gamma}_1,\overline{\Gamma}_2,\overline{\Gamma}_3  \},r}{\longleftrightarrow}  B_1 \\
&\circ u_1 \overset{\overline{\mathcal{L}} \setminus \{\overline{\Gamma},\overline{\Gamma}_1,\overline{\Gamma}_2,\overline{\Gamma}_3 \}}{\longleftrightarrow}  w_1 \circ u_2 \overset{\overline{\mathcal{L}} \setminus \{\overline{\Gamma},\overline{\Gamma}_1,\overline{\Gamma}_2,\overline{\Gamma}_3 \}}{\longleftrightarrow}  w_2 \circ u_3 \overset{\overline{\mathcal{L}} \setminus \{\overline{\Gamma},\overline{\Gamma}_1,\overline{\Gamma}_2,\overline{\Gamma}_3 \}}{\longleftrightarrow}  w_3. 
 \end{align*}
\end{enumerate} 
\end{proposition}
While we state the above result in some generality, in our applications later, $A$ will be taken as a set consisting of one  or two lattice points.

~

Proposition \ref{grand} is an immediate consequence of the following two  lemmas. 
{The first lemma extracts a $\ble$ and disjoint connections therefrom to $B_1, B_2$ and $A$.}
We start with some notations.
For a collection {$\overline{\mathscr{S}}$} of glued loops, let
{\begin{align}\label{notation100}
    \textsf{Dis}(\overline{\mathscr{S}}):=  \cup_{\overline{\Gamma}\in      \overline{\mathscr{S}}}  \textsf{Dis}(\overline{\Gamma}).
    \end{align}} 
Similarly for a glued loop sequence $\overline{L}$, set $\textsf{Dis}(\overline{L}):= \textsf{Dis} ( \sett(\overline{L}))$.

\begin{lemma} \label{geometry}
Let $A,B_1,B_2 \subseteq \Z^d$ and $r\in \N$. Suppose that $\overline{\mathscr{S}}$ is a collection of glued loops.
Then, under the event  
\begin{align*}
    \{A \overset{\overline{\mathscr{S}},r}{\longleftrightarrow} B_1, \  A \overset{\overline{\mathscr{S}}}{\longleftrightarrow}  B_2\},
\end{align*}
there exists a $\textup{\ble}$$(\{{\Gamma}_i\}_{i=1}^3,\{u_i\}_{i=1}^3)$ with  ${\Gamma}_1,{\Gamma}_2,{\Gamma}_3 \in     \textup{\textsf{Dis}} (\overline{\mathscr{S}}) \cup \Z^d$   and $v_1,v_2,v_3\in \Z^d$ such that
\begin{enumerate}
    \item $d^\extr({\Gamma}_i,v_i) \le 1$ for $i=1,2,3$.
\item  The following disjoint connections hold:
\begin{align} \label{later}
v_1\overset{\overline{\mathscr{S}} \setminus \{\overline{\Gamma}_1,\overline{\Gamma}_2,\overline{\Gamma}_3  \}}{\longleftrightarrow} B_2 \circ v_2 \overset{\overline{\mathscr{S}} \setminus \{\overline{\Gamma}_1,\overline{\Gamma}_2,\overline{\Gamma}_3  \},r}{\longleftrightarrow}  A \circ   v_{3} \overset{\overline{\mathscr{S}}\setminus \{\overline{\Gamma}_1,\overline{\Gamma}_2,\overline{\Gamma}_3  \},r}{\longleftrightarrow}  B_1 \circ \{u_1\overset{\overline{\mathscr{S}} \setminus \{\overline{\Gamma}_1,\overline{\Gamma}_2,\overline{\Gamma}_3  \}}{\longleftrightarrow} u_2, u_2\overset{\overline{\mathscr{S}} \setminus \{\overline{\Gamma}_1,\overline{\Gamma}_3  \}}{\longleftrightarrow} u_3\}.
 \end{align}
 \end{enumerate}
In addition, ${\Gamma}_1,{\Gamma}_2,{\Gamma}_3$ can be taken to satisfy $d^\extr({\Gamma}_i,A) \le r$ for  $i=1,2,3$. Note that the last connection $u_2\overset{\overline{\mathscr{S}} \setminus \{\overline{\Gamma}_1,\overline{\Gamma}_3  \}}{\longleftrightarrow} u_3$  above may use the glued loop  $\overline{\Gamma}_2$. 
\end{lemma}

The content of the lemma is illustrated earlier in Figure \ref{iop4}, except now the points $0, z, x$ are replaced by sets $A, B_1, B_2$ (see also the upcoming Figure \ref{3.9}).

The next ingredient in the proof of Proposition \ref{grand}, analyzes the event 
\begin{equation}\label{keyevent123}
\{u_1\overset{\overline{\mathscr{S}} \setminus \{\overline{\Gamma}_1,\overline{\Gamma}_2,\overline{\Gamma}_3  \}}{\longleftrightarrow} u_2, u_2\overset{\overline{\mathscr{S}} \setminus \{\overline{\Gamma}_1,\overline{\Gamma}_3  \}}{\longleftrightarrow} u_3\}
\end{equation} and further extracts disjoint connections based on a tree expansion argument.
For a path $\ell$ on $\cable,$  we denote {by   $\trace(\ell)$ the collection of lattice points on $\ell$.}  
\begin{figure}[h]
\centering
\includegraphics[scale=.7]{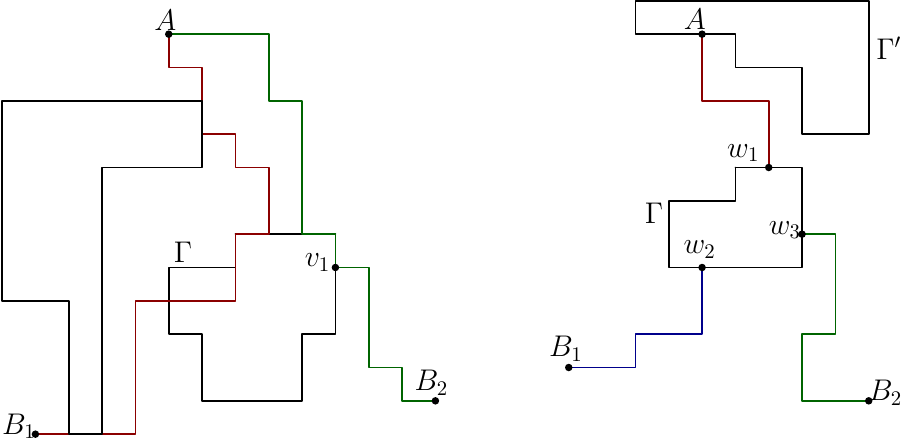}
\caption{The left figure  illustrates the first item in Lemma \ref{geometry6}. The right figure describes the $\ete$ to extract disjoint connections connecting the elements of the $\ble$, which is the content of the second item in Lemma \ref{geometry6}.
}
\label{3.8}
\end{figure}

\begin{lemma}[Tree expansion] \label{geometry6}
Assume that $A,B_1,B_2 \subseteq \Z^d$. 
For $i=1,2$, let $\ell_i$   be a path from $A$ to $B_i$ and   $\overline{L}_i$ be  any  corresponding glued loop sequence. We then have the following two conclusions.\\

\noindent
1. There exist   ${\Gamma}  \in     \textup{\textsf{Dis}} (\overline{L}_1) \cup \trace(\ell_1)$  and  $v_1\in \Z^d$ such that 
{ \begin{enumerate}
     \item $d^\extr({\Gamma},v_1) \le 1$.
     \item  The following connection holds:
     \begin{align} \label{304}
 v_1 \overset{   \textup{\textsf{Set}}(\overline{L}_2) \setminus \textup{\textsf{Set}}(\overline{L}_1) }{\longleftrightarrow} B_2.
 \end{align}
 Further, as the proof will show, the chemical length of the above connection can be ensured to be at most that of $\ell_2.$
 \end{enumerate}}

 \noindent
 2. {Let $\overline{\Gamma}'$ be any glued loop such that $d^\extr(\overline{\Gamma}', A) \le 1$.} While we are stating this result quite generally without providing any further elaboration on the role of $\overline \Gamma',$ in our applications it will typically be taken to be a glued loop in $\overline L_2$ (for instance, $\overline \Gamma_2$ appearing in \eqref{keyevent123}). Then there exist  ${\Gamma}   \in    {\mathcal{L}}$ along with points $w_1,w_2,w_3\in \Z^d$ such that 
\begin{enumerate}
    \item $d^\extr({\Gamma},w_i) \le 1$ for $i=1,2,3$.
    \item The following connection holds: 
    \begin{align} \label{300}
    w_1 \overset{\textup{\textsf{Set}}(\overline{L}_1) \setminus \{\overline{\Gamma} \}}{\longleftrightarrow}  A \circ w_2 \overset{\textup{\textsf{Set}}(\overline{L}_1) \setminus \{\overline{\Gamma} \}}{\longleftrightarrow}  B_1 \circ w_3\overset{\textup{\textsf{Set}}(\overline{L}_2) \setminus \{\overline{\Gamma},\overline{\Gamma}' \} }{\longleftrightarrow}  B_2.
\end{align}
\end{enumerate}

\end{lemma}
{To help the reader parse the two statements, we include a brief comment pointing out the distinction between the two.  In ${1}.$,   for a discrete loop  ${\Gamma}  \in     \textup{\textsf{Dis}} (\overline{L}_1) \cup \trace(\ell_1)$ as in the statement, the  two connections $A\ar \Gamma$ and $\Gamma \ar B_1$ are not guaranteed to use disjoint sets of glued loops. On the other hand,  in $2.,$ all the connections in \eqref{300} occur disjointly making the all useful BKR inequality. On the down side, unlike \eqref{304}, the connections $w_1\ar A$ and  $w_2\ar B_1$ in  \eqref{300} can have their chemical lengths to be significantly larger compared to that of $\ell_1$.}

~

Before the proofs of Lemmas \ref{geometry} and \ref{geometry6}, we first treat them as given and establish  Proposition \ref{grand} quickly.

\begin{proof}[Proof of Proposition \ref{grand}]
By Lemma  \ref{geometry} with $\overline{\mathscr{S}} := \overline{\mathcal{L}}$, there exist a  tuple $(\{ {\Gamma}_i\}_{i=1}^3,\{u_i\}_{i=1}^3)\in \textup{\ble}$ with $ {\Gamma}_1, {\Gamma}_2, {\Gamma}_3 \in  {\mathcal{L}}$ which could be of singleton type and $v_1,v_2,v_3\in \Z^d$  such that  $d^\extr( {\Gamma}_i,v_i) \le 1$   and the two events
\begin{align} \label{307}
v_1\overset{\overline{\mathcal{L}} \setminus \{\overline{\Gamma}_1,\overline{\Gamma}_2,\overline{\Gamma}_3  \}}{\longleftrightarrow} B_2 \circ v_2 \overset{\overline{\mathcal{L}} \setminus \{\overline{\Gamma}_1,\overline{\Gamma}_2,\overline{\Gamma}_3  \},r}{\longleftrightarrow}  A \circ   v_{3} \overset{\overline{\mathcal{L}} \setminus \{\overline{\Gamma}_1,\overline{\Gamma}_2,\overline{\Gamma}_3  \},r}{\longleftrightarrow}  B_1
\end{align}
    and
    \begin{align} \label{308}
         u_1\overset{\overline{\mathcal{L}}  \setminus \{\overline{\Gamma}_1,\overline{\Gamma}_2,\overline{\Gamma}_3  \}}{\longleftrightarrow} u_2, \quad u_2\overset{\overline{\mathcal{L}}  \setminus \{\overline{\Gamma}_1,\overline{\Gamma}_3  \}}{\longleftrightarrow} u_3
    \end{align}
    occur disjointly. We will now apply Lemma \ref{geometry6} to disjointify the events appearing in \eqref{308}. {More precisely, by the second item in Lemma \ref{geometry6} with {$A:= \{u_2\},  \ B_1:=\{u_1\}, \ B_2:=\{u_3\}$} and $\overline{\Gamma}' := \overline{\Gamma}_2 $, the connection \eqref{308} ensures the existence of a discrete loop
 {${\Gamma} \in {\mathcal{L}}$}}
and points $w_1,w_2,w_3\in \Z^d$  with $d^\extr({\Gamma},w_i) \le 1$  ($i=1,2,3$) such that   the event
    \begin{align}
        u_1 \overset{\overline{\mathcal{L}} \setminus \{\overline{\Gamma},\overline{\Gamma}_1,\overline{\Gamma}_2,\overline{\Gamma}_3 \}}{\longleftrightarrow}  w_1 \circ u_2 \overset{\overline{\mathcal{L}} \setminus \{\overline{\Gamma},\overline{\Gamma}_1,\overline{\Gamma}_2,\overline{\Gamma}_3 \}}{\longleftrightarrow}  w_2 \circ u_3 \overset{\overline{\mathcal{L}} \setminus \{\overline{\Gamma},\overline{\Gamma}_1,\overline{\Gamma}_2,\overline{\Gamma}_3 \}}{\longleftrightarrow}  w_3
    \end{align}
happens disjointly from  \eqref{307}. 
\end{proof}

We now move on to the proofs of Lemmas \ref{geometry} and \ref{geometry6}.
\subsection{Proofs of Lemmas \ref{geometry} and \ref{geometry6}}
A crucial input in the proof of Lemma \ref{geometry} will be the following lemma {which extracts a $\ble$ structure from a given local loop-geodesic. For $A,B\subseteq \Z^d$, we say that a local loop-geodesic $\overline{L}$ from $A$ to $B$  has  length at most $r$ if the length of a $\textsf{Set}(\overline{L})$-geodesic from $A$ to $B$ is  at most $r$.}

 \begin{figure}[h]
\centering
\includegraphics[scale=.7]{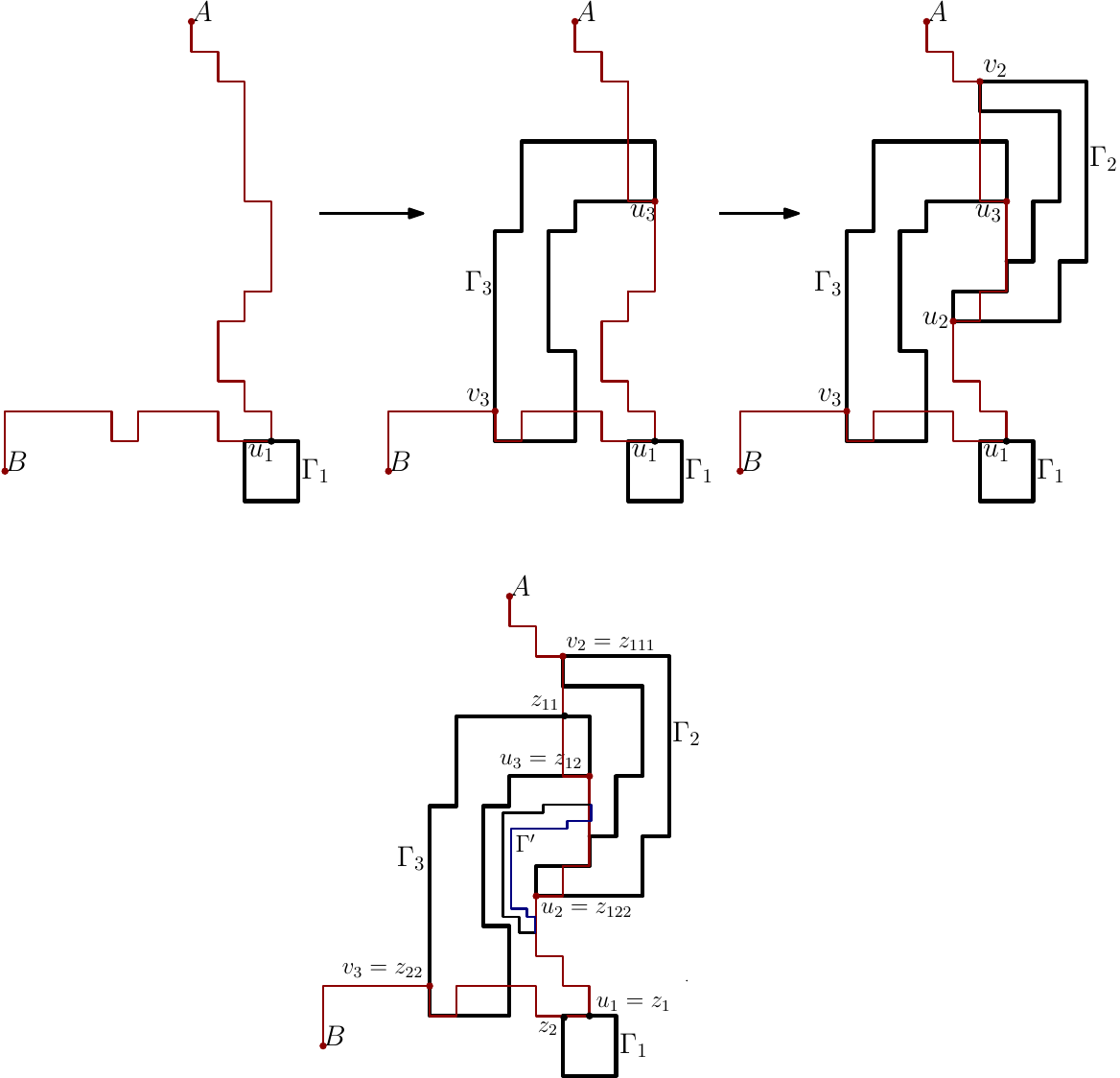}
\caption{The top three figures illustrate the $\ble$ extraction process which is the content of the proof of Lemma \ref{geometry2}. The bottom figure indicates some of the different entry and exit points that feature in the argument as well as that the connections $u_1\ar u_2$ and $u_2 \ar u_3$ are not disjoint.
}
\label{3.9}
\end{figure}

\begin{lemma} \label{geometry2}
   Assume that $A, B \subseteq \Z^d$ and $r\in \N$. Let $\overline{L}$ {be a local loop-geodesic of length at most $r$ from $A$ to $B$ and $\ell$ be {a} corresponding $\sett(\overline{L})$-geodesic. Then for any {${\Gamma}_1 \in     \textup{\textsf{Dis}} (\overline{L})\cup \trace(\ell)$},}  there exist ${\Gamma}_2,{\Gamma}_3\in     \textup{\textsf{Dis}} (\overline{L}) $ and $v_2,v_3,u_1,u_2,u_3\in \Z^d$ such that  
    \begin{enumerate}
    \item $(\{{\Gamma}_i\}_{i=1}^3,\{u_i\}_{i=1}^3)\in \textup{\ble}$.
        \item $d^\extr({\Gamma}_i,v_i) \le 1$ for $i=2,3.$
\item The following connection holds:
    \begin{align} \label{disjoint}
       v_2  \overset{\textup{\textsf{Set}}(\overline{L}) \setminus \{\overline{\Gamma}_1,\overline{\Gamma}_2,\overline{\Gamma}_3  \},r}{\longleftrightarrow}   A \circ  v_3 \overset{\textup{\textsf{Set}}(\overline{L}) \setminus \{\overline{\Gamma}_1,\overline{\Gamma}_2,\overline{\Gamma}_3  \},r}{\longleftrightarrow}   B \circ   \Big\{u_1\overset{\textup{\textsf{Set}}(\overline{L}) \setminus \{\overline{\Gamma}_1,\overline{\Gamma}_2,\overline{\Gamma}_3  \}}{\longleftrightarrow}   u_2, u_2 \overset{\textup{\textsf{Set}}(\overline{L}) \setminus \{\overline{\Gamma}_1,\overline{\Gamma}_3  \}}{\longleftrightarrow}  u_3\Big \}.
    \end{align}
        \end{enumerate}
        \end{lemma}

Given this, the proof of Lemma \ref{geometry} is quick.

\begin{proof}[Proof of Lemma \ref{geometry}]
For each $i=1,2$, let $\ell_i$ be a $\overline{\mathscr{S}}$-geodesic from $A$ to $B_i$ and  $\overline{L}_i$ be any glued loop sequence of $\ell_i$ such that  $\textup{\textsf{Set}}(\overline{L}_i) \subseteq \overline{\mathscr{S}}$.  Then, $\overline{L}_1$ is a local loop-geodesic of length at most $r$.

 By  the first item in Lemma \ref{geometry6}, there exist $ {\Gamma}_1 \in   \textup{\textsf{Dis}} (\overline{L}_1) \cup \trace(\ell_1)$ and  $v_1\in \Z^d$ with   $d^\extr( {\Gamma}_1 ,v_1) \le 1$ such that $ v_1 \overset{   \textup{\textsf{Set}}(\overline{L}_2)\setminus \textup{\textsf{Set}}(\overline{L}_1) }{\longleftrightarrow} B_2$. 
{Next, by Lemma \ref{geometry2}, there exist ${\Gamma}_2,{\Gamma}_3\in     \overline{L}_1^{\textup{dis}} $ and $v_2,v_3,u_1,u_2,u_3\in \Z^d$ such that the connection \eqref{disjoint} holds, with $\textup{\textsf{Set}}(\overline{L})$ replaced by $\textup{\textsf{Set}}(\overline{L}_1)$. Since the connection $ v_1 \overset{   \textup{\textsf{Set}}(\overline{L}_2)\setminus \textup{\textsf{Set}}(\overline{L}_1) }{\longleftrightarrow} B_2$ does not use glued loops in $\textup{\textsf{Set}}(\overline{L}_1)$, we conclude the proof.}
{Note that the condition $d^\extr({\Gamma}_i,A) \le r$ is satisfied  since $A \overset{\overline{\mathscr{S}},r}{\longleftrightarrow} B_1$ and ${\Gamma}_i$ intersects  this connection. }

 \end{proof}

We now move on to the proof of Lemma \ref{geometry2}. First,
it will be convenient to introduce the notion of ``entry'' and ``exit'' points. Recall that  for $v,w\in \Z^d$, $[v,w] $ denotes the compact interval, regarded as a subset of $\cable$, connecting $v$ and   $w$.
Let {$\ell = (v_1,v_2,\cdots,v_m)$} be a {lattice} path on $\cable$  connecting two lattice points ({i.e. $v_i$s denote consecutive lattice points on $\ell$}), and {let $A$ be a subset of $ \cable$ such that $A \cap \ell \neq \emptyset$ where we, abusing notation, also use $\ell$ to denote the range of the path  regarded as a subset of $\cable$, obtained by taking a union of compact intervals $[v_i,v_{i+1}] $s}. Then, the \emph{entry point} of $\ell$ into $A$ is the lattice point $w$ defined as
\begin{align*}
        w:=v_k \ &\text{ with }  \ k:=\min \{1\le i\le m: [v_i,v_{i+1}] \cap A \neq \emptyset\}.
\end{align*}
Analogously, the \emph{exit point} of $\ell$ from $A$  is the lattice point $w'$ defined as
{\begin{align*}
        w':=v_{k'} \ &\text{ with }  \ k':=\max \{1\le i\le m: [v_{i-1},v_{i}] \cap A \neq \emptyset\}.
\end{align*}}

\begin{proof}[Proof of Lemma \ref{geometry2}]
The proof consists of three steps (see Figures \ref{3.9} as well). 

\noindent
\textbf{Step 1. Construction of glued loops and lattice points.}
 Let $z_1$ and $z_2$ be the entry and exit points of $\ell$ into and from $\Gamma_1$ respectively.
Let $\ell_{1}:= \ell|_{A\rightarrow z_1}$ and $\ell_{2}:= \ell|_{z_2\rightarrow B}$ (i.e. sub-paths of $\ell_1$ and $\ell_2$ starting and ending at the specified sets),
and similarly set $\overline{L}_{1}:= \overline{L}|_{A\rightarrow z_1}$   and  $\overline{L}_{2}:= \overline{L}|_{z_2\rightarrow B}$ where the latter expressions are shorthands to denote the prefix and the suffix of the sequence $\overline{L}$ which cover the corresponding subpaths. 
Note that $\ell_{1}$ (resp. $\ell_{2}$) is a $\textsf{Set}(\overline{L})$-geodesic from $A$ to $z_1$ (resp. from $z_2$ to $B$) as well, and $\overline{L}_{i}$  is a glued loop sequence of $\ell_{i}$ for  $i=1,2$.

If $\sett(\overline{L}_1) \cap \sett(\overline{L}_2) = \emptyset$, then
\begin{align}
z_1 \overset{\textup{\textsf{Set}}(\overline{L}) \setminus \{\overline{\Gamma}_1\},r}{\longleftrightarrow}   A \circ  z_2 \overset{\textup{\textsf{Set}}(\overline{L}) \setminus \{\overline{\Gamma}_1\},r}{\longleftrightarrow}   B,
\end{align}
{and thus one can take $\Gamma_2 = \Gamma_3 :=\Gamma_1$ along with $v_2$ and $v_3$ the entry and exit points of $\ell$ into and from $\Gamma_1$ and $u_1=u_2=u_3:=v_2$.}

It remains to consider the case $\sett(\overline{L}_1) \cap \sett(\overline{L}_2) \neq  \emptyset$.  
Define $\overline{\Gamma}_3$ to be the last glued loop in  the sequence of glued loops $\overline{L}_{2}$ that belongs to $\textup{\textsf{Set}}(\overline{L}_1)$. Let   $z_{11}$ and $z_{12}$ be the entry and exit points of $\ell_{1}$ into and from $\overline{\Gamma}_3$ respectively. Similarly as above, set
\begin{align*}
    \ell_{11} :=\ell_{1}|_{A\rightarrow z_{11}} ,\quad  \ell_{12}:= \ell_{1}|_{z_{12}\rightarrow z_1} , \quad \overline{L}_{11} :=\overline{L}_{1}|_{A\rightarrow z_{11}},\quad  \overline{L}_{12}:= \overline{L}_{1}|_{z_{12}\rightarrow z_1}.
\end{align*}
Let $z_{22}$ be the point that $\ell_{2}$ exits from $\overline{\Gamma}_3$.
If $\sett(\overline{L}_{11}) \cap \sett( \overline{L}_{12} )= \emptyset$, then  we have
\begin{align}
    z_{11} \overset{\textup{\textsf{Set}}(\overline{L}) \setminus \{\overline{\Gamma}_1,\overline{\Gamma}_3\},r}{\longleftrightarrow} A \circ z_{12} \overset{\textup{\textsf{Set}}(\overline{L}) \setminus \{\overline{\Gamma}_1,\overline{\Gamma}_3\},r}{\longleftrightarrow} z_1 \circ  z_{22}
 \overset{\textup{\textsf{Set}}(\overline{L}) \setminus \{\overline{\Gamma}_1,\overline{\Gamma}_3\},r}{\longleftrightarrow} B.
\end{align} 
{Thus setting  $\Gamma_3$ to be
any discrete loop in $\textup{\textsf{Dis}} (\overline{\Gamma}_3)$,  one can take $\Gamma_2:= \Gamma_3$ along with $v_2:=z_{11},$ $v_3:=z_{22}$ and $u_2=u_3:= z_{12}$.}

Assume that $\sett(\overline{L}_{11}) \cap \sett( \overline{L}_{12} ) \neq  \emptyset$.  
Define $\overline{\Gamma}_2$  to be the first  glued loop in the sequence of  glued  loops $\overline{L}_{11}$ that belongs to $\textup{\textsf{Set}}(\overline{L}_{12})$. Let $z_{111}$ be the point that $\ell_{11}$ enters into  $\overline{\Gamma}_2$ and $z_{122}$ be  the point that $\ell_{12}$ exits from $\overline{\Gamma}_2$.
Letting $v_2 := z_{111}$ and $v_3 := z_{22}$, by our construction,
 \begin{align}\label{connectionBLE0}
 v_2  \overset{\textup{\textsf{Set}}(\overline{L}) \setminus \{\overline{\Gamma}_1,\overline{\Gamma}_2,\overline{\Gamma}_3  \},r}{\longleftrightarrow}   A \circ  v_3 \overset{\textup{\textsf{Set}}(\overline{L}) \setminus \{\overline{\Gamma}_1,\overline{\Gamma}_2,\overline{\Gamma}_3  \},r}{\longleftrightarrow}   B. 
 \end{align}

 Further we will extract loops $\Gamma_1, \Gamma_2$ and $\Gamma_3$ from the correspondingly indexed glued loops along with points $u_1,u_2,u_3$ on them and show that they satisfy the conclusions of the lemma. 
We first proceed towards verifying the $\ble$ condition. \\

\noindent
\textbf{Step 2. Verification of the $\ble$ condition.} 
{Note that by our construction, $\overline{\Gamma}_2$ and $\overline{\Gamma}_3$ are of fundamental type. This is because each of these loops  intersects at least two distinct lattice points on $\ell$.}
For $i=2,3$, let $\Gamma_i$ be
any discrete loop in {$\textup{\textsf{Dis}} (\overline{\Gamma}_i)$}.  
We claim that ${\Gamma}_1,{\Gamma}_2,{\Gamma}_3$  
along with the lattice  points
\begin{align*}
    u_1:=z_1 ,\quad u_2:= z_{122} , \quad u_3:= z_{12} 
\end{align*}
satisfy the conditions in Definition \ref{nearby}.  Property (1)  {is an immediate consequence of our} choice of the lattice points  $z_1, z_{122}$ and  $ z_{12} $.
We next verify condition (2).
 The shortest length of a path from $z_{111}$ to $z_{122}$ along  ${\Gamma}_2$ (with at most two additional edges connected to these points respectively) is at most $|{\Gamma}_2|+2  $. Since $\ell$ is a $\textsf{Set}( \overline{L})$-geodesic,  
\begin{align*}
    | {\Gamma}_2|+2  \ge |\ell|_{z_{111}\rightarrow z_{122}} |.
\end{align*}
As
\begin{align*}
    |\ell|_{z_{111}\rightarrow z_{122}} | \ge |\ell|_{z_{12}\rightarrow z_{122}} | \ge |z_{12}-z_{122}|  = |u_3-u_2|,
\end{align*} 
we deduce that $| {\Gamma}_2|+2 \ge  |u_3-u_2|$.
 Similarly, using a path from $z_{12}$ to $z_{22}$ along $\Gamma_3$, we obtain
 \begin{align*}
     | {\Gamma}_3|  +2 \ge |\ell|_{z_{12} \rightarrow z_{22}} |  \ge   |\ell|_{z_{12}\rightarrow z_{1}}| \ge |z_{12}-z_1|= |u_3-u_1|.
 \end{align*} 
This verifies that $(\{{\Gamma}_i\}_{i=1}^3,\{u_i\}_{i=1}^3) \in \textup{\ble}$.  \\

It remains to verify \eqref{disjoint}.\\

\noindent
\textbf{Step 3. Verification of the connection properties.}
Recall that with $v_2 = z_{111}$ and $v_3 = z_{22}$, $d^\extr({\Gamma}_i,v_i) \le 1$  ($i=2,3$). Also by our construction,   the connections  \eqref{connectionBLE0} (as a subsequence of $\overline{L}$) do \emph{not} use glued loops in  $\textsf{Set}(\overline{L}_{12})$.
As the connections $u_1\ar u_2$ and $ u_2 \ar u_3$ only use glued loops in  $\textsf{Set}(\overline{L}_{12})$,  the disjointness condition in \eqref{disjoint} holds. Finally,  the connection $u_1\ar u_2$ does \emph{not} use the glued loops $\overline{\Gamma}_1,\overline{\Gamma}_2,\overline{\Gamma}_3$ but the connection $ u_2 \ar u_3$  can use the glued loop $\overline{\Gamma}_2.$
\end{proof}

 \begin{figure}[h]
\centering
\includegraphics[scale=.7]{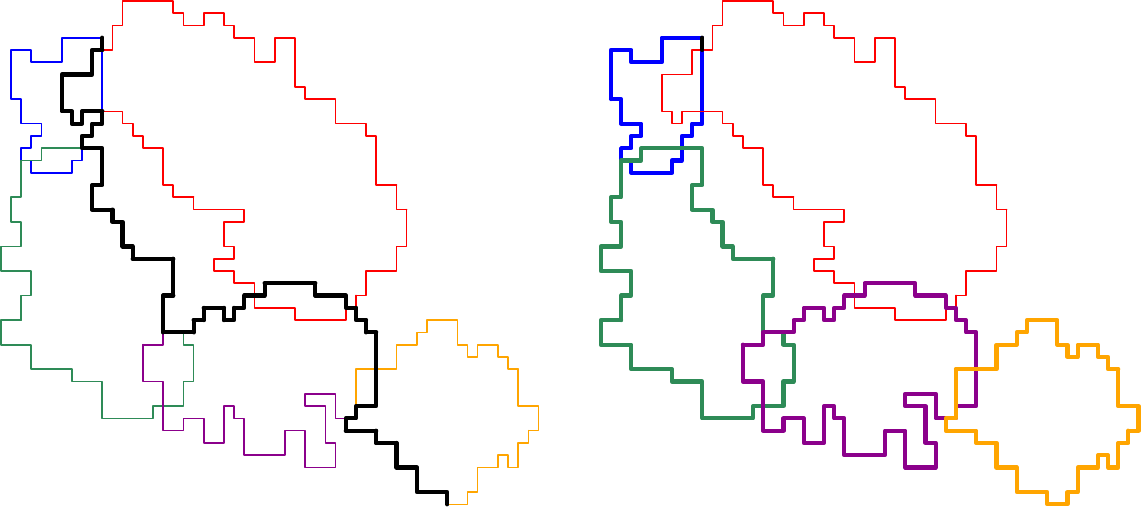}
\caption{Illustration of the content of \eqref{simple}. On the left is a path running through various loops, on the right, a subsequence of loops is boldened which form a simple chain.
}
\label{3.11}
\end{figure}

We conclude this section by providing the proof of  Lemma \ref{geometry6}.
Before proceeding with the proof, we introduce the notion of a ``simple chain''.
For $x,y\in \Z^d,$
we say that a sequence of glued loops $\overline{L}' = (\overline{\Gamma}_1',\overline{\Gamma}'_2,\cdots,\overline{\Gamma}'_m)$ is  a \emph{simple chain} from $x$ to $y$ if all the glued loops $\overline{\Gamma}_i$s are distinct and there exists a path $\ell$ from $x$ to $y$ such that $\overline{L}'$ is a glued loop sequence of $\ell.$
{Note that for any glued loop sequence $\overline{L}$ from $A$ to $B$ (recall Definition \ref{def1})},
\begin{align} \label{simple}
    \exists \text{a simple chain of glued loops, only using the glued loops in $\sett(\overline{L})$, from $A$ to $B$. }
\end{align}
{To see this note that one can simply consider a subsequence of $\overline L := (\overline{\Gamma}_1,\overline{\Gamma}_2,\cdots,\overline{\Gamma}_k)$ from $A$ to $B$ by iteratively making it shorter in the following manner (we assume that $\overline \Gamma_i \neq \overline \Gamma_{i+1}$ for every $i$). First one can assume that every glued loop of edge type appears at most once in the sequence $\overline{L}$. This is because otherwise one can drop them except one, but still maintaining the connection property from $A$ to $B$, since there must be some glued loops of  fundamental  or point type connecting these glued loops of edge type. Now we describe the algorithm. If $\overline \Gamma_1$ appears only once in $\overline{L}$, then consider $\overline \Gamma_2$. Otherwise take the last $j$ such that $\overline \Gamma_1=\overline \Gamma_j$ and consider the sequence $(\overline \Gamma_1,\overline \Gamma_{j+1},\ldots \overline \Gamma_k).$ This is still a glued loop sequence from $A$ to $B$ since $\overline \Gamma_1$ is \emph{not} of an edge type and thus is itself connected. Repeating this recursively, we establish \eqref{simple}.}

We are now in a position to finish the proof of Lemma \ref{geometry6}.

\begin{proof}[Proof of Lemma \ref{geometry6}]  
We prove the two items separately.

\noindent
\textbf{(1) Proof of 1.}
If $ \sett(\overline{L}_1) \cap \sett(\overline{L}_2) = \emptyset$, then {one can take both $\Gamma$ and $v_1$ to be any point in $A$ connected to $B_2$ (i.e. $\Gamma$ is of singleton type).}
 Now assume that $ \sett(\overline{L}_1) \cap \sett(\overline{L}_2) \neq  \emptyset$.  Let $\overline{\Gamma}$ be the last glued loop in $\overline{L}_2$ that belongs to $\textup{\textsf{Set}}(\overline{L}_1)$.  
Then define $v_1\in \Z^d$ to be the exit point of $\ell_2$ from $\overline{\Gamma}$. 
 
\textbf{Case 1.} $\overline{\Gamma}$ is either a fundamental or point type. 
Then any discrete loop $\Gamma$ in  {$\textsf{Dis}(\overline{\Gamma})$} along with the above point $v_1 \in \Z^d$  satisfy the desired properties.

    \textbf{Case 2.} $\overline{\Gamma}$ is an edge type.
    Then one can take $\Gamma:= \{v_1\}$ (i.e. singleton type).

    ~

\noindent
\textbf{(2) Proof of 2.}
Let $\overline{L}_1'$ be any \emph{simple} chain of glued loops, associated to $\overline{L}_1$ as defined above in \eqref{simple}.
{Set $u_2' \in \Z^d$ to be the exit point  of  $\ell_2$ from $\overline{\Gamma}'$ (note that we take $u_2'$ to be the initial point of $\ell_2$,  if $\ell_2$ does not use $\overline{\Gamma}'$).}
Then the connection $u_2' \ar B_2$, realized by a subsequence of a glued loop sequence $\overline{L}_2,$ does not use the glued loop $\overline{\Gamma}'$.
Let $\overline{L}_2'$ be its associated \emph{simple} chain.
Note that there may be several simple chains and we take any of them.

Assume that  $\textsf{Set}(\overline{L}_1')\cap \textsf{Set}(\overline{L}_2') \neq  \emptyset$. 
 Let $\overline{\Gamma}$ be the last glued loop in $\overline{L}_{2}'$  that belongs to $\textup{\textsf{Set}}(\overline{L}'_1)$. Define $w_1$ and $w_2$ to be the entry and exit points  of $\ell_1$ into and from  $\overline{\Gamma}$. Also let $w_3$ be the exit point of $\ell_2$ from $\overline{\Gamma}$. Then in the case when $\overline{\Gamma}$ is fundamental of point type, for any discrete loop $\Gamma$ in  {$\textsf{Dis}(\overline{\Gamma})$}, we have  
  $d^\extr(\Gamma,w_i) \le 1$ for $i=1,2,3$ and the connection property \eqref{300} holds. If $\overline{\Gamma}$ is of edge type, then similarly as before we can take $\Gamma$ of a singleton type.

In the simple case $\textsf{Set}(\overline{L}_1')\cap \textsf{Set}(\overline{L}_2') = \emptyset$, we have $A \overset{\textup{\textsf{Set}}(\overline{L_1}')}{\longleftrightarrow}  B_1 \circ u_2' \overset{\textup{\textsf{Set}}(\overline{L_2}') \setminus  \overline{\Gamma}' }{\longleftrightarrow}   B_2$, {whence one can take $\overline{\Gamma} := \overline{\Gamma}'$ and set $\Gamma$ to be any corresponding discrete loop. A point $w_1\in A$  with $d^\extr(w_1,\overline{\Gamma}')\le 1$ satisfies  the first connection in \eqref{300} and $w_2,w_3$ above enjoy the desired connections in \eqref{300}.}

\end{proof}

\subsection{Key estimates}
As outlined in Section \ref{iop}, it will be important to analyze the event of type $\{0\arr x, 0 \ar y\}$ later, in particular when controlling the volume of (intrinsic) balls. Recall that by Lemma \ref{geometry},  the  event  $\{0\arr x, 0 \ar y\}$ implies the existence of a $\ble$ such that the connections  of the type in \eqref{later} holds. By BKR inequality, the probability of such an event can be bounded in terms of the probabilities of the four disjoint events appearing in \eqref{later}. {However note that the two connections $u_1\ar u_2$ and $u_2\ar u_3$ in the last event may \emph{not} be disjoint}.
For the ease of reading we recall the expression again:
$$\{u_1\overset{\overline{\mathscr{S}} \setminus \{\overline{\Gamma}_1,\overline{\Gamma}_2,\overline{\Gamma}_3  \}}{\longleftrightarrow} u_2, u_2\overset{\overline{\mathscr{S}} \setminus \{\overline{\Gamma}_1,\overline{\Gamma}_3  \}}{\longleftrightarrow} u_3\}.$$
The following proposition provides a quantitative probability bound for such connections with the aid of Lemma \ref{geometry6} which was used to disjointify them.

  \begin{proposition}\label{key2} Let $d>20.$ Then
{there exists $C>0$ such that  
\begin{align} \label{310}
         \E \Big[  \sum_{\substack{
         (\{\Gamma_i\}_{i=1}^3,\{u_i\}_{i=1}^3) \in \textup{\ble} \\ d^\extr(\Gamma_2,0)\le 1 } }|\Gamma_1| |\Gamma_3|  \1_{u_1\overset{\overline{\mathcal{L}} \setminus \{\overline{\Gamma}_1,\overline{\Gamma}_2,\overline{\Gamma}_3  \}}{\longleftrightarrow} u_2, u_2\overset{\overline{\mathcal{L}} \setminus \{\overline{\Gamma}_1,\overline{\Gamma}_3  \}}{\longleftrightarrow} u_3}\1_{\Gamma_1,\Gamma_2,\Gamma_3  \in \mathcal{L}}
         \Big] \le C.
\end{align}}
In addition, there exists $C>0$ such that  for any $z_1,z_2\in \Z^d,$ 
       \begin{align} \label{311}
         \E \Big[\sum_{\substack{(\{\Gamma_i\}_{i=1}^3,\{u_i\}_{i=1}^3) \in \textup{\ble} \\ d^\extr(\Gamma_i, z_i) \le 1, \ i=1,2 } } |\Gamma_3|\1_{u_1\overset{\overline{\mathcal{L}} \setminus \{\overline{\Gamma}_1,\overline{\Gamma}_2,\overline{\Gamma}_3  \}}{\longleftrightarrow} u_2, u_2\overset{\overline{\mathcal{L}} \setminus \{\overline{\Gamma}_1,\overline{\Gamma}_3  \}}{\longleftrightarrow} u_3}\1_{\Gamma_1,\Gamma_2,\Gamma_3  \in \mathcal{L}} \Big] \le C|z_1-z_2|^{  {6-d}}.
      \end{align}

  \end{proposition} 

\begin{remark}
The estimate \eqref{311} will continue to hold even when one of the two length constraints in the definition of a $\ble$ is dropped, and indeed we will  only take advantage of the length constraint $|{\Gamma}_2|+2\ge |u_2-u_3|.$  This distinction from \eqref{310}, where both the length constraints play an important role, arises primarily because in the sum in \eqref{311}  two discrete loops $\Gamma_1$ and $\Gamma_2$ are pinned around two given points whereas in \eqref{310} only one discrete loop, namely $\Gamma_2,$ is stipulated to be pinned around the origin. 
\end{remark}

{Recall that for a discrete loop $\Gamma$ and $x\in \Z^d,$ we say that $\Gamma \sim x$ if $d^\extr(\Gamma,x)\le 1$.}

\begin{proof}[Proof of Proposition \ref{key2}]
We first prove the estimate \eqref{311}.
   By  the second item of Lemma \ref{geometry6} with {$A:= \{u_2\},  \ B_1:=\{u_1\}, \ B_2:=\{u_3\}$}  and 
   $\overline{\Gamma}'  := \overline{\Gamma}_2 $,  under the event $\{u_1\overset{\overline{\mathcal{L}} \setminus \{\overline{\Gamma}_1,\overline{\Gamma}_2,\overline{\Gamma}_3  \}}{\longleftrightarrow} u_2, u_2\overset{\overline{\mathcal{L}} \setminus \{\overline{\Gamma}_1,\overline{\Gamma}_3  \}}{\longleftrightarrow} u_3\}$, there exist a discrete loop ${\Gamma}\in\mathcal{L}$ and $w_1,w_2,w_3\in \Z^d$ such that $d^\extr({\Gamma},w_i) \le 1$, for $i=1,2,3$, and 
   \begin{align}\label{disconn432}
       u_1 \overset{\overline{\mathcal{L}} \setminus \{\overline{\Gamma},\overline{\Gamma}_1,\overline{\Gamma}_2,\overline{\Gamma}_3 \}}{\longleftrightarrow}  w_1 \circ u_2 \overset{\overline{\mathcal{L}} \setminus \{\overline{\Gamma},\overline{\Gamma}_1,\overline{\Gamma}_2,\overline{\Gamma}_3 \}}{\longleftrightarrow}  w_2 \circ u_3 \overset{\overline{\mathcal{L}} \setminus \{\overline{\Gamma},\overline{\Gamma}_1,\overline{\Gamma}_2,\overline{\Gamma}_3 \}}{\longleftrightarrow} w_3.
   \end{align}
By  a union bound and recalling the definition of $\ble$, the LHS of \eqref{311} is bounded by  \begin{align} \label{366}
 &\sum_{u_1,u_2,u_3\in \Z^d}  \sum_{w_1,w_2,w_3\in \Z^d} \sum_{\substack{{\Gamma} \sim w_1,w_2,w_3 }}   \sum_{\substack{  {\Gamma}_1\sim z_1,u_1}} \sum_{\substack{  {\Gamma}_2 \sim z_2,u_2\\|{\Gamma}_2|+2\ge |u_2-u_3|}} \sum_{\substack{  {\Gamma}_3\sim u_3}} \nonumber \\
 &\qquad |{\Gamma}_3|
 \P(u_1 \overset{\overline{\mathcal{L}} \setminus \{\overline{\Gamma},\overline{\Gamma}_1,\overline{\Gamma}_2,\overline{\Gamma}_3 \}}{\longleftrightarrow}  w_1 \circ u_2 \overset{\overline{\mathcal{L}} \setminus \{\overline{\Gamma},\overline{\Gamma}_1,\overline{\Gamma}_2,\overline{\Gamma}_3 \}}{\longleftrightarrow}  w_2 \circ u_3 \overset{\overline{\mathcal{L}} \setminus \{\overline{\Gamma},\overline{\Gamma}_1,\overline{\Gamma}_2,\overline{\Gamma}_3 \}}{\longleftrightarrow} w_3 ,\  {\Gamma}, {\Gamma}_1, {\Gamma}_2, {\Gamma}_3\in  {\mathcal{L}}) .
\end{align}
{Note that for given discrete-loops  $\Gamma_1,\Gamma_2,\Gamma_3, \Gamma$ and points $u_1,u_2,u_3, w_1,w_2,w_3,$
 the two events above, namely the one consisting of three disjoint connections and the other one being $\{\Gamma,\Gamma_1,\Gamma_2,\Gamma_3\in \mathcal{L}\}$, {are measurable with respect to disjoint collections of (continuous-time) loops} on $\cable$ and thus are decoupled by the thinning property of the Poisson point process.} This motivates the definition of the following quantity.
For $k \ge 1$ and  $u_1,u_2,u_3,w_1,w_2,w_3 \in \Z^d$, let  
\begin{align*} 
          T_{k}(u_1,u_2,u_3,w_1,w_2,w_3)
          := \sum_{\substack{{\Gamma} \sim w_1,w_2,w_3 }}   \sum_{\substack{  {\Gamma}_1\sim z_1,u_1}} \sum_{\substack{  {\Gamma}_2\sim z_2,u_2\\|{\Gamma}_2|\ge k-2}} \sum_{\substack{  {\Gamma}_3\sim u_3}}  |{\Gamma}_3| \P({\Gamma}, {\Gamma}_1, {\Gamma}_2, {\Gamma}_3\in  {\mathcal{L}}).
\end{align*} 
Then by the BKR inequality,  \eqref{366} is bounded by 
\begin{align}  \label{361}
& \sum_{u_1,u_2,u_3\in \Z^d} \sum_{w_1,w_2,w_3\in \Z^d}  \P(u_1\ar w_1) \P(u_2\ar w_2) \P(u_3\ar w_3) \cdot   T_{|u_2-u_3|}( u_1,u_2,u_3,w_1,w_2,w_3)  \nonumber \\
    &\le        C   \sum_{u_1,u_2,u_3\in \Z^d} \sum_{w_1,w_2,w_3\in \Z^d}  |u_1-w_1|^{2-d} |u_2-w_2|^{2-d} |u_3-w_3|^{2-d} \cdot   T_{|u_2-u_3|}( u_1,u_2,u_3,w_1,w_2,w_3)  .
\end{align}

Thus it remains to bound  $ T_{k }( u_1,u_2,u_3,w_1,w_2,w_3) .$ \\

{ Recall that while defining a $\ble$ (see Definition \ref{nearby}) we had commented that we will allow the loops to be identical. Nonetheless, throughout the  paper, for brevity,  we will only provide the details for the generic case, i.e., when the discrete loops $\Gamma_1,\Gamma_2,\Gamma_3$ and $\Gamma$ are all distinct. The arguments in the degenerate cases of identical loops will be simpler and will typically yield improved bounds. While we will refrain from providing all the details in those cases, we will occasionally add remarks pointing out  simplifications. For instance, a discussion to this effect appears in Remark \ref{deg}.}

By Lemmas \ref{key0} and  \ref{key6},  \begin{align}  \label{303}
          &T_{k }( u_1,u_2,u_3,w_1,w_2,w_3)
 \nonumber \\
          &= \Big[\sum_{\substack{{\Gamma} \sim w_1,w_2,w_3 }}    \P({\Gamma}\in  {\mathcal{L}})\Big]\Big[\sum_{\substack{  {\Gamma}_1\sim z_1,u_1 }}   \P({\Gamma}_1\in  {\mathcal{L}})\Big]\Big[
          \sum_{\substack{  {\Gamma}_2\sim z_2,u_2\\|{\Gamma}_2|\ge k-2}}  \P({\Gamma}_2\in  {\mathcal{L}}) \Big]\Big[
          \sum_{\substack{  {\Gamma}_3\sim u_3 }} |{\Gamma}_3| \P(\Gamma_3\in  {\mathcal{L}})\Big] \nonumber  \\
          &\le  C|w_1-w_3|^{2-d}|w_2-w_3|^{2-d} \cdot   
    {|u_1-z_1|^{4-2d}}   \cdot k^{1-d/2}|u_2-z_2|^{2-d},
\end{align} 
where we particularly used \eqref{25200}, \eqref{2510}, \eqref{251} and  \eqref{211} respectively (see the explanation in  Remark \ref{adjacent1} for the validity under the replacement of the condition $\Gamma \ni u_1,\cdots,u_k$ by $\Gamma \sim u_1,\cdots,u_k$).   {Note that when using the bound \eqref{25200} on the first factor term, we dropped the term $|w_1-w_2|^{2-d}$.} This makes the algebra significantly simpler but on the other had leads to a worse bound on the dimension $d$. {What one might gain by carrying this out optimally will be  elaborated  in Remark \ref{long} shortly.}

Proceeding, using \eqref{303}, we bound the quantity \eqref{361} by
\begin{align} \label{302}
& C  \sum_{u_1,u_2,u_3\in \Z^d} \sum_{w_1,w_2,w_3\in \Z^d}  |u_1-w_1|^{2-d} |u_2-w_2|^{2-d} |u_3-w_3|^{2-d} \cdot {|w_1-w_3|^{2-d}}|w_2-w_3|^{2-d} \nonumber   \\
&\qquad \qquad \qquad\qquad\qquad \cdot {|u_1-z_1|^{ 4-2d}}  \cdot |u_2-u_3|^{1-d/2} |u_2-z_2|^{2-d} .
\end{align}
Using {\eqref{boundbound}} recursively (by first taking the summation over  $w_1$ and $w_2$), we control \eqref{302} by
\begin{align} \label{338}
     &C\sum_{u_1,u_2,u_3,w_3\in \Z^d} |u_1 - w_3|^{4-d} |u_2 - w_3|^{4-d}|u_3-w_3|^{2-d} {|u_1-z_1|^{4-2d} } |u_2-u_3|^{1-d/2}   { |u_2-z_2|^{2-d}}  \nonumber\\    
     &\le  C\sum_{u_1,u_2,w_3\in \Z^d} |u_1 - w_3|^{4-d} |u_2 - w_3|^{4-d}\cdot |u_2-w_3|^{3-d/2}\cdot  {|u_1-z_1|^{4-2d} }   { |u_2-z_2|^{2-d}}  \nonumber\\  
     &= C\sum_{u_1,u_2,w_3\in \Z^d} |u_1 - w_3|^{4-d} |u_2 - w_3|^{7-3d/2}{|u_1-z_1|^{4-2d} }   { |u_2-z_2|^{2-d}}  \nonumber\\  
      &\le C\sum_{u_1,u_2\in \Z^d} |u_1 - u_2|^{4-d}  {|u_1-z_1|^{4-2d} }   { |u_2-z_2|^{2-d}}  \nonumber\\  
       &\le C\sum_{u_1\in \Z^d} |u_1 - z_2|^{6-d}  {|u_1-z_1|^{4-2d} } \le C|z_1 - z_2|^{6-d}
\end{align}
{(we used the fact $7-3d/2 < -d$ and $4-2d<-d$).} This concludes the proof of \eqref{311}.

~

For the proof of \eqref{310}, we proceed with the similar argument as above. 
Similarly as in \eqref{303}, for given points $v_1,v_2,v_3,w_1,w_2,w_3\in \Z^d$ and $k_2,k_3 \ge 1$, 
\begin{align}\label{disjconn452}
 & \sum_{\substack{{\Gamma} \sim w_1,w_2,w_3 }}   \sum_{\substack{  {\Gamma}_1\sim u_1}} \sum_{\substack{  {\Gamma}_2\sim 0,u_2\\|{\Gamma}_2|\ge k_2-2}} \sum_{\substack{  {\Gamma}_3\sim u_3\\|{\Gamma}_3|\ge k_3-2}}  |{\Gamma}_1||{\Gamma}_3| \P( {{\Gamma}, {\Gamma}_1, {\Gamma}_2, {\Gamma}_3\in  {\mathcal{L}}} ) \\
 \nonumber
    &= \Big[\sum_{\substack{{\Gamma} \sim w_1,w_2,w_3 }}    \P({\Gamma}\in  {\mathcal{L}})\Big]\Big[\sum_{\substack{  {\Gamma}_1\sim u_1 }}   |{\Gamma}_1| \P({\Gamma}_1\in  {\mathcal{L}})\Big]\Big[
          \sum_{\substack{  {\Gamma}_2\sim 0,u_2\\|{\Gamma}_2|\ge k_2-2}}  \P({\Gamma}_2\in  {\mathcal{L}}) \Big]\Big[
          \sum_{\substack{  {\Gamma}_3\sim u_3\\|{\Gamma}_3|\ge k_3-2}} |{\Gamma}_3| \P({\Gamma}_3\in  {\mathcal{L}})\Big]\\
          \nonumber
          &\le C{|w_1-w_2|^{2-d}}|w_2-w_3|^{2-d}  \cdot 
          k_2^{1/2-d/4} |u_2|^{3-3d/2} 
          \cdot  k_3^{2-d/2},
\end{align}
where we particularly used   \eqref{25200},  \eqref{211},  \eqref{25}  and \eqref{21200} respectively.
Thus, by the similar reasoning as above,  \eqref{310} is bounded by 
\begin{align*} 
    & C  \sum_{u_1,u_2,u_3\in \Z^d} \sum_{w_1,w_2,w_3\in \Z^d}  |u_1-w_1|^{2-d} |u_2-w_2|^{2-d} |u_3-w_3|^{2-d} \cdot {|w_1-w_2|^{2-d}}|w_2-w_3|^{2-d} \nonumber   \\
&\qquad \qquad \qquad\qquad\qquad \cdot  |u_2-u_3|^{1/2-d/4} |u_2|^{3-3d/2} |u_1-u_3|^{2-d/2} 
 \nonumber   \\
    &\le  C  \sum_{u_1,u_2,u_3,w_2\in \Z^d}  |u_1-w_2|^{4-d} |u_3-w_2|^{4-d} |u_2-w_2|^{2-d}  \nonumber   \\
    &\qquad \qquad \qquad\qquad\cdot  |u_2-u_3|^{1/2-d/4} |u_2|^{3-3d/2} |u_1-u_3|^{2-d/2}  \nonumber   \\
    &\le  C  \sum_{u_2,u_3,w_2\in \Z^d}  |u_3-w_2|^{6-d/2}\cdot |u_3-w_2|^{4-d} |u_2-w_2|^{2-d}   |u_2-u_3|^{1/2-d/4} |u_2|^{3-3d/2} \nonumber   \\
        &=  C  \sum_{u_2,u_3,w_2\in \Z^d}  |u_3-w_2|^{10-3d/2}  |u_2-w_2|^{2-d}   |u_2-u_3|^{1/2-d/4} |u_2|^{3-3d/2} \nonumber   \\
   &\le  C  \sum_{u_2,u_3\in \Z^d} |u_2-u_3|^{2-d} \cdot   |u_2-u_3|^{1/2-d/4} |u_2|^{3-3d/2} \le C  \sum_{u_2\in \Z^d}|u_3|^{5/2-5d/4}\le C
\end{align*}
(we used  $10-3d/2 < -d$ and $ 3-3d/2<5/2-5d/4<-d$, {which follows from the assumption $d>20$}).
Therefore we conclude the proof.

\end{proof}
A few remarks are in order.

\begin{remark} \label{long}
As is evident, there is a lot of room to tighten things. For instance in the inequality \eqref{303}, when using the bound \eqref{25200} without dropping the term $|w_1-w_3|^{2-d}$, we encounter the following types of terms 
  \begin{align*}
        \sum_{z\in \Z^d} |x_1-z|^{\alpha_1}|x_2-z|^{\alpha_2}|x_3-z|^{\alpha_3}, \qquad x_1,x_2,x_3\in \Z^d.
 \end{align*}
However, the above quantity depends on the relative locations of the three points $x_1,x_2,x_3$ in an unwieldy way, which makes the proof much more technical. Since, as already indicated in Section \ref{iop}, the conditions for the bounds on $\ble$s (for instance, in the arguments in the forthcoming section)  has no hope of going down to the optimal $d>6$ condition, in the interest of simplifying the expressions we pursue the above simpler but suboptimal approach.
\end{remark}

{
\begin{remark} \label{deg}
Finally, note that above we had only considered the generic case, i.e. where all the discrete loops $\Gamma_i$s and $\Gamma$ are  distinct in the above proof.  The degenerate cases can be handled similarly and are in fact simpler to analyze, and the details for those, are quite similar while not being exactly the same. In the interest of avoiding repetitions and containing the length, we omit the details. 
Nonetheless, to provide an illustration of the nature of the arguments needed to address those, let us consider the  particular case, when, in \eqref{disconn432}, the $\Gamma_i$s are distinct but $\Gamma = \Gamma_2$. In this case, one can replace the connection in \eqref{disconn432} by the statement that there exists $u_2' \sim \overline{\Gamma}_2$ such that 
$$u_1\overset{\overline{\mathcal{L}} \setminus \{\overline{\Gamma}_1,\overline{\Gamma}_2,\overline{\Gamma}_3  \}}{\longleftrightarrow} u_2 \circ u'_2\overset{\overline{\mathcal{L}} \setminus \{\overline{\Gamma}_1,\overline{\Gamma}_2,\overline{\Gamma}_3  \}}{\longleftrightarrow} u_3.$$ 

A direct application of the BKR inequality and following the arguments  in the above proof  allows us to bound the quantity \eqref{310} by 
   \begin{align*}
& \sum_{u_1,u_2,u_3,u_2'\in \Z^d}  \P(u_1 \ar u_2)  \P(u_2' \ar u_3)  \\
& \qquad \qquad     \cdot\Big[\sum_{\substack{  {\Gamma}_1\sim u_1 }} |{\Gamma}_1|   \P({\Gamma}_1\in  {\mathcal{L}})\Big]\Big[
          \sum_{\substack{  {\Gamma}_2\sim 0,u_2,u_2'\\|{\Gamma}_2|\ge |u_2-u_3| - 2}}  \P({\Gamma}_2\in  {\mathcal{L}}) \Big]\Big[
\sum_{\substack{  {\Gamma}_3\sim u_3\\|{\Gamma}_3|\ge |u_1-u_3| - 2}}   |{\Gamma}_3| \P(\Gamma_3\in  {\mathcal{L}})\Big].
        \end{align*}
Note that the above expression is similar to the one in \eqref{disjconn452} and can be bounded similarly, for instance, by disregarding the constraint $|{\Gamma}_2|\ge |u_2-u_3| - 2$.
\end{remark}
}

\section{Volume of intrinsic balls} \label{section 4}

In this section, we provide an upper  bound on the expectation of the volume of balls in the loop soup $\widetilde \cL$. {From now on, we denote by $d(\cdot,\cdot)$ the intrinsic (or chemical) distance induced by the loop soup $\widetilde \cL$. Also for $x\in \Z^d$ and $r\in \N$, let
   \begin{align*}
 B(x,r):= \{y\in \Z^d: d(x,y) \le r\},\quad \widetilde B(x,r):= \{y\in \cable: d(x,y) \le r\}
        \end{align*}
  be the intrinsic balls of center $x$ and radius $r$, where the former only contains lattice points and the latter contains points in $\cable$. Recall that $| \widetilde B(x,r)| =|  B(x,r)|  $ denotes the number of lattice points in $\widetilde B(x,r)$ or $ B(x,r)$. The following proposition is the main result of this section.}

\begin{proposition} \label{ball}
Let $d>20$. There exists $C>0$ such that for any $r\in \N$,
    \begin{align}
      \E |B (0,r)| \le Cr.
    \end{align}
\end{proposition}
{The condition $d>20$ will be assume throughout this section without being repeated further.}
\begin{remark}\label{lb1}
A corresponding lower bound, while not serving as an input for us is in fact a reasonably straightforward consequence of the two point function. The proof follows via establishing that there exists constants $c_1, c_2>0,$ such that for any $z\in \Z^d$,
\begin{align}\label{conditionaldis}
\P(d(0,z)\le c_1 |z|^2 \mid 0\ar z)&\ge c_2  \text{\,\,which implies,}\\
\label{conditionaldis1}
\P(0\overset{c_1|z|^2}{\lbij} z)&\gtrsim \frac{1}{|z|^{d-2}}. 
\end{align}
 Given this, a lower bound on $\E |B (0,r)|$ follows simply by summing \eqref{conditionaldis1}, over $z\in B^{\textup{ext}}(0,r^{1/2}).$
 The proof of \eqref{conditionaldis} follows from the following estimate and Markov's inequality.
\begin{align*}
\sum_{x\in \Z^d} \P(0\ar x\circ x\ar z)\approx |z|^2\frac{1}{|z|^{d-2}}.
\end{align*} 
 Note that the above in particular implies that $\E(d(0,z)\mid 0\lbij z)\lesssim |z|^2.$
 \end{remark}

To prove Proposition \ref{ball}, it will be convenient to define
for $r\in \N$,
\begin{align}
    G(r) := \E |B(0,r)|.
\end{align}
Note that we have the expression
\begin{align} \label{express}
    G(r) =  \sum_{x\in \Z^d} \P(0\arr x ) .
\end{align}

Inspired by \cite{kn2} we prove the following recursive bound  between $G(2r)$ and $G(r)$, which along with the trivial bound $G(r)\le Cr^d$ immediately implies Proposition \ref{ball} (see \cite[Theorem 1.2]{kn2} for the arguments).
\begin{proposition} \label{key prop}
There exists $c>0$ such that for any $r\in \N,$
    \begin{align} \label{332}
        G(2r) \ge c\frac{G(r)^2}{r}.
    \end{align}
\end{proposition}

The two key steps are the following: 
\begin{enumerate}
    \item ``Reverse'' BKR inequality. 
    Recall that the standard BKR inequality says that  for any $x,y\in \Z^d,$
\begin{align*}
    \P(0\arr x \circ x \overset{r}{\ar} y) \le \P(0\arr x ) \P( x \arr y).
\end{align*}
However the reverse inequality is not true in general. Our first step involves proving a lower bound of the expected number of pairs $(x,y)\in \Z^d \times \Z^d$    such that  $0\arr x \circ x \overset{r}{\ar} y$ is  at least
\begin{align*}
\sum_{x,y\in \Z^d}  \P(0\arr x \circ x \overset{r}{\ar} y) &\ge     c_1 \sum_{x,y\in \Z^d} \P(0\arr x ) \P( x \arr y)\\& =    c_1 \sum_{x\in \Z^d} \P(0\arr x ) \sum_{y\in \Z^d} \P( x \arr y) \\
    & =    c_1  G(r) \sum_{x\in \Z^d} \P(0\arr x )  =c_1G(r)^2,
\end{align*} 
where $c_1>0$ is some constant. Note that above we are relying on translation invariance of the model.    
    
    \item The final step addresses the overcounting in the previous step. Note that the occurrence of $0\arr x \circ x \overset{r}{\ar} y$ implies in particularly $d(0,y)\le 2r.$ However,  in this step we establish the closeness of $x$ to a geodesic from $0$ to $y$. 
This implies that for any $y\in B(0,2r)$, the number of such $x$ is $O(r).$ Thus the expected number of pairs $(x,y)\in \Z^d \times \Z^d$ satisfying $0\arr x \circ x \overset{r}{\ar} y$  is at most $c_2rG(2r)$ for some constant $c_2>0.$
\end{enumerate}
This yields the recursion $$rG(2r) \ge c_1c_2^{-1} G(r)^2,$$ yielding Proposition \ref{key prop}.

~
We now proceed to implementing the above two steps followed by the proof of 
Proposition \ref{key prop}.

\subsection{Reverse BKR inequality} 
\begin{lemma} \label{lemma 3.4}
There exists $c>0$ such that  for any $r\in \N,$
    \begin{align*}
        \sum_{x,y\in \Z^d} \P (0\arr x \circ x \overset{r}{\ar} y  )  \ge cG(r)^2.
    \end{align*}
\end{lemma}

We start with some notations.
 For $\tilde A\subseteq \cable,$ define $\tilde A^{\textup{int}}$ to be a subgraph of $(\Z^d,E(\Z^d))$ obtained by deleting all partial edges in $\tilde A,$ i.e. edges $e$ such that $\tilde A\cap [e]\neq [e].$ Thus $\tilde A^{\textup{int}}$ is the subgraph of $\tilde A$ graph formed by these fully covered edges and the vertex set induced by these edges. }
In addition, analogous to \eqref{abb12} and \eqref{abb34}, {for $x,y\in \Z^d$ and $r\in \N$,} we say that $x\overset{\tilde A}{\longleftrightarrow}y $ (resp. $x\overset{\tilde A,r}{\longleftrightarrow}y $) {if there exists a lattice path $\ell$ from $x$ to $y$, only using full edges contained in $\tilde A$ (resp. if further $\ell$ has a length at most $r$).} We say 
    $x \overset{\tilde A}{\nleftrightarrow}   y$ otherwise.
{Then, as all the fully covered  edges contained in $\tilde A$  are also contained in $\tilde A^{\textup{int}}$,
    \begin{align} \label{int}
         x\overset{\tilde A,r}{\longleftrightarrow}y \Leftrightarrow  x\overset{\tilde A^{\textup{int}},r}{\longleftrightarrow}y
    \end{align} 
    and     \begin{align}
         x\overset{\tilde A}{\nleftrightarrow} 
 y  \Leftrightarrow  x\overset{\tilde A^{\textup{int}}}{\nleftrightarrow} y.
    \end{align}  }

{Also, for convenience, we introduce a slightly changed notion of  (continuous-time)  loops on $\cable$. Recall that there are three types of loops 
  on $\cable$,  i.e. fundamental, point and edge types. For technical reasons, it will be convenient throughout the paper to not consider loops of the edge type individually but in their glued form. Recall that for any edge, we have been referring to the union of all the loops whose ranges are subset of the open interval associated to $e$ as the glued loop $\overline{\gamma}_e$.\\

 \textit{Thus, from now on, for us any collection of loops, in particular the loop soup $\widetilde \cL$, will be interpreted as a collection of  fundamental loops, point loops and glued edge loops $\overline{\gamma}_e$ indexed by some subset of edges $e$ in $E(\Z^d)$ (for $\widetilde \cL$, the subset of edges is the entirety of $E(\Z^d)$).}}

~

    Now, we provide the proof of Lemma \ref{lemma 3.4}.
\begin{proof}[Proof of Lemma \ref{lemma 3.4}]

The proof consists of five main steps.  \\

\noindent
 \textbf{Step 1. Reduction using local resampling.}
By translation invariance, it suffices to prove the existence of $c>0$ such that 
        \begin{align} \label{goal}
        \sum_{x,y\in \Z^d} \P (0\arr x \circ 0 \overset{r}{\ar} y  )  \ge cG(r)^2.
    \end{align}
It will be convenient instead to work with    $\P (u \overset{r}{\longleftrightarrow} x  \circ  v \overset{r}{\longleftrightarrow} y )$ where $u$ and $v$ are at some distance from each other (this will allow us to take advantage of the triangle condition in \eqref{triangle condition}).  We will take $u=(-K,0,\cdots,0)$ and $v=(K,0,\cdots,0)$ for some large but fixed $K.$
We now claim that for any such positive integer $K$, there exists  a constant $c(K)>0$ such that for any $x,y\in \Z^d$,
\begin{align} \label{336}
    \P (0\arr x \circ 0 \overset{r}{\ar} y  ) \ge  c(K) \P (u \overset{r-K}{\longleftrightarrow} x \circ v \overset{r-K}{\longleftrightarrow} y).
\end{align}
This follows from a quick resampling argument aided by the FKG inequality which we now describe. Unlike edge percolation, since we are aiming to achieve disjointness not in terms of edges but in terms of glued loops, we proceed by showing the existence of a glued loop that is not used by either of the connections $u \overset{r-K}{\longleftrightarrow} x$ or $ v \overset{r-K}{\longleftrightarrow} y$  and connects $u$ and $v.$
Towards this let $H_0 \subset \Z^d$ be the straight line
\begin{align*}
   H_0&:=\{(k,0,\cdots,0,0) : k = -K ,-K+1\cdots, K\} \subseteq \Z^d
   \end{align*}
   and let $H_i:= H_{i-1} \cup \{(-K,0,\cdots,0,i)\}$ for $i=1,2,3$.
Then, recalling that {$\overline \gamma_{A}$} denotes the glued loop indexed by $A\subset \Z^d,$ define the event
\begin{align*}
    \mathcal{G}:=\bigcap_{i=0}^3\big\{\overline{\gamma}_{H_i}\in \bar \cL\big\}.
    \end{align*}
Since both $\{u \overset{r-K}{\longleftrightarrow} x \circ v \overset{r-K}{\longleftrightarrow} y\}$ and $\mathcal{G}$ are increasing events, by the FKG inequality,
\begin{align} \label{777}
     \P (u \overset{r-K}{\longleftrightarrow} x \circ v \overset{r-K}{\longleftrightarrow} y,  \  \mathcal{G}) \ge   \P (u \overset{r-K}{\longleftrightarrow} x \circ v \overset{r-K}{\longleftrightarrow} y  )  \P(\mathcal{G}).
\end{align}
Next we observe the implication
\begin{align} \label{7777}
\{u \overset{r-K}{\longleftrightarrow} x \circ v \overset{r-K}{\longleftrightarrow} y\} \cap  \mathcal{G} \Rightarrow  0\arr x \circ 0 \overset{r}{\ar} y  .
   \end{align}
To see this, note that if the connection $u \overset{r-K}{\longleftrightarrow} x $ uses at least two glued loops $\overline  \gamma_{H_j}$ and $\overline \gamma_{H_{j'}}$ ($0\le j<j'\le 3$), then as $H_j \subseteq H_{j'}$, we may assume that it uses  only one glued loop $ \overline  \gamma_{H_{j'}}$ among $\overline  \gamma_{H_i}$s  ($i=0,1,2,3$). This is also the case for the connection $v \overset{r-K}{\longleftrightarrow} y$.
 Hence there exist two distinct glued loops $\overline  \gamma_{H_{j_1}}$ and $\overline \gamma_{H_{j_2}}$, which are  \emph{not} used in any of the connections $u \overset{r-K}{\longleftrightarrow} x$ and  $v \overset{r-K}{\longleftrightarrow} y$. As
 \begin{align*}
 \{u \overset{r-K}{\longleftrightarrow} x\}\cap \{ {\overline  \gamma_{H_{j_1}}}\in \overline\cL \}     \Rightarrow 0\arr x  \  \text{ and }  \  \{v \overset{r-K}{\longleftrightarrow} y\} \cap  \{ {\overline  \gamma_{H_{j_2}}}\in \overline\cL \}\Rightarrow 0\arr y,
        \end{align*}
we obtain \eqref{7777}.
Thus by \eqref{777} and \eqref{7777}, along with the fact that
\begin{align*}
     \P(\mathcal{G})= \P( {\overline  \gamma_{H_{i}}} \in \overline \cL \text{ for } i=0,1,2,3)  \ge  c(K),
     \end{align*}
 \eqref{336} is shown to hold.

Therefore, {noting that  $G(r) \le CK^d G(r-K)$} (which follows from the fact that any element in $B(0,r)$ is within Euclidean/extrinsic distance $K$ from $B(0,r-K)$), in order to prove \eqref{goal}, it reduces to showing  that for sufficiently large $K$, for any $u,v\in \Z^d$ with $|u-v| \ge K$ and $r>K,$
            \begin{align} \label{330}
        \sum_{x,y\in \Z^d} \P (u \overset{r-K}{\longleftrightarrow} x \circ v \overset{r-K}{\longleftrightarrow} y  )  \ge \frac{1}{2} G(r-K)^2.
    \end{align}
    \textbf{Step 2. Division into disjoint components.}  Instead of lower bounding $\P(u \overset{r-K}{\longleftrightarrow} x \circ v \overset{r-K}{\longleftrightarrow} y),$   
we will in fact bound the probability that  $u \overset{r-K}{\longleftrightarrow} x$ and $v \overset{r-K}{\longleftrightarrow} y$ and further that \emph{$u$ and $v$ are not connected to each other.}  Towards implementing this, {for $z\in \Z^d,$} let $\tilde{\mathcal{C}}(z)$ be the connected component  in  {$\widetilde \cL,$ containing $z$.}     

{A subgraph $S$ of $(\Z^d,E(\Z^d))$  is called \emph{$z$-admissible} if $\P(\tilde{\mathcal{C}}^{\textup{int}}(z)=S) \neq 0$ (recall that   $\tilde{\mathcal{C}}^{\textup{int}}(z)$ consists only of full edges of $\tilde{\mathcal{C}}(z)$). Note that any  $z$-admissible set is finite due to Proposition \ref{cluster upper}.}  Then, {for any $u,v,x,y\in \Z^d$,}
\begin{align} \label{331}
     \P (u \overset{r-K}{\longleftrightarrow} x \circ v \overset{r-K}{\longleftrightarrow} y  ) &\ge  \P (u \overset{r-K}{\longleftrightarrow} x, v \overset{r-K}{\longleftrightarrow} y, v\notin \tilde{\mathcal{C}}^{\text{int}}(u)) \nonumber \\     &= \sum_{\substack {S: \text{ $u$-admissible}\\ {u \overset{S,r-K}{\longleftrightarrow} x},
  u \overset{S}{\nleftrightarrow}  v}}     \P (v \overset{r-K}{\longleftrightarrow} y \mid \tilde{\mathcal{C}}^{\text{int}}(u)=S) \P(\tilde{\mathcal{C}}^{\text{int}}(u)=S),
\end{align}
{where the above inequality follows since   if $v\notin \tilde{\mathcal{C}}^{\text{int}}(u)$ then   $v\notin  \tilde{\mathcal{C}}(u)$ (because $v\in \Z^d$), which implies  $\tilde{\mathcal{C}}(u) \cap \tilde{\mathcal{C}}(v) = \emptyset$.}

{For any connected subgraph $S$ of $(\Z^d,E(\Z^d))$, let $S^{\text{out}}$ be the subgraph of $(\Z^d,E(\Z^d))$,  obtained by adding all (full) edges in $E(\Z^d)$ intersecting $S$. Then, define  $\textsf{loop}(S)$ to  be the collection of all loops  on $\cable$ (not necessarily in the loop soup $\widetilde \cL$) intersecting $S^{\text{out}}$, where every (full) edge in $S^{\text{out}}$ is considered as a closed interval and {$S^{\text{out}}$ is accordingly regarded as a subset of $\cable$.}} 
We claim that for any $u$-admissible  $S$ such that $u \overset{S}{\nleftrightarrow}   v$ (equivalently $v\notin S$),
{\begin{align} \label{888}
    \P (v \overset{r-K}{\longleftrightarrow} y \mid   \tilde{\mathcal{C}}^{\text{int}}(u)=S)  \ge   \P (v \overset{\widetilde \cL\setminus \textsf{loop}(S),  \ r-K}{\longleftrightarrow} y).
\end{align}}
{This is because for any $\tilde A\subseteq \cable$ such that  $\tilde A^{\text{int}}=S$,
\begin{align*}
   \P (v \overset{r-K}{\longleftrightarrow} y \mid   C(u) = \tilde A)  =   \P (v \overset{\widetilde \cL\setminus \{ \text{loops intersecting } \tilde A \},  \ r-K}{\longleftrightarrow} y) \ge   \P (v \overset{\widetilde \cL\setminus \textsf{loop}(S),  \ r-K}{\longleftrightarrow} y),
    \end{align*}
{where the first identity follows from the thinning property of the Poisson point process and the last inequality follows from the fact that any}  loop intersecting $\tilde A$ is an element of $\textsf{loop}(S)$ (recall that  $\tilde A^{\text{int}}=S$).}

Applying \eqref{888} to \eqref{331},
\begin{align}  \label{335}
        \P (u \overset{r-K}{\longleftrightarrow} x \circ v \overset{r-K}{\longleftrightarrow} y  ) &\ge  
        \sum_{\substack { S: \text{ $u$-admissible} \\ u \overset{S,r-K}{\longleftrightarrow} x,  \ 
 u \overset{S}{\nleftrightarrow}  v}} 
 \P (v \overset{\widetilde \cL\setminus \textsf{loop}(S),  \ r-K}{\longleftrightarrow} y) \P(\tilde{\mathcal{C}}^{\text{int}}(u)=S) \nonumber  \\
 &=\sum_{\substack {S: \text{ $u$-admissible}\\ u \overset{S,r-K}{\longleftrightarrow} x } }   \P (v \overset{\widetilde \cL\setminus \textsf{loop}(S),  \ r-K}{\longleftrightarrow} y) \P(\tilde{\mathcal{C}}^{\text{int}}(u)=S) ,
\end{align}
where the last identity follows from the fact that {$v \overset{ \widetilde \cL\setminus \textsf{loop}(S),  \ r-K}{\longleftrightarrow} y$ cannot occur if $u \overset{S}{\ar} v$.}

Define {$\{v \overset{r-K}{\longleftrightarrow} y$}  only on $\textsf{loop}(S)\}$ to be the event that there exists a path of  length is at most $r-k$, connecting $v$ and $y$, and any such path must use some loop in $\textsf{loop}(S)$.  Then, one can write 
\begin{align*}
\P (v \overset{ \widetilde \cL\setminus  \textsf{loop}(S),  \ r-K}{\longleftrightarrow} y) = \P(  v \overset{r-K}{\longleftrightarrow} y  )  -  \P (v \overset{r-K}{\longleftrightarrow} y \text{ only on }  \textsf{loop}(S)) .
\end{align*}
Note that by \eqref{int},
\begin{align*}
     \P (u \overset{r-K}{\longleftrightarrow} x) =  \sum_{\substack {S: \text{ $u$-admissible}\\ u \overset{S,r-K}{\longleftrightarrow} x } }   \P(\tilde{\mathcal{C}}^{\text{int}}(u)=S).
\end{align*}
Thus applying the above two identities to \eqref{335}, 
\begin{align} \label{333}
     \P (u \overset{r-K}{\longleftrightarrow} x &\circ v \overset{r-K}{\longleftrightarrow} y  ) \nonumber \\
     & \ge  \P (u \overset{r-K}{\longleftrightarrow} x ) \P(  v \overset{r-K}{\longleftrightarrow} y  )  - \sum_{\substack {S: \text{ $u$-admissible}\\ u \overset{S,r-K}{\longleftrightarrow} x } }   \P (v \overset{r-K}{\longleftrightarrow} y \text{ only on } \textsf{loop}(S))   \P(\tilde{\mathcal{C}}^{\text{int}}(u)=S).
\end{align}

\noindent
\textbf{Step 3. Control on the probability $ \P (v \overset{r-K}{\longleftrightarrow} y \text{ only on } \textsf{loop}(S))   $.} 
{In order to bound the second term above, observe that the event $\{v \overset{r-K}{\longleftrightarrow} y \text{ only on } \textsf{loop}(S)\}$ implies the existence of a path from $v$ to $y$, whose length is at most $r-K$, which uses a loop in $\textsf{loop}(S)$, i.e. there  exist a discrete loop} $\Gamma \in \mathcal{L}$, intersecting $S^{\text{out}}$,  and the points $w_1,w_2 \in \Z^d$  with {$\Gamma \sim w_i$ (recall that this notation denotes $d^\extr({\Gamma},w_i) \le 1$)} such that  
     \begin{align} \label{7272}
     w_1 \overset{\overline{\mathcal{L}} \setminus \{\overline{\Gamma}\},r-K}{\longleftrightarrow}  y  , \  w_2  \overset{\overline{\mathcal{L}} \setminus \{\overline{\Gamma}\}}{\longleftrightarrow} v    .
     \end{align}
\begin{figure}[h]
\centering
\includegraphics[scale=.7]{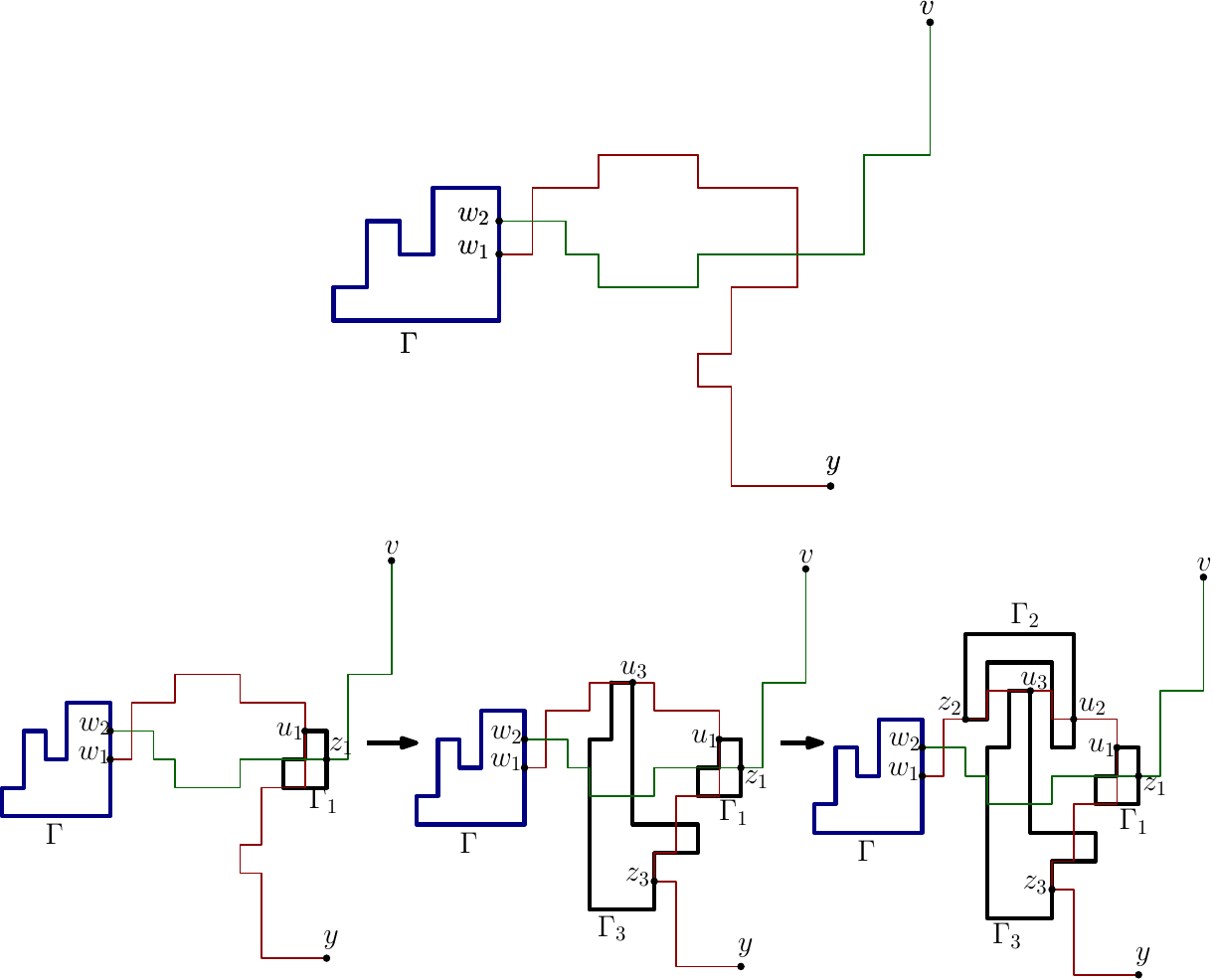}
\caption{Proof sketch of Lemma \ref{lemma 3.4} where given $w_1 \overset{\overline{\mathcal{L}} \setminus \{\overline{\Gamma}\},r-K}{\longleftrightarrow}  y  $ and $ w_2  \overset{\overline{\mathcal{L}} \setminus \{\overline{\Gamma}\}}{\longleftrightarrow} v,$ a $\ble$ is extracted to obtain disjoint arms.  
}
\label{4.1}
\end{figure}   
Note that although the second connection  $w_2 \ar v$  also  has length at most $r-K$, we drop this condition since the length constraint  only in the first connection will be enough for our purposes. 
By Lemma \ref{geometry}
with $A := \{w_1,w_2\}$ and $\overline{S}:= \overline{\mathcal{L}} \setminus \{\overline{\Gamma}\}$,
there exists a $\ble$  tuple 
 $(\{\Gamma_i\}_{i=1}^3,\{u_i\}_{i=1}^3)$ with $\Gamma_1,\Gamma_2,\Gamma_3\in \mathcal{L}$  and $z_1,z_2,z_3\in \Z^d$ with
 ${\Gamma}_i\sim z_i$ ($i=1,2,3$) such that
\begin{align}
 &z_1 \overset{\overline{\mathcal{L}} \setminus \{\overline{\Gamma},\overline{\Gamma}_1,\overline{\Gamma}_2,\overline{\Gamma}_3  \}}{\longleftrightarrow}  v \circ z_2 \overset{\overline{\mathcal{L}} \setminus \{\overline{\Gamma},\overline{\Gamma}_1,\overline{\Gamma}_2,\overline{\Gamma}_3  \}}{\longleftrightarrow} \{w_1,w_2\} \circ z_{3} \overset{\overline{\mathcal{L}} \setminus \{\overline{\Gamma},\overline{\Gamma}_1,\overline{\Gamma}_2,\overline{\Gamma}_3  \},r-K}{\longleftrightarrow}  y  \nonumber \\
 &\qquad \qquad  \circ \{u_1\overset{\overline{\mathcal{L}} \setminus \{\overline{\Gamma},\overline{\Gamma}_1,\overline{\Gamma}_2,\overline{\Gamma}_3  \}}{\longleftrightarrow} u_2, u_2 \overset{\overline{\mathcal{L}} \setminus \{\overline{\Gamma},\overline{\Gamma}_1,\overline{\Gamma}_3  \} }{\longleftrightarrow} u_3\}.
 \end{align}
{Note that above we  drop the distance constraint (i.e. the length is at most $r-K$) for the connection $ z_2 \overset{\overline{\mathcal{L}} \setminus \{\overline{\Gamma},\overline{\Gamma}_1,\overline{\Gamma}_2,\overline{\Gamma}_3  \}}{\longleftrightarrow} \{w_1,w_2\} $  since  just a single distance constraint will be sufficient for our argument.}
Thus, by a union bound and the BKR inequality, for any connected subgraph $S$ of $(\Z^d,E(\Z^d))$,
\begin{align*}
         \P &(v \overset{r-K}{\longleftrightarrow} y \text{ only on } \textsf{loop}(S))\\
         &\le \sum_{\Gamma \in  \textsf{loop}(S)} \sum_{\substack{w\in \Z^d \\ \Gamma \sim w  } }   \sum_{z_1,z_2,z_3\in \Z^d} 
\sum_{\substack{(\{\Gamma_i\}_{i=1}^3,\{u_i\}_{i=1}^3)\in \ble \\  {\Gamma}_i\sim z_i, \ i=1,2,3} }\P(  z_1\ar v ) \P( z_2\ar w) \P(z_{3} \overset{r-K}{\longleftrightarrow}  y )\\
 &\qquad\qquad \qquad\qquad \qquad \qquad \qquad \cdot \P(u_1\overset{ \overline{\mathcal{L}} \setminus \{\overline{\Gamma},\overline{\Gamma}_1,\overline{\Gamma}_2,\overline{\Gamma}_3  \} }{\longleftrightarrow}  u_2, u_2 \overset{\overline{\mathcal{L}} \setminus \{\overline{\Gamma},\overline{\Gamma}_1,\overline{\Gamma}_3  \}}{\longleftrightarrow} u_3, \ \Gamma,\Gamma_1,\Gamma_2,\Gamma_3 \in \mathcal{L}).
    \end{align*}
     {As already indicated before, the degenerate cases can be dealt with in a simpler way. For instance if two connections in  \eqref{7272}  are already disjoint,  corresponding to the special case $\Gamma_1=\Gamma_2=\Gamma_3=\Gamma$, then  one can just decouple these connections by the BKR inequality.}

This allows to upper bound the second term in  \eqref{333}. Indeed,  by interchanging the summation over the subgraphs  $S$ and discrete loops $\Gamma$,
    \begin{align} \label{334}
        &\sum_{\substack {S: \text{ $u$-admissible}\\ u \overset{S,r-K}{\longleftrightarrow} x } }  \P (v \overset{r-K}{\longleftrightarrow} y \text{ only on } \textsf{loop}(S))  \P(\tilde{\mathcal{C}}^{\text{int}}(u)=S)\nonumber   \\
        &\le   \sum_{\Gamma} \sum_{\substack{w\in \Z^d \\ \Gamma \sim w  } }  \sum_{z_1,z_2,z_3\in \Z^d}  
\sum_{\substack{(\{\Gamma_i\}_{i=1}^3,\{u_i\}_{i=1}^3)\in \ble \\  {\Gamma}_i\sim z_i, \ i=1,2,3} } \P(  z_1\ar v ) \P( z_2\ar w) \P(  z_{3} \overset{r-K}{\longleftrightarrow}  y )  \nonumber \\
& \qquad\cdot \P(u_1\overset{ \overline{\mathcal{L}} \setminus \{\overline{\Gamma},\overline{\Gamma}_1,\overline{\Gamma}_2,\overline{\Gamma}_3  \}  }{\longleftrightarrow}  u_2, u_2 \overset{\overline{\mathcal{L}} \setminus \{\overline{\Gamma},\overline{\Gamma}_1,\overline{\Gamma}_3  \}  }{\longleftrightarrow} u_3, \ \Gamma,\Gamma_1,\Gamma_2,\Gamma_3 \in \mathcal{L})  \sum_{\substack {S: \text{ $u$-admissible}\\ u \overset{S,r-K}{\longleftrightarrow} x \\
        \textsf{loop}(S)\ni \Gamma } }    \P(\tilde{\mathcal{C}}^{\text{int}}(u)=S).
    \end{align}
We claim that there exists a constant $\textsf{C}_1>0$ such that for any $u,x\in \Z^d$ and a discrete loop  $\Gamma$, the last summation over $S$ is bounded by
\begin{align} \label{local2}
    \P(\exists a\in \Z^d \text{ with $\Gamma\sim a$ such that } u \overset{r-K}{\longleftrightarrow} x, u\ar a) \le \textsf{C}_1 |\Gamma| \P(u \overset{r-K}{\longleftrightarrow} x, u\ar \Gamma)  .
\end{align}
While, again, a tighter bound without the extra $|\Gamma|$ factor is not  hard to establish, we will be content with the above which is simpler to argue. 
By a union bound, it suffices to show that there exists a constant $\textsf{C}_0>0$ such that for any $u,x,a\in \Z^d$ and a discrete loop  $\Gamma$ with $\Gamma\sim a$,
\begin{align} \label{local}
     \P(u \overset{r-K}{\longleftrightarrow} x, u\ar a) \le \textsf{C}_0  \P(u \overset{r-K}{\longleftrightarrow} x, u\ar \Gamma)  .
\end{align}
To see this, observe that by FKG inequality,
\begin{align*}
     \P(u \overset{r-K}{\longleftrightarrow} x, u\ar a, a\ar \Gamma)  \ge  \P(u \overset{r-K}{\longleftrightarrow} x, u\ar a) \P(a \ar \Gamma).
\end{align*}
Since $\P(a \ar \Gamma) \ge \textsf{C}_0$ for some constant $\textsf{C}_0>0$ {(uniformly in $\Gamma$ and $a \sim \Gamma$)} and  $\{u\ar a, a\ar \Gamma\}$ implies $u\ar \Gamma$, we obtain \eqref{local}.

\noindent
\textbf{Step 4. Bounding $ \P(u \overset{r-K}{\longleftrightarrow} x, u\ar \Gamma) .$} 
\begin{figure}[h]
\centering
\includegraphics[scale=.8]{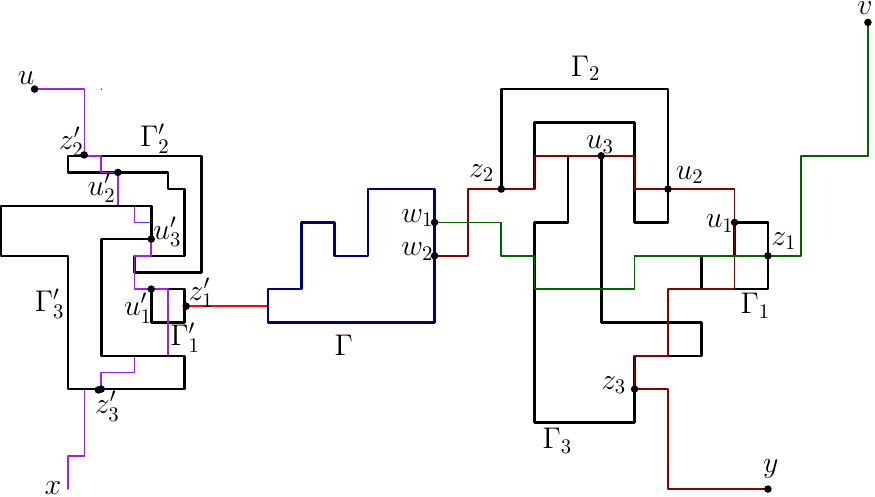}
\caption{Illustration of  Step 4 which considers $u \overset{r-K}{\longleftrightarrow} x, u\ar \Gamma$ in conjunction with Figure \ref{4.1}.}
\label{4.10}
\end{figure}
We apply our $\ble$ extraction technology again. By  Lemma  \ref{geometry}, for any discrete loop $\Gamma$, the event $\{u \overset{r-K}{\longleftrightarrow} x, u\ar \Gamma\}$ implies the existence of a tuple $(\{\Gamma_{i}'\}_{i=1}^3,\{u_{i}'\}_{i=1}^3)\in \ble $  with  $\Gamma'_1,\Gamma'_2,\Gamma'_3\in \mathcal{L}$  and $z_1',z'_2,z'_3\in \Z^d$ with
 ${\Gamma}'_i\sim z_i'$ ($i=1,2,3$) such that  
\begin{align*}
z_1' \overset{\overline{\mathcal{L}} \setminus \{ \overline{\Gamma}'_1,\overline{\Gamma}'_2,\overline{\Gamma}'_3  \}}{\longleftrightarrow} {\Gamma} \circ z_2'  \overset{\overline{\mathcal{L}} \setminus \{ \overline{\Gamma}'_1,\overline{\Gamma}'_2,\overline{\Gamma}'_3  \} }{\longleftrightarrow} u \circ   z_{3}'  \overset{\overline{\mathcal{L}} \setminus \{ \overline{\Gamma}'_1,\overline{\Gamma}'_2,\overline{\Gamma}'_3  \},r-K}{\longleftrightarrow}   x \circ \{ u_1' \overset{\overline{\mathcal{L}} \setminus \{ \overline{\Gamma}'_1,\overline{\Gamma}'_2,\overline{\Gamma}'_3  \}}{\longleftrightarrow} u_2' , u_2' \overset{\overline{\mathcal{L}} \setminus \{ \overline{\Gamma}'_1,\overline{\Gamma}'_3  \}}{\longleftrightarrow} u_3'\}.
 \end{align*}
Thus by a union bound and the BKR inequality,
\begin{align*}
    \P(u \overset{r-K}{\longleftrightarrow} x, u\ar \Gamma) 
       &\le   \sum_{\substack{w'\in \Gamma  } } \sum_{z_1',z_2',z_3'\in \Z^d} 
\sum_{\substack{(\{\Gamma'_i\}_{i=1}^3,\{u'_i\}_{i=1}^3)\in \ble \\  {\Gamma}'_i\sim z_i', \ i=1,2,3} }   \P(z_1' \ar w') \P( z'_2\ar u) \P(z_{3}' \overset{r-K}{\longleftrightarrow}  x )  \\
            &\qquad \qquad\qquad\cdot  \P( u_1' \overset{\overline{\mathcal{L}} \setminus \{ \overline{\Gamma}'_1,\overline{\Gamma}'_2,\overline{\Gamma}'_3  \} }{\longleftrightarrow} u_2' , u_2' \overset{\overline{\mathcal{L}} \setminus \{ \overline{\Gamma}'_1,\overline{\Gamma}'_3  \}}{\longleftrightarrow} u_3', 
 \ \Gamma'_1,\Gamma'_2,\Gamma'_3 \in \mathcal{L}).
       \end{align*}
Plugging this together with \eqref{local2} into \eqref{334} and then taking the summation over $x,y\in \Z^d,$
\begin{align} \label{433}
     &\sum_{x,y\in \Z^d}\sum_{\substack {S: \text{ $u$-admissible}\\ u \overset{S,r-K}{\longleftrightarrow} x } }  \P (v \overset{r-K}{\longleftrightarrow} y \text{ only on } \textsf{loop}(S))  \P(\tilde{\mathcal{C}}^{\text{int}}(u)=S) \nonumber \\
     &\le \textsf{C}_1 \sum_{x,y\in \Z^d}   \sum_{\Gamma} \sum_{\substack{w\in \Z^d \\ \Gamma \sim w  } } \sum_{z_1,z_2,z_3\in \Z^d}   \sum_{\substack{(\{\Gamma_i\}_{i=1}^3,\{u_i\}_{i=1}^3)\in \ble \\  {\Gamma}_i\sim z_i, \ i=1,2,3} }  \sum_{\substack{w'\in \Gamma  } }   \sum_{z_1',z_2',z_3'\in \Z^d} 
\sum_{\substack{(\{\Gamma'_i\}_{i=1}^3,\{u'_i\}_{i=1}^3)\in \ble \\  {\Gamma}'_i\sim z_i', \ i=1,2,3} }   \nonumber  \\
         &\qquad\qquad  \Big[ |\Gamma| \P(  z_1\ar v ) \P( z_2\ar w )\P( z_{3} \overset{r-K}{\longleftrightarrow}  y ) \P(z_1' \ar w' ) \P( z'_2\ar u) \P(z_{3}' \overset{r-K}{\longleftrightarrow}  x )  \nonumber  \\
          &\qquad\qquad \cdot \P(u_1\overset{\overline{\mathcal{L}} \setminus \{ \overline{\Gamma},\overline{\Gamma}_1,\overline{\Gamma}_2,\overline{\Gamma}_3  \}}{\longleftrightarrow}  u_2, u_2 \overset{\overline{\mathcal{L}} \setminus \{\overline{\Gamma}, \overline{\Gamma}_1, \overline{\Gamma}_3  \}}{\longleftrightarrow} u_3, \ \Gamma,\Gamma_1,\Gamma_2,\Gamma_3 \in \mathcal{L}) 
          \nonumber  \\
          &\qquad\qquad \cdot 
          \P( u_1' \overset{\overline{\mathcal{L}} \setminus \{ \overline{\Gamma}'_1,\overline{\Gamma}'_2,\overline{\Gamma}'_3  \}}{\longleftrightarrow} u_2' , u_2' \overset{\overline{\mathcal{L}} \setminus \{ \overline{\Gamma}'_1,\overline{\Gamma}'_3  \}}{\longleftrightarrow} u_3', 
 \ \Gamma'_1,\Gamma'_2,\Gamma'_3 \in \mathcal{L}) \Big]\nonumber \\
          &\le CG(r-K)^2\sum_{z_1,z_2,z'_1,z'_2\in \Z^d} \sum_{\Gamma} \sum_{\substack{w\in \Z^d \\ \Gamma \sim w  } }  \sum_{\substack{w'\in \Gamma  } }   |\Gamma|
 \P( z_2\ar w)  \P(z_1' \ar w') \P(\Gamma \in \mathcal{L})  \nonumber  \\
         &\qquad\qquad \cdot \P(  z_1\ar v )  \P( z'_2\ar u)  \sum_{\substack{(\{\Gamma_i\}_{i=1}^3,\{u_i\}_{i=1}^3)\in \ble \\  {\Gamma}_i\sim z_i, \ i=1,2 } }    \sum_{\substack{(\{\Gamma'_i\}_{i=1}^3,\{u'_i\}_{i=1}^3)\in \ble \\  \Gamma'_i \sim z'_i , \ i=1,2 } }    \Big[| \Gamma_{3}|  | \Gamma_{3}' |\nonumber \\
          &\qquad\qquad\qquad\qquad  \qquad\qquad \cdot \P(u_1\overset{\overline{\mathcal{L}} \setminus \{ \overline{\Gamma}_1,\overline{\Gamma}_2,\overline{\Gamma}_3  \}}{\longleftrightarrow}  u_2, u_2 \overset{\overline{\mathcal{L}} \setminus \{ \overline{\Gamma}_1,\overline{\Gamma}_3  \}}{\longleftrightarrow} u_3, \ \Gamma_1,\Gamma_2,\Gamma_3 \in \mathcal{L}) 
          \nonumber  \\
          &\qquad\qquad \qquad\qquad \qquad\qquad \cdot 
          \P( u_1' \overset{\overline{\mathcal{L}} \setminus \{ \overline{\Gamma}'_1,\overline{\Gamma}'_2,\overline{\Gamma}'_3  \}}{\longleftrightarrow} u_2' , u_2' \overset{\overline{\mathcal{L}} \setminus \{ \overline{\Gamma}'_1,\overline{\Gamma}'_3  \}}{\longleftrightarrow} u_3', 
 \ \Gamma'_1,\Gamma'_2,\Gamma'_3 \in \mathcal{L}) \Big].
\end{align}
Using Lemma \ref{lemma2.9} {(to bound the summation  over $\Gamma$, $w\in \Z^d \text{ with }  \Gamma\sim w$, and $w'\in \Gamma$)} and  Proposition \ref{key2}, the above summation (momentarily ignoring the pre-factor $G(r-K)^2$) is at most \begin{align} \label{9999}
    &C\sum_{z_1,z_2,z'_1,z'_2\in \Z^d}  |z'_1-z_2|^{4-d}   \P(  z_1\ar v ) \P( z'_2\ar u)  \nonumber  \\&\quad \quad \quad 
 \cdot \Big[ \sum_{\substack{(\{\Gamma_i\}_{i=1}^3,\{u_i\}_{i=1}^3)\in \ble \\  {\Gamma}_i\sim z_i , \ i=1,2 } }  | \Gamma_{3}|  \P(u_1\overset{\overline{\mathcal{L}} \setminus \{ \overline{\Gamma}_1,\overline{\Gamma}_2,\overline{\Gamma}_3  \}}{\longleftrightarrow}  u_2, u_2 \overset{\overline{\mathcal{L}} \setminus \{ \overline{\Gamma}_1,\overline{\Gamma}_3  \}}{\longleftrightarrow} u_3, \ \Gamma_1,\Gamma_2,\Gamma_3 \in \mathcal{L})  \Big] \nonumber \\
&\quad \quad \quad \cdot \Big[ \sum_{\substack{(\{\Gamma'_i\}_{i=1}^3,\{u'_i\}_{i=1}^3)\in \ble \\  \Gamma'_i \sim z'_i , \ i=1,2 } }   | \Gamma_{3}' |  \P( u_1' \overset{\overline{\mathcal{L}} \setminus \{ \overline{\Gamma}'_1,\overline{\Gamma}'_2,\overline{\Gamma}'_3  \}}{\longleftrightarrow} u_2' , u_2' \overset{\overline{\mathcal{L}} \setminus \{ \overline{\Gamma}'_1,\overline{\Gamma}'_3  \}}{\longleftrightarrow} u_3', 
 \ \Gamma'_1,\Gamma'_2,\Gamma'_3 \in \mathcal{L}) \Big] \nonumber \\
  &\le C\sum_{z_1,z_2,z'_1,z'_2\in \Z^d} |v-z_1|^{2-d} |u-z'_2|^{2-d} |z'_1-z_2|^{4-d}|z_1-z_2|^{ {6-d}}  |z'_1-z_2'|^{ {6-d}} \nonumber   \\
   &\le C\sum_{z_1,z'_2\in \Z^d}   |v-z_1|^{2-d} |u-z'_2|^{2-d}   |z_1-z'_2|^{{2(6-d)+4+d}} ,
\end{align}
{where in the last inequality we used the fact that for $\alpha_1,\alpha_2,\alpha_3<0$ satisfying certain conditions which are straightforward but tedious to verify in this case,
   \begin{align*} 
         \sum_{w_1,w_2\in \Z^d} |p-w_1|^{\alpha_1} |q-w_2|^{\alpha_2}  |w_1-w_2|^{\alpha_3} \le  C|p-q|^{\alpha_1+\alpha_2+\alpha_3+2d}
    \end{align*}
(The formal result appears later in the appendix as Lemma \ref{appendix1}).} By  Lemma \ref{appendix1} again (in particular \eqref{7200}) along with the condition $d>20$, if $|u-v| \ge K$ with sufficiently large $K$, then the  quantity \eqref{9999} is bounded by $ \frac{1}{2}.$

~

\noindent
\textbf{Step 5. Conclusion.}
Applying the above conclusion to \eqref{333}, for any $u,v\in \Z^d$ with sufficiently large $|u-v|$,
\begin{align*}
    \sum_{x,y\in \Z^d}  \P (u \overset{r-K}{\longleftrightarrow} x \circ v \overset{r-K}{\longleftrightarrow} y  ) &\ge    \sum_{x,y\in \Z^d}  \P (u \overset{r-K}{\longleftrightarrow} x )\P(v \overset{r-K}{\longleftrightarrow} y  )  - \frac{1}{2}G(r-K)^2 \\
    &= G(r-K)^2 -  \frac{1}{2}G(r-K)^2  = \frac{1}{2}G(r-K)^2,
\end{align*}
yielding \eqref{330} which finishes proof.

\end{proof}

\begin{remark}
We remark that the assumption $d>20$ is essentially required in this lemma. Indeed in order to deduce that the quantity \eqref{9999} is arbitrarily small for $|u-v|$ sufficiently large, we need the condition $(2-d) + (2-d) + (16-d) = 20-3d < -2d$ {(i.e. the sum of exponents in   \eqref{9999} is less than $-2d$, see  the hypothesis of Lemma \ref{appendix1}),} which is equivalent to $d>20$. 

It might be worth reiterating that we expect there to be some room to sharpen our results by using more accurate loop estimates, although any such attempt will nonetheless fall short of achieving the expected $d>6$ bound. \end{remark}

\subsection{Closeness of intermediate points to geodesics}
In this section,  we establish (2), the second part of our argument stating that fixing a geodesic $\gamma$ between $0$ and $y,$ on average, a point $x$ such that  $0 \arr x \circ x \arr y$ holds, must necessarily be {close to $\gamma$} (Lemma \ref{lemma 3.5} records the precise statement). The rough intuition behind the plausibility of this statement is the following. Let $\ell$  be any path from $0$ to $y$. Then there exists a triangle-like structure consisting of open edges, one of whose  edges is a part of $\ell$, from which there are two emanating arms (of  length at most $r$) to $0$ and $y$ respectively (see {Figure \ref{4.20}}). Also these five connections (three from a triangle and other two from emanating arms) are all disjoint, which is then unlikely by the triangle condition \eqref{triangle condition} unless $x$ is close enough. 
\begin{figure}[h]
\centering
\includegraphics[scale=1]{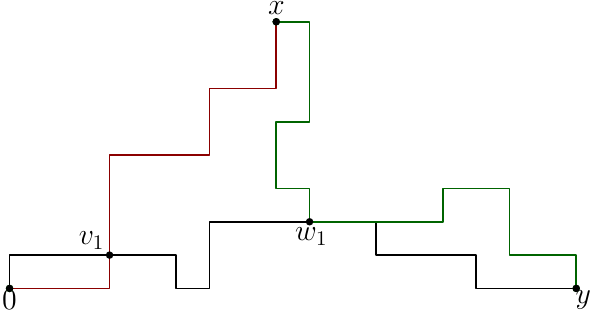}
\caption{An illustration of the situation where the point $x$ admitting $0 \arr x \circ x \arr y$ and also far away from the geodesic from $0$ to $y$ leads to a triangle formed by $v_1, x,w_1$ and disjoint arms certifying $0\arr v_1$ and $w_1 \arr y.$ While this argument works in the bond case, the $\ite$ involving $\ble$s to make a similar argument in our setting is illustrated below in Figure \ref{4.2}.
}
\label{4.20}
\end{figure}
The formal proof relies on our $\ite$ technique extracting two $\ble$s connected together and also both connected to $x$, from which there are two emanating arms (of size at most $r$) to $0$ and $y$ respectively. 
Thus our first order of business is to establish the following geometric lemma which reveals the $\ble$ structure under disjoint intrinsic connections.

\begin{lemma} \label{geometry3}
    Let $x_1,x_2,x_3\in \Z^d$, $r\in \N$ and  $\ell$   be any {simple path  (i.e. the lattice points that $\ell$ passes through are all distinct)} from $x_1$ to $x_3$.
Under the event $x_1\arr x_2 \circ x_2 \overset{r}{\ar} x_3 $, there exist  $(\{\Gamma_{1i}\}_{i=1}^3, \{u_{1i}\}_{i=1}^3),(\{\Gamma_{2i}\}_{i=1}^3, \{u_{2i}\}_{i=1}^3) \in \textup{\ble}$  with $\Gamma_{11},\Gamma_{12},\Gamma_{13},\Gamma_{21},\Gamma_{22},\Gamma_{23}\in \mathcal{L}$  together with the points $v_1,v_2,v_3,w_1,w_2,w_3\in \Z^d$ such that
\begin{align} \label{442}
  v_1,w_1\in \ell,\qquad \Gamma_{1i} \sim v_i,\quad \Gamma_{2i} \sim w_i,\quad \text{ for } i=1,2,3, 
\end{align}
and the following  seven connections occur disjointly: 
\begin{enumerate}
    \item $x_2 \overset{\overline{\mathcal{L}} \setminus \{\overline{\Gamma}_{11},\overline{\Gamma}_{12},\overline{\Gamma}_{13},\overline{\Gamma}_{21},\overline{\Gamma}_{22},\overline{\Gamma}_{23}  \}}{\longleftrightarrow}   v_2 $,
    \item  $x_1  \overset{\overline{\mathcal{L}} \setminus \{\overline{\Gamma}_{11},\overline{\Gamma}_{12},\overline{\Gamma}_{13},\overline{\Gamma}_{21},\overline{\Gamma}_{22},\overline{\Gamma}_{23}  \} ,r}{\longleftrightarrow}   v_3$,
    \item $ x_2  \overset{\overline{\mathcal{L}} \setminus \{\overline{\Gamma}_{11},\overline{\Gamma}_{12},\overline{\Gamma}_{13},\overline{\Gamma}_{21},\overline{\Gamma}_{22},\overline{\Gamma}_{23}  \} }{\longleftrightarrow}  w_2 $,
    \item $x_3  \overset{\overline{\mathcal{L}} \setminus \{\overline{\Gamma}_{11},\overline{\Gamma}_{12},\overline{\Gamma}_{13},\overline{\Gamma}_{21},\overline{\Gamma}_{22},\overline{\Gamma}_{23}  \},r}{\longleftrightarrow}  
w_3 ,$
    \item $v_1  \overset{\overline{\mathcal{L}} \setminus\{\overline{\Gamma}_{11},\overline{\Gamma}_{12},\overline{\Gamma}_{13},\overline{\Gamma}_{21},\overline{\Gamma}_{22},\overline{\Gamma}_{23}  \}}{\longleftrightarrow}   w_1 $,
    \item $u_{11}  \overset{\overline{\mathcal{L}} \setminus\{\overline{\Gamma}_{11},\overline{\Gamma}_{12},\overline{\Gamma}_{13},\overline{\Gamma}_{21},\overline{\Gamma}_{22},\overline{\Gamma}_{23}  \}}{\longleftrightarrow}   u_{12},  \ u_{12}  \overset{\overline{\mathcal{L}} \setminus\{\overline{\Gamma}_{11},\overline{\Gamma}_{13},\overline{\Gamma}_{21},\overline{\Gamma}_{22},\overline{\Gamma}_{23}  \}}{\longleftrightarrow}   u_{13}$,
\item $u_{21}  \overset{\overline{\mathcal{L}} \setminus\{\overline{\Gamma}_{11},\overline{\Gamma}_{12},\overline{\Gamma}_{13},\overline{\Gamma}_{21},\overline{\Gamma}_{22},\overline{\Gamma}_{23}  \}}{\longleftrightarrow}   u_{22},  \ u_{22}  \overset{\overline{\mathcal{L}} \setminus\{\overline{\Gamma}_{11},\overline{\Gamma}_{12},\overline{\Gamma}_{13},\overline{\Gamma}_{21},\overline{\Gamma}_{23}  \}}{\longleftrightarrow}  u_{23}$.
\end{enumerate} 

\end{lemma}
\begin{figure}[h]
\centering
\includegraphics[scale=0.9]{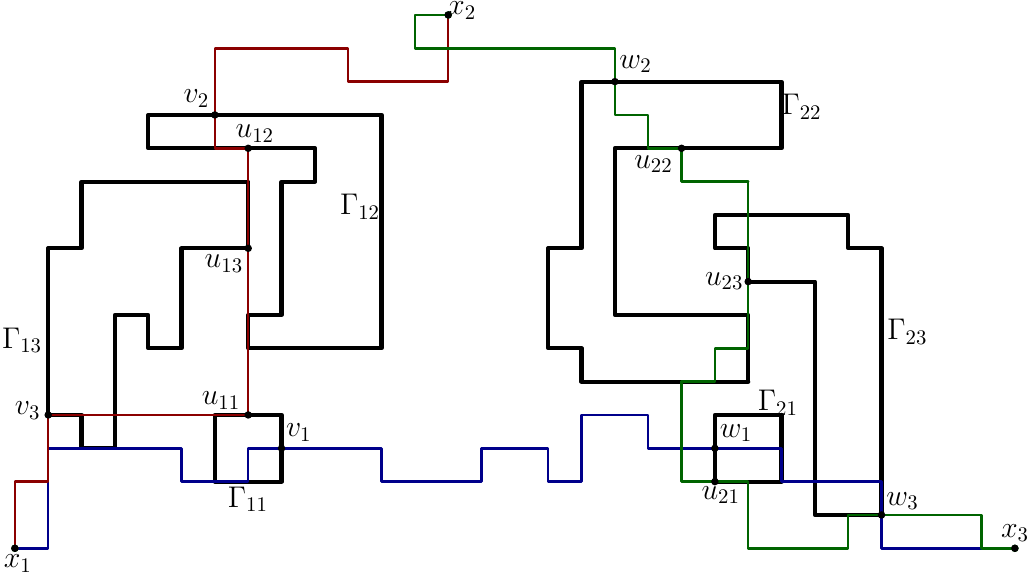}
\caption{The two $\ble$s extracted to prove Lemma \ref{geometry3}.
}
\label{4.2}
\end{figure}
{Figure \ref{4.2} illustrates this.}
\begin{proof}  
Under the event $x_1\arr x_2 \circ x_2 \overset{r}{\ar} x_3 $,
there exist two collection of glued loops $\overline{S}_1$ and $\overline{S}_2$ such that  
\begin{enumerate}
    \item There exists a path of length at most $r$, from $x_2$ to $x_1$, only using  glued loops in $\overline{S}_1$.
    \item There exists a path of length at most $r$, from $x_2$ to $x_3$, only using glued loops in $\overline{S}_2$.
    \item $\overline{S}_1\cap \overline{S}_2  = \emptyset.$
\end{enumerate}
{By Lemma \ref{geometry4}, there exist  {local loop-geodesics} $\overline{L}_1$ (from $x_2$ to $x_1$)  and $\overline{L}_2$  (from $x_2$ to $x_3$) such that $\textsf{Set}(\overline{L}_1) \subseteq \overline{S}_1$ and  $\textsf{Set}(\overline{L}_2) \subseteq \overline{S}_2.$}  As   $\overline{S}_1\cap \overline{S}_2  = \emptyset,$ we have $\textsf{Set}(\overline{L}_1)\cap \textsf{Set}(\overline{L}_2)  = \emptyset.$

Let $\overline{L} = (\overline{\Gamma}_1,\overline{\Gamma}_2,\cdots,\overline{\Gamma}_m)$ be any glued loop sequence of $\ell$. 
Setting the indices
\begin{align}   
    a_1&:= \max \{ 1\le k\le m : \overline{\Gamma}_k \in \textsf{Set}(\overline{L}_1) \}, \label{401}\\
    a_2&:= \min \{ a_1< k\le m : \overline{\Gamma}_k \in \textsf{Set}(\overline{L}_2) \} \label{4011},
\end{align}
define $v_1$ to be the exit point of $\ell$ from $\overline{\Gamma}_{a_1}$. Next, let $w_1$ be the entry point of $\ell|_{v_1 \rightarrow x_3}$ (i.e. {a}  sub-path of $\ell$ from $v_1$ to $x_3$) into $\overline{\Gamma}_{a_2}.$
We set $a_1:= 0$ and $v_1 := x_1$ if the set in \eqref{401} is empty, and set $w_1 := x_3$ if the set in \eqref{4011} is empty.   

{Let $\Gamma_{11}$ and $\Gamma_{21}$  be any {discrete loops} corresponding in the sets  $\textsf{Dis}(\overline{\Gamma}_{a_1})$ and  $\textsf{Dis}(\overline{\Gamma}_{a_2})$ respectively (see  after \eqref{trace} for the definition). Note that if $\overline{\Gamma}_{a_1}$ or $\overline{\Gamma}_{a_2}$ is a glued loop of edge type, then one can take any adjacent lattice point as a candidate for $\Gamma_{11}$ or $\Gamma_{21}$, which is an element in $\mathcal{L}$ of singleton type.  Then   by our construction of $v_1$ and $w_1$, we have $ \Gamma_{11} \sim v_1$ and $\Gamma_{21}\sim w_1$.}
In addition as $\overline{L}_1$ and $\overline{L}_2$ are local loop-geodesics, by Lemma \ref{geometry2},  there exist $$(\{\Gamma_{1i}\}_{i=1}^3, \{u_{1i}\}_{i=1}^3),(\{\Gamma_{2i}\}_{i=1}^3, \{u_{2i}\}_{i=1}^3) \in \ble$$ (with all {discrete loops} in $\mathcal{L}$) together with  $v_2,v_3,w_2,w_3 \in \Z^d$ such that $\Gamma_{1i}\sim v_i$,  $\Gamma_{1i}\sim w_i$   ($i=2,3$) and
    \begin{align} \label{403}
        &x_2 \overset{\textsf{Set}(\overline{L}_1) \setminus \{\overline{\Gamma}_{11},\overline{\Gamma}_{12},\overline{\Gamma}_{13}  \},r}{\longleftrightarrow}  v_2 \circ x_1 \overset{\textsf{Set}(\overline{L}_1) \setminus \{\overline{\Gamma}_{11},\overline{\Gamma}_{12},\overline{\Gamma}_{13}  \},r}{\longleftrightarrow}  v_3  \nonumber \\
        &\circ   \{u_{11}  \overset{\textsf{Set}(\overline{L}_1) \setminus \{\overline{\Gamma}_{11},\overline{\Gamma}_{12},\overline{\Gamma}_{13}  \}}{\longleftrightarrow}  
  u_{12}, u_{12} \overset{\textsf{Set}(\overline{L}_1) \setminus \{\overline{\Gamma}_{11},\overline{\Gamma}_{13}  \}}{\longleftrightarrow}   u_{13}\}
    \end{align}
as well as
    \begin{align} \label{404}
        &x_2 \overset{\textsf{Set}(\overline{L}_2) \setminus \{\overline{\Gamma}_{21},\overline{\Gamma}_{22},\overline{\Gamma}_{23}  \},r}{\longleftrightarrow}  w_2 \circ x_3 \overset{\textsf{Set}(\overline{L}_2) \setminus \{\overline{\Gamma}_{21},\overline{\Gamma}_{22},\overline{\Gamma}_{23}  \},r}{\longleftrightarrow}  
w_3
  \nonumber \\
  &\circ   \{u_{21} \overset{\textsf{Set}(\overline{L}_2) \setminus \{\overline{\Gamma}_{21},\overline{\Gamma}_{22},\overline{\Gamma}_{23}  \}}{\longleftrightarrow}   u_{22}, u_{22} \overset{\textsf{Set}(\overline{L}_2) \setminus \{\overline{\Gamma}_{21},\overline{\Gamma}_{23}  \}}{\longleftrightarrow}    u_{23}\}.
    \end{align} 
{Note that \eqref{403} (resp.   \eqref{404}) implies the connections (1), (2) and (6) (resp. (3), (4) and  (7)) in the statement of  the lemma. Thus
it remains to establish the connection (5).} This follows since,
by our construction of $v_1$ and $w_1$, the glued loop sequence from $v_1$ to $ w_1$, as a subsequence of $\overline{L}$, {does \emph{not} use any glued loop from} $\textsf{Set}(\overline{L}_1)$ and $\textsf{Set}(\overline{L}_2)$ {(otherwise  it contradicts \eqref{401} or  \eqref{4011}).} Hence, recalling 
$\textsf{Set}(\overline{L}_1)\cap \textsf{Set}(\overline{L}_2)  = \emptyset$, all the seven connections in this lemma happen disjointly and thus we conclude the proof. Note that we drop the distance condition in the connections (1) and (3) since  just a connection property will be  enough for our purpose.
\end{proof}

Relying on the above, the following lemma claims that  most $x, y\in \Z^d$ satisfying $0\arr x \circ x \overset{r}{\ar} y$  are close to  $\ell^{(0,y)}$ where
{for any $a,b \in \Z^d,$ we use $\ell^{(a,b)}$  to denote a geodesic from $a$ to $b$ in $\widetilde{\cL}$}. Note that there may be several geodesics and we take any of them. 
\begin{lemma} \label{lemma 3.5}
Let  $\e>0$ be any constant. Then, there exists a sufficiently large constant $K>0$ such that for any $r\ge 1,$ 
        \begin{align} \label{444}
        \sum_{x,y\in \Z^d} \P (0\arr x \circ x \overset{r}{\ar} y,  \  {d^\extr}(x,\ell^{(0,y)})  \ge K) \le   \e G(r)^2. 
        \end{align}
\end{lemma}
\begin{proof}

For each $x,y\in \Z^d,$ by  Lemma \ref{geometry3} with
\begin{align*}
    x_1:=0,\quad x_2:=x,\quad x_3:=y,\quad \ell:=\ell^{(0,y)},
\end{align*}the event
\begin{align*}
    \{ 0\arr x   \circ  x \overset{r}{\ar} y \} \cap \{{d^\extr}(x,\ell^{(0,y)}) \ge K\} 
\end{align*}
implies the existence of  $(\{\Gamma_{1i}\}_{i=1}^3, \{u_{1i}\}_{i=1}^3),(\{\Gamma_{2i}\}_{i=1}^3, \{u_{2i}\}_{i=1}^3) \in \ble$ (with all {discrete loops} in $\mathcal{L}$)  
 together with points  $v_1,v_2,v_3,w_1,w_2,w_3\in \Z^d$ which satisfy the properties in Lemma \ref{geometry3} and
 \begin{align*}
    |x-v_1| \ge K,\qquad |x-w_1| \ge K .
\end{align*}
{Note that these inequalities follow from the condition ${d^\extr}(x,\ell^{(0,y)}) \ge K$ along with the fact that $v_1$ and $w_1$ lie on $\ell^{(0,y)}$.}
{Hence by a union bound over all pairs of $\ble$s and the points $v_1,v_2,v_3,w_1,w_2,w_3\in \Z^d$  along with the BKR inequality,
    \begin{align} \label{342}
& \P (0\arr x \circ x \overset{r}{\ar} y,  \  d^\extr(x,\ell^{(0,y)})  \ge K)  \nonumber \\  
&\le   \sum_{ \substack{v_1,w_1\in \Z^d \\ \min \{|x-v_1|,|x-w_1|\} \ge   K} }
 \sum_{v_2,v_3,w_2,w_3\in \Z^d} \sum_{\substack{ (\{\Gamma_{1i}\}_{i=1}^3,\{u_{1i}\}_{i=1}^3)\in \ble \\ \Gamma_{1i} \sim v_i,\ i=1,2,3 }}  \sum_{\substack{ (\{\Gamma_{2i}\}_{i=1}^3,\{u_{2i}\}_{i=1}^3)\in \ble \\ \Gamma_{2i} \sim w_i,\ i=1,2,3 }}    \nonumber  \\
       &\qquad \qquad \cdot  \Big[ \P (x \ar v_2 ) \P( 0  \arr v_3)   \P( x  \ar  w_2) \P( y  \arr w_3)  \P( v_1 \ar w_1)    \nonumber  \\
&\qquad \qquad \cdot\P(u_{11}  \overset{\overline{\mathcal{L}} \setminus \{\overline{\Gamma}_{11},\overline{\Gamma}_{12},\overline{\Gamma}_{13} \}}{\longleftrightarrow}  u_{12}, u_{12}  \overset{\overline{\mathcal{L}} \setminus \{\overline{\Gamma}_{11},\overline{\Gamma}_{13} \}}{\longleftrightarrow}  u_{13},\  {\Gamma}_{11}, {\Gamma}_{12},{\Gamma}_{13} \in \mathcal{L} ) \Big]  \nonumber \\
&\qquad \qquad \cdot\P(u_{21}  \overset{\overline{\mathcal{L}} \setminus \{\overline{\Gamma}_{21},\overline{\Gamma}_{22},\overline{\Gamma}_{23} \}}{\longleftrightarrow}  u_{22}, u_{22}  \overset{\overline{\mathcal{L}} \setminus \{\overline{\Gamma}_{21},\overline{\Gamma}_{23} \}}{\longleftrightarrow}  u_{23},\  {\Gamma}_{21}, {\Gamma}_{22},\Gamma_{23} \in \mathcal{L} ) \Big]  \nonumber \\
        &=  \sum_{ \substack{v_1',w_1'\in \Z^d \\ \min \{|v_1'|,|w_1'|\} \ge   K} }
 \sum_{v_2',v'_3,w'_2,w'_3\in \Z^d} \sum_{\substack{ (\{\Gamma'_{1i}\}_{i=1}^3,\{u'_{1i}\}_{i=1}^3)\in \ble \\ \Gamma'_{1i} \sim v'_i,\ i=1,2,3 }}  \sum_{\substack{ (\{\Gamma'_{2i}\}_{i=1}^3,\{u'_{2i}\}_{i=1}^3)\in \ble \\\Gamma'_{2i} \sim w'_i,\ i=1,2,3 }}  \nonumber  \\
       & \qquad \qquad  \Big[ \P (0 \ar v_2' ) \P( -x  \arr v_3' )   \P(0\ar  w_2') \P( y-x \arr w_3')  \P( v_1' \ar w_1')  \nonumber  \\    
       &\qquad \qquad \cdot\P(u'_{11}  \overset{\overline{\mathcal{L}} \setminus \{\overline{\Gamma}'_{11},\overline{\Gamma}'_{12},\overline{\Gamma}'_{13} \}}{\longleftrightarrow}  u'_{12}, u'_{12}  \overset{\overline{\mathcal{L}} \setminus \{\overline{\Gamma}'_{11},\overline{\Gamma}'_{13} \}}{\longleftrightarrow}  u'_{13},\ {\Gamma}'_{11},{\Gamma}'_{12},{\Gamma}'_{13} \in \mathcal{L} ) \Big]  \nonumber \\
&\qquad \qquad \cdot\P(u'_{21}  \overset{\overline{\mathcal{L}} \setminus \{\overline{\Gamma}'_{21},\overline{\Gamma}'_{22},\overline{\Gamma}'_{23} \}}{\longleftrightarrow}  u'_{22}, u'_{22}  \overset{\overline{\mathcal{L}} \setminus \{\overline{\Gamma}'_{21},\overline{\Gamma}'_{23} \}}{\longleftrightarrow}  u'_{23},\ \Gamma'_{21},\Gamma'_{22},\Gamma'_{23} \in \mathcal{L} ) \Big]  ,
           \end{align}}
where we used translation invariance (translation by $x$) in the last identity.  
Summing over $y$ and then over $x$, {and then also over $v_3',w_3'$ to get a multiplicative factor  $ |\Gamma'_{13}|\cdot |\Gamma'_{23}|$,}
the quantity on the LHS of \eqref{444} is bounded by
\begin{align*}
&G(r)^2 \sum_{ \substack{v_1',w_1'\in \Z^d \\ \min \{|v_1'|,|w_1'|\} \ge   K} } \sum_{v_2',w_2'\in \Z^d} \P (0\ar   v_2' ) \P(  0\ar   w_2') \P( v_1' \ar  w_1')  \nonumber   \\
& \qquad\qquad\qquad\qquad \cdot \sum_{\substack{ (\{\Gamma'_{1i}\}_{i=1}^3,\{u'_{1i}\}_{i=1}^3)\in \ble \\ \Gamma'_{1i} \sim v'_i,\ i=1,2 }}  \sum_{\substack{ (\{\Gamma'_{2i}\}_{i=1}^3,\{u'_{2i}\}_{i=1}^3)\in \ble \\\Gamma'_{2i} \sim w'_i,\ i=1,2 }} 
\Big[ |\Gamma'_{13}||\Gamma'_{23}| \nonumber \\
 &\qquad\qquad \qquad\qquad\qquad\qquad\qquad \cdot\P(u'_{11}  \overset{\overline{\mathcal{L}} \setminus \{\overline{\Gamma}_{11},\overline{\Gamma}_{12},\overline{\Gamma}_{13} \}}{\longleftrightarrow}  u'_{12}, u'_{12}  \overset{\overline{\mathcal{L}} \setminus \{\overline{\Gamma}_{11},\overline{\Gamma}_{13} \}}{\longleftrightarrow}  u'_{13},\ \Gamma'_{11},\Gamma'_{12},\Gamma'_{13} \in \mathcal{L} ) \Big]  \nonumber \\
&\qquad \qquad \qquad\qquad\qquad\qquad\qquad\cdot\P(u'_{21}  \overset{\overline{\mathcal{L}} \setminus \{\overline{\Gamma}_{21},\overline{\Gamma}_{22},\overline{\Gamma}_{23} \}}{\longleftrightarrow}  u'_{22}, u'_{22}  \overset{\overline{\mathcal{L}} \setminus \{\overline{\Gamma}_{21},\overline{\Gamma}_{23} \}}{\longleftrightarrow}  u'_{23},\ \Gamma'_{21},\Gamma'_{22},\Gamma'_{23} \in \mathcal{L} ) \Big]   \\
& \overset{\eqref{311} }{\le} CG(r)^2\sum_{ \substack{v_1',w_1'\in \Z^d \\ \min \{|v_1'|,|w_1'|\} \ge   K} }  \sum_{v'_2,w_2' \in \Z^d} |v'_2|^{2-d}|w_2'|^{2-d}  |v_1'-w_1'|^{2-d} \cdot  |v_1'-v'_2|^{ {6-d}}  |w_1'-w_2'|^{ {6-d}}  \\
&\le  CG(r)^2 \sum_{ \substack{v_1',w_1'\in \Z^d \\ \min \{|v_1'|,|w_1'|\} \ge   K} }  |v_1'|^{ {8-d}}|w_1'|^{ {8-d}}  |v_1'-w_1'|^{2-d},
\end{align*}
{where the last inequality follows from using Lemma \ref{lemma basic} to bound the summation over $v_2'$ and $w_2'$ respectively. The final summation, as in the proof of Lemma \ref{lemma 3.4}, can be bounded using Lemma  \ref{appendix1} (in particular using \eqref{7200} with $\alpha_1 := 8-d$, $\alpha_2 := 8-d$ and   $\alpha_3 := 2-d$): Since $\alpha_1, \alpha_2$ and $\alpha_3$ satisfy the hypothesis in Lemma  \ref{appendix1} (due to the condition $d>20$), for sufficiently large $K$, 
 \begin{align*}
 \sum_{ \substack{v_1',w_1'\in \Z^d \\ \min \{|v_1'|,|w_1'|\} \ge   K} }  |v_1'|^{ {8-d}}|w_1'|^{ {8-d}}  |v_1'-w_1'|^{2-d} \le \e.
  \end{align*} 
  This concludes the proof.}
\end{proof}

Using Lemmas \ref{lemma 3.4} and \ref{lemma 3.5}, we conclude the  proof of Proposition \ref{key prop}.
\begin{proof}[Proof of Proposition \ref{key prop}]
For $K,r \in \N$, define
 \begin{align*}
    A^{(r,K)}:= \{ (x,y)\in \Z^d \times \Z^d : 0\arr x \circ x \overset{r}{\ar} y, \ {d^\extr}(x,\ell^{(0,y)}) < K
    \}.
\end{align*}
Since  $0\arr x \circ x \overset{r}{\ar} y$ implies $|\ell^{(0,y)}| \le 2r$, there exists $C>0$ such that for any $K,r\in \N$,
\begin{align*}
   |A^{(r,K)} |  \le  C\cdot  2rK^{d}  \cdot |B(0,2r)|.
\end{align*}
Taking the expectation,
\begin{align} \label{323}
   \E |A^{(r,K)} |  \le  2CrK^{d}  \cdot  G(2r).
\end{align}
In addition, by Lemmas \ref{lemma 3.4} and \ref{lemma 3.5},  there exist $K',c'>0$ such that for any $r\in \N,$
\begin{align} \label{322}
        \E &|A^{(r,K')} | = \sum_{x,y\in \Z^d} \P (0\arr x \circ x \overset{r}{\ar} y,  \  {d^\extr}(x,\ell^{(0,y)}) < K')  \nonumber \\
        &= \sum_{x,y\in \Z^d} \P (0\arr x \circ x \overset{r}{\ar} y) -  \sum_{x,y\in \Z^d} \P (0\arr x \circ x \overset{r}{\ar} y,  \  {d^\extr}(x,\ell^{(0,y)}) \ge K' ) \nonumber  \\
        & \ge c' G(r)^2. 
        \end{align}
This together with \eqref{323} (with $K$ replaced by $K'$) conclude the proof.
    
\end{proof}


In the next section we rely on volume growth estimates to bounding the intrinsic one-arm probability which will then serve as a crucial input in obtaining desired resistance estimates. 

\section{Intrinsic one-arm exponent} \label{section 5}

Our main result in this section is the following.
\begin{proposition} \label{one arm 2}{Let $d> 20.$} There exists $C>0$ such that for any $r\in \N$,
    \begin{align*}
\P(  \partial B(0,r) \neq \emptyset) \leq Cr^{-1}.
    \end{align*}
\end{proposition}
While we will not be relying on it, as in the case of the volume bound in Proposition \ref{ball} outlined in Remark \ref{lb1},
we include another remark here outlining briefly how to obtain a corresponding lower bound.
\begin{remark}\label{lb2}
The argument is based on the second moment method (a similar argument for the bond case was made in \cite[Theorem 1.3]{kn2}).
Note by Proposition \ref{ball} and Remark \ref{lb1}, for some universal large constant $\lambda,$
$$
\E|\B(0,\lambda r)\setminus B(0,r)|\gtrsim r.
$$
   Thus by the Paley-Zygmund inequality, an estimate of the following form suffices.     \begin{align*}
\E [| B(0, r) |^2 ] \lesssim r^3.
    \end{align*} 
Note that the above expectation is the same as $\sum_{x,y\in \Z^d}\P(0\arr x, 0\arr y).$
The analysis of this calls for a generalization of Proposition \ref{grand} which in fact follows from the same argument as the proof of the latter verbatim to  yield that $0\arr x, 0\arr y$ implies that with the same notations as in the statement of Proposition \ref{grand}
    \begin{align*}
&v_1\overset{\overline{\mathcal{L}} \setminus \{\overline{\Gamma},\overline{\Gamma}_1,\overline{\Gamma}_2,\overline{\Gamma}_3  \},r}{\longleftrightarrow} x \circ v_2 \overset{\overline{\mathcal{L}} \setminus \{\overline{\Gamma},\overline{\Gamma}_1,\overline{\Gamma}_2,\overline{\Gamma}_3  \},r}{\longleftrightarrow}  0 \circ  v_{3} \overset{\overline{\mathcal{L}} \setminus \{\overline{\Gamma},\overline{\Gamma}_1,\overline{\Gamma}_2,\overline{\Gamma}_3  \},r}{\longleftrightarrow}  y \\
&\circ u_1 \overset{\overline{\mathcal{L}} \setminus \{\overline{\Gamma},\overline{\Gamma}_1,\overline{\Gamma}_2,\overline{\Gamma}_3 \}}{\longleftrightarrow}  w_1 \circ u_2 \overset{\overline{\mathcal{L}} \setminus \{\overline{\Gamma},\overline{\Gamma}_1,\overline{\Gamma}_2,\overline{\Gamma}_3 \}}{\longleftrightarrow}  w_2 \circ u_3 \overset{\overline{\mathcal{L}} \setminus \{\overline{\Gamma},\overline{\Gamma}_1,\overline{\Gamma}_2,\overline{\Gamma}_3 \}}{\longleftrightarrow}  w_3. 
 \end{align*}
 Note that in the first line all three connections have an intrinsic length less than $r$ constraint as opposed to the statement of Proposition \ref{grand} where only two of the terms have such a constraint. The presence of these three terms, on an identical analysis as in the proof of Proposition \ref{grand}, leads to the desired bound $\E [| B(0, r) |^2 ]  \lesssim r^3.$  
\end{remark}

The classical \emph{extrinsic} one-arm   exponent 
 describes the power law decay of the
probability that the origin is connected to sphere of radius $r$ w.r.t. the extrinsic distance:
\begin{align}
  \P(0 \ar \partial B^\extr(0,r))  = r^{-1/\rho + o(1)}.
\end{align}
It is known \cite{kol} that $\rho=1$ in the case of critical percolation on the  infinite regular tree. In a 
breakthrough work obtained by Lawler-Schramm-Werner \cite{sle1}, which relies on the work by Smirnov \cite{sle2}, it is proved that  $\rho= 48/5 $  for critical percolation on 
the two dimensional triangular-lattice. In high dimensions, the mean field exponent was established to be $\rho=1/2$ in the influential work \cite{kn1}. 
Recall that in \cite{cd}, the authors extended this to our model of critical level set percolation of GFF on $\cable$ establishing the analogous version of \eqref{onearmextrinsic}.  Namely,   for $d>6,$ there exist $C_1,C_2>0$ such that for any $r\in \N$,
\begin{align}
   C_1 r^{-2}  \le  \P(0 \ar \partial B^\extr(0,r)) \le C_2 r^{-2} .
\end{align}

\subsection{Absence of large loops}
As a useful input we first establish a crude bound showing that all loops intersecting $B(0,r)$ have a size {at most $r^{0.9}$ with high probability (the exponent 0.9 has no specific meaning and is taken to be close to one).}
{For  $i\in \N$, let  ${\mathcal{L}}^{\le i}$ be the collection of {discrete loops} in $\mathcal{L}$ which share some vertex with  $ B(0,i)$.}
The following proposition  excludes the existence of a large {discrete loop} in  ${\mathcal{L}}^{\le r}$.

\begin{proposition} \label{no big}
{Suppose that $d>20$.} Then there exists $C>0$ such that for any $L,r \in \N,$ 
    \begin{align} \label{501}
        \P(\exists \Gamma\in  {\mathcal{L}}^{\le r}\textup{ such that } |\Gamma|  \ge  L) \le  CrL^{1-d/2}.
    \end{align}
    In particular,
        \begin{align*}
        \P(\exists \Gamma\in  {\mathcal{L}}^{\le r}\textup{ such that } |\Gamma|  \ge  r^{0.9}) \le  Cr^{-5}.
    \end{align*}
\end{proposition}
To prove this proposition, we need the following simple connection property of a {discrete loop} $\Gamma$ in   ${\mathcal{L}}^{\le r}$ to the origin. For any discrete loop $\Gamma$, we denote by  $\textsf{Cont}(\Gamma)$ the collection of loops $\tiloop$ in  $\widetilde \cL$ {whose trace is the discrete loop} $\Gamma$ (i.e.  {a loop $\tiloop$  in $\textsf{Cont}(\Gamma)$ satisfies $ \textsf{Trace}(\tiloop) = \Gamma$, see \eqref{trace} for the definition of the map $ \textsf{Trace}$}). 
\begin{lemma} \label{geometry5}
Let $r\in \N$
 and $\Gamma$ be any {discrete loop} in   ${\mathcal{L}}^{\le r}$.   Then there exists  $x\in \Z^d$ with {$\Gamma \sim x$} such that $0 \overset{\widetilde \cL \setminus  \textup{\textsf{Cont}}(\Gamma)      ,r}{\longleftrightarrow} x$.
\end{lemma}
 
\begin{proof}
{
Let $y \in \Z^d$ be a point on ${\Gamma}$ which is closest to the origin (w.r.t. intrinsic distance). Note that there may be several such points and we pick any of them. Define $x \neq y$ to be the first lattice point that the geodesic $\ell$  from $y$ to $0$ passes through (there may be multiple choice of $x$ but we choose any of them). {Then the connection from $x$ to $0$ (as a sub-path of $\ell$) does \emph{not} use any loops in $\textsf{Cont}(\Gamma)$. This is because otherwise it contradicts the fact that $y$ is one of the  closest points on  ${\Gamma}$   to the origin.} In addition, as $\Gamma$ shares a vertex with $B(0,r)$, we have $|\ell| \le r.$  } See for instance an illustration in Figure \ref{5.3}.

\begin{figure}[h]
\centering
\includegraphics[scale=1]{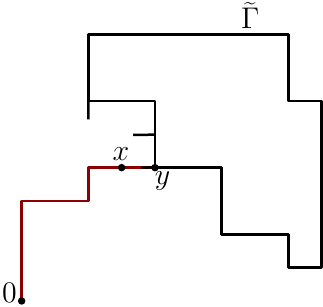}
\caption{Illustration of the proof of Lemma \ref{geometry5} where no part of the loop $\widetilde \Gamma$ which may possess partial edges as indicated in the figure is used by the path from $x$ to $0.$
}
\label{5.3}
\end{figure}
\end{proof}

Next, we upper bound the probability of the connection $0 \overset{\widetilde \cL \setminus \textsf{Cont}(\Gamma)      ,r}{\longleftrightarrow} x$ appearing in Lemma \ref{geometry5}. The following lemma is essential not only in the proof of Proposition \ref{no big} but will also be useful later.
\begin{lemma} \label{no big2}
There exists $C>0$ such that for any $L,r \in \N$ and $i\ge 0,$ 
        \begin{align} \label{521}
           \sum_{x\in \Z^d} \sum_{ \substack{d^\extr( \Gamma,x) \le 1 \\ |\Gamma|  \ge L }} |\Gamma|^i \P(0 \overset{\widetilde \cL \setminus \textup{\textsf{Cont}}(\Gamma) ,r}{\longleftrightarrow} x,\Gamma \in \mathcal{L})&\le \sum_{x\in \Z^d} \sum_{ \substack{ d^\extr( \Gamma,x) \le 1\\ |\Gamma|  \ge L }} |\Gamma|^i \P(0\arr x) \P(\Gamma \in \mathcal{L})  \nonumber \\
           &\le  C r L^{i+1-d/2}
    \end{align}
    and 
    \begin{align} \label{522}
          \sum_{x\in \Z^d} \sum_{ \substack{ d^\extr( \Gamma,x) \le 1\\ |\Gamma|  \ge L }} |\Gamma|^i \P(0\arr x,\Gamma \in \mathcal{L})  \le  C r L^{i+2-d/2}.
    \end{align}
\end{lemma}

Before proving the above,
assuming the same we finish the proof of Proposition \ref{no big}.
\begin{proof}[Proof of Proposition \ref{no big}]
By Lemma \ref{geometry5} and
a union bound, the probability in \eqref{501} 
    is bounded by
\begin{align*}
     \sum_{x\in \Z^d} \sum_{ \substack{ d^\extr( \Gamma,x) \le 1\\ |\Gamma|  \ge L }}  \P(0 \overset{\widetilde \cL \setminus \textsf{Cont}(\Gamma) ,r}{\longleftrightarrow} x,\Gamma \in \mathcal{L}).
\end{align*}
By Lemma \ref{no big2} (in particular, \eqref{521} with $i=0$), this quantity is bounded by $C r L^{1-d/2}.$

The second assertion is a consequence of  \eqref{501} with $L:=r^{0.9}$ and the assumption $d>20$ {(note that  the  weaker condition $d>15$ already suffices but we keep the condition $d>20$  which is assumed throughout the paper).}
\end{proof}

The proof of Lemma \ref{no big2} is now presented.
\begin{proof}[Proof of Lemma \ref{no big2}]
Let us first prove \eqref{521}.
By translation invariance,
\begin{align*}
     \sum_{x\in \Z^d} \sum_{ \substack{ d^\extr( \Gamma,x) \le 1\\ |\Gamma|  \ge L }} |\Gamma|^i \P(0 \overset{\widetilde \cL \setminus \textsf{Cont}(\Gamma) ,r}{\longleftrightarrow} x,\Gamma \in \mathcal{L})& \le \sum_{x\in \Z^d}\sum_{\substack{d^\extr( \Gamma,x) \le 1 \\ 
    |\Gamma|  \ge  L}} |\Gamma|^i\P(0 \arr x)\P( \Gamma \in \mathcal{L})\\
    &= \sum_{x\in \Z^d}\sum_{\substack{d^\extr( \Gamma',0) \le 1 \\ 
    |\Gamma'|  \ge  L}} |\Gamma'|^i \P(-x \arr 0) \P( \Gamma' \in \mathcal{L}) \\
  &=\sum_{\substack{d^\extr( \Gamma',0) \le 1 \\
    |\Gamma'|  \ge  L}}   \sum_{x\in \Z^d} 
 |\Gamma'|^i \P(0 \arr -x) \P(\Gamma' \in \mathcal{L}) \\
  &\le Cr \sum_{\substack{d^\extr( \Gamma',0) \le 1 \\
    |\Gamma'|  \ge  L}}  |\Gamma'|^i \P( \Gamma' \in \mathcal{L}) \le Cr L^{i+1-d/2},
\end{align*}
where the last two inequalities are consequences of  Proposition \ref{ball}  and Lemma \ref{second moment} respectively.

~

Next, we verify \eqref{522}. {By the similar reasoning as in Lemma \ref{geometry5}, $ \{d^\extr( \Gamma,x) \le 1 , \ 0 \arr x\}$ implies the existence of  $y\in \Z^d$ with $d^\extr( \Gamma,y) \le 1$ such that   $0 \overset{\widetilde \cL \setminus \textsf{Cont}(\Gamma) ,r}{\longleftrightarrow} y$.}
Hence
by a union bound,
\begin{align*} 
 \sum_{x\in \Z^d} \sum_{ \substack{ d^\extr( \Gamma,x) \le 1\\ |\Gamma|  \ge L }} |\Gamma|^i \P(0\arr x,\Gamma \in \mathcal{L})   &\le  \sum_{x\in \Z^d} \sum_{ \substack{ d^\extr( \Gamma,x) \le 1\\ |\Gamma|  \ge L }}   \sum_{ \substack{y\in \Z^d \\ d^\extr( \Gamma,y) \le 1 }}|\Gamma|^i \P( 0\overset{\widetilde \cL \setminus \textsf{Cont}(\Gamma) ,r}{\longleftrightarrow}y,\Gamma \in \mathcal{L})   \\
   &\le  \sum_{x\in \Z^d} \sum_{ \substack{ d^\extr( \Gamma,x) \le 1\\ |\Gamma|  \ge L }} \sum_{ \substack{y\in \Z^d \\ d^\extr( \Gamma,y) \le 1 }} |\Gamma|^i  \P( 0\overset{r}{\ar} y ) \P(\Gamma \in \mathcal{L})  \\
     &= \sum_{x\in \Z^d} \sum_{ \substack{ d^\extr( \Gamma',0) \le 1\\ |\Gamma'|  \ge L }}  \sum_{ \substack{y'\in \Z^d \\ d^\extr( \Gamma',y') \le 1 }} |\Gamma'|^i   \P( -x\arr y' ) \P(\Gamma' \in \mathcal{L})  \\
     &= \sum_{ \substack{ d^\extr( \Gamma',0) \le 1\\ |\Gamma'|  \ge L }} \sum_{ \substack{y'\in \Z^d \\ d^\extr( \Gamma',y') \le 1 }}|\Gamma'|^i    \sum_{x\in \Z^d}\P(y'\arr -x ) \P(\Gamma' \in \mathcal{L})  \\
      &\le   Cr  \sum_{ \substack{ d^\extr( \Gamma',0) \le 1\\ |\Gamma'|  \ge L }} \sum_{ \substack{y'\in \Z^d \\ d^\extr( \Gamma',y') \le 1 }} |\Gamma'|^i \P(\Gamma'
      \in \mathcal{L})  \\ 
 &\le  Cr  \sum_{ \substack{ d^\extr( \Gamma',0) \le 1\\ |\Gamma'|  \ge L }}|\Gamma'|^{i+1} \P(\Gamma'
      \in \mathcal{L})   \le  C r L^{i+2-d/2}.
 \end{align*}
Note that above we employed the change of variables  $\Gamma' = \Gamma-x$ and $y'=y-x$  in the third line above.

\end{proof}

Finally, we introduce the following useful lemma.
For $v,w\in \Z^d,$ we  say that $v \Join w$ if there exists a {discrete loop} $\Gamma \in \mathcal{L}$ such that  
$v,w\in \Gamma$.

\begin{lemma} \label{modified second}
There exists $C>0$ such that  for any $r\in \N,$  
    \begin{align} \label{540}
        \E\Big [ \sum_{v\in \Z^d } \sum_{w \in \Z^d  } \1_{0 \arr w} \1_{w  \Join  v}\Big ] \le Cr.
    \end{align}
\end{lemma}

\begin{proof}
By a union bound, the above quantity is bounded by 
    \begin{align*} 
     \E\Big [ \sum_{v\in \Z^d } \sum_{w \in \Z^d  } \1_{0 \arr w} \sum_{\Gamma \ni v,w} \1_{\Gamma \in \mathcal{L}} \Big ] &= \E\Big [ \sum_{w \in \Z^d  } \1_{0 \arr w} \sum_{\Gamma \ni w}  \sum_{v\in \Gamma }  \1_{\Gamma \in \mathcal{L}} \Big ]  \\
     &\le 
         \E \Big[ \sum_{w \in \Z^d  } \1_{0 \arr w} \sum_{\Gamma \ni w} |\Gamma| \1_{\Gamma \in \mathcal{L}}  \Big]  \\
         &= \sum_{w\in \Z^d }     \sum_{\Gamma \ni w}  |\Gamma|  \P(0 \arr w, \Gamma\in \mathcal{L}) \le Cr,
    \end{align*}
    where we used Lemma \ref{no big2} (in particular, \eqref{522} with $i=1$)  in the last inequality.
\end{proof}

\subsection{Arm exponent}
{In this section, we establish the (intrinsic) one-arm exponent.  Recall that when considering loops of edge type associated to the edge $e$, we do not consider these edge loops individually but consider their union which we have been  referring to as the glued loop $\overline{\Gamma}_e$.}

For  any (possibly random) collection $\widetilde{\mathscr{S}}$ {of loops} on $\cable$,  define
\begin{align} \label{vs}
    V(\widetilde{\mathscr{S}}):=\{v\in \Z^d: d^\extr(v,\tiloop) \le 1 \text{ for some loop }  \tiloop  \in \widetilde{\mathscr{S}}\}.
\end{align}
For $u,v\in \cable$, let $d(u,v;\widetilde{\mathscr{S}})$  be the length of the shortest path from $u$ to $v$, only using the loops in  $\widetilde{\mathscr{S}}$. 
Define the $r$-ball and $r$-sphere w.r.t. $\widetilde{\mathscr{S}}$ as follows:   For $v\in \Z^d$ and $r\in \N$,
{\begin{align}\label{ballsphere}
  \widetilde B(v,r;\widetilde{\mathscr{S}})= \{ w\in \cable: d(w,v;\widetilde{\mathscr{S}}) \le r\}, \quad  \partial B(v,r;\widetilde{\mathscr{S}})= \{ w\in \Z^d: d(w,v;\widetilde{\mathscr{S}}) =r\}.
\end{align}
Note that the $r$-ball includes points in $\cable$, whereas the $r$-sphere includes only lattice points.}

~

From now on, we assume that $\widetilde{\mathscr{S}}$ is a \emph{deterministic} collection of loops. 
As we will see, for our applications, it will suffice to restrict to classes of $\widetilde{\mathscr{S}}$ with the following property. 

\begin{assumption}\label{ass123}
The collection will contain fundamental, point loops (with certain bounds on the lengths of their partial edges) and all possible glued edge loops associated to all but finitely many edges, i.e. the entire support of the glued edge loop for those edges. Further, for any of the finitely many edges whose glued edges are not included, there is at least one of its incident points such that no point loop corresponding to that point, or fundamental loop passing through that point is included.  
\end{assumption}

The final property above is the continuous counterpart of the discrete statement that if any edge is not included, one of its endpoints must also be missing.  
 These properties will be used crucially in certain topological arguments involving the geometry of the connected components.

Let $\bwt \cL_{\widetilde{\mathscr{S}}}$ be the loop soup $\widetilde \cL$ restricted to  $\widetilde{\mathscr{S}}$.
Set 
\begin{align} \label{larger}
    \Lambda(r) := \sup_{v,\widetilde{\mathscr{S}}} \P(  \partial B(v,r;\bwt \cL_{\widetilde{\mathscr{S}}}) \neq \emptyset),
\end{align}
where the {supremum is taken over all vertices $v\in \Z^d$ and (deterministic) collections $\widetilde{\mathscr{S}}$ of loops on $\cable$} satisfying Assumption \ref{ass123}.  {A crucial aspect is that  $\{\partial B(0,r;\bwt \cL_{\widetilde{\mathscr{S}}}) \neq \emptyset\}$ is \emph{not} a monotone event in $\widetilde{\mathscr{S}}$. This is because for $\widetilde{\mathscr{S}}_1\subseteq \widetilde{\mathscr{S}}_2,$ we have $d(0,v;\widetilde{\mathscr{S}}_1) \ge d(0,v;\widetilde{\mathscr{S}}_2)$ and thus  $   \widetilde B(0,r;\widetilde \cL_{\widetilde{\mathscr{S}}_1}) \subseteq   \widetilde B(0,r;\widetilde \cL_{\widetilde{\mathscr{S}}_2})$. Hence it can happen that $\{\partial B(0,r;\widetilde \cL_{\widetilde{\mathscr{S}}_2})  = \emptyset\}$ even if $\{\partial B(0,r;\widetilde \cL_{\widetilde{\mathscr{S}}_1}) \neq  \emptyset\}$.}

~

The following  proposition bounds the probability \eqref{larger} uniformly in all vertices  $v\in \Z^d$ and (deterministic)  collections of loops  $\widetilde{\mathscr{S}}$ satisfying Assumption \ref{ass123},
\begin{proposition} \label{one arm}
There exists $C>0$ such that for any $r\in \N$,
    \begin{align*}
        \Lambda(r) \leq Cr^{-1}.
    \end{align*}
\end{proposition}
Proposition \ref{one arm 2} is an immediate consequence of Proposition \ref{one arm} by taking $\widetilde{\mathscr{S}}$ to be the collection of all loops and $v=0$ in \eqref{larger}.    From now on, we aim to establish  Proposition \ref{one arm}.

  Before embarking on the proof of  Proposition \ref{one arm}, we introduce some notations.
  \begin{definition}
   For  $i\in \N$, {let $\bwt \cL_{\widetilde{\mathscr{S}}}^i     $ (resp. $ \bwt \cL_{\widetilde{\mathscr{S}}}^{\le i}$)   be the collection of loops in $\bwt \cL_{\widetilde{\mathscr{S}}}$ which  intersect  $\partial B(0,i;\bwt \cL_{\widetilde{\mathscr{S}}})$ (resp. {${ \widetilde B(0,i;\bwt \cL_{\widetilde{\mathscr{S}}})}$}). Note that $\bwt \cL_{\widetilde{\mathscr{S}}}^i     $ does not contain any loop of edge type (i.e. $\overline{\gamma}_e$ for some edge $e$) since $\partial B(0,i;\bwt \cL_{\widetilde{\mathscr{S}}})$  is a lattice set and $\overline{\gamma}_e$ does not contain a lattice point.} Also $\bwt \cL_{\widetilde{\mathscr{S}}}^i$s are \emph{not} disjoint in general since a large loop can be contained in several  $\bwt \cL_{\widetilde{\mathscr{S}}}^i$s. 
   
   In addition, we say that an edge $e$ is {\emph{within $\widetilde{\mathscr{S}}$-level $i$}}   if it {intersects} some loop in $\bwt \cL_{\widetilde{\mathscr{S}}}^{\le i},$ {in other words if there exists a fundamental or point loop in  $\bwt \cL_{\widetilde{\mathscr{S}}}^{\le i}$ passing through some  vertex of $e$. This is because if $e$ intersects some  (glued) edge loop $\overline{\gamma}_e$ in $  \bwt \cL_{\widetilde{\mathscr{S}}}^{\le i}$, then some vertex of $e$ belongs to $  B(0,i;\bwt \cL_{\widetilde{\mathscr{S}}})$ since  $\overline{\gamma}_e$ is contained in $e$ and $\overline{\gamma}_e \in \bwt \cL_{\widetilde{\mathscr{S}}}^{\le i}$.}
\end{definition}

For an edge $e$  contained in some loop in $\bwt \cL_{\widetilde{\mathscr{S}}},$
denote by $\textsf{v}^{\text{out}}_e$  (resp. $\textsf{v}^{\text{in}}_e$)  {the vertex of $e$ whose intrinsic distance from 0 is greater (resp. less) than the other vertex (we drop the notation $\widetilde{\mathscr{S}}$ in these definitions for brevity which will be clear from context in our arguments). Note that these vertices are well-defined since distances from 0 to the two vertices of $e$ have different parities and thus are distinct. By our definition, 
\begin{align*}
d(0,\textsf{v}^{\text{out}}_{e};\bwt \cL_{\widetilde{\mathscr{S}}}) = d(0,\textsf{v}^{\text{in}}_{e};\bwt \cL_{\widetilde{\mathscr{S}}}) + 1.
\end{align*}
}

The following simple bounds recorded next will be useful.
   First, for any events $A_1,\cdots,A_m,B$ and $n,m\in \N$ with $n\le m$,
\begin{align} \label{inequality}
    \P(\{\text{At least $n$ events among the events $A_1,\cdots,A_m$ occur}\} \cap B) \le \frac{1}{n}\sum_{i=1}^m\P(A_i\cap B).
\end{align}
This is an immediate consequence of the inequality (by taking the expectation)
\begin{align*}
    \1_{\text{at least $n$ events  among the events $A_1,\cdots,A_m$ occur}} \cdot \1_B \le \frac{1}{n}\sum_{i=1}^m\1_{A_i} \cdot \1_B.
\end{align*}
Second, assume that {$F$ is an event} and for some $p>0$,  uniformly in  $\bwt \cL_{\widetilde{\mathscr{S}}}^{\le i}$, 
\begin{align} \label{506}
    \P(F \mid \bwt \cL_{\widetilde{\mathscr{S}}}^{\le i}) \le p.
\end{align}
Then for any event $E$ {measurable with respect to $\bwt \cL_{\widetilde{\mathscr{S}}}^{\le i}$}, \begin{align} \label{universal2}
    \P(E\cap F)  =  \E [ \P(  F  \mid  \bwt \cL_{\widetilde{\mathscr{S}}}^{\le i}) \1_E ] \le p \P(E).
\end{align}
{Here we say that the event $E$ is  measurable with respect to $\bwt \cL_{\widetilde{\mathscr{S}}}^{\le i}$, or $E$ is determined only by the information of $\bwt \cL_{\widetilde{\mathscr{S}}}^{\le i}$, if the occurrence of the event $E$ does not depend on the realization of loops outside the set $\bwt \cL_{\widetilde{\mathscr{S}}}^{\le i}$. }

\vspace{3mm}

Now, we begin the proof of Proposition \ref{one arm}. {From now on, it would be helpful to denote a lattice path as $\ell = (e_1,\cdots,e_m)$, where $e_i$s are the consecutive edges in $E(\Z^d)$ along $\ell$.}

\begin{proof}[Proof of Proposition \ref{one arm}]
The proof consists of four steps.\\

\noindent
 \textbf{Step 1. Setup.}
By translation invariance,
\begin{align*}
    \P(  \partial B(v,r;\bwt \cL_{\widetilde{\mathscr{S}}}) \neq \emptyset) = \P(  \partial B(0,r;\widetilde \cL_{\widetilde{\mathscr{S}}-v}) \neq \emptyset),
\end{align*}
where $\widetilde{\mathscr{S}}-v$ denotes the collection of loops in $\widetilde{\mathscr{S}}$, translated by $v$. Hence, it suffices to bound
\begin{align*}
     \Lambda(r) =
    \sup_{\widetilde{\mathscr{S}}} \P(  \partial B(0,r;\bwt \cL_{\widetilde{\mathscr{S}}}) \neq \emptyset).
\end{align*}
As indicated in Section \ref{iop}, we will aim to establish a recursive property of $\Lambda(\cdot).$ In particular,
considering $r=3^k$ ($k\in \N$) we will establish relation between $\Lambda(3^k)$ and $\Lambda(3^{k-1})$.
Let $\tilde{\mathcal{C}}_{\widetilde{\mathscr{S}}}(0)$ be the cluster containing the origin w.r.t. the loop soup $\bwt \cL_{\widetilde{\mathscr{S}}}$. Then, for any $\e>0,$
    \begin{align} \label{411}
        \P( \partial B   (0,3^k; \bwt \cL_{\widetilde{\mathscr{S}}}) \neq \emptyset) \le  \P(  \partial B(0,3^k; \bwt \cL_{\widetilde{\mathscr{S}}}) \neq   \emptyset,  |\tilde{\mathcal{C}}_{\widetilde{\mathscr{S}}}(0)| < \e 9^k) + \P( |\tilde{\mathcal{C}}_{\widetilde{\mathscr{S}}}(0)|  \ge  \e 9^k).
    \end{align}
{Since for each $M>0$, the event $\{|\tilde{\mathcal{C}}_{\widetilde{\mathscr{S}}}(0)|  \ge M\}$ is monotone in  the collection of loops $\widetilde{\mathscr{S}}$ (w.r.t. the  inclusion relation)}, by Proposition \ref{cluster upper}, the second term above is bounded by
\begin{align} \label{410}   
    \P( |\tilde{\mathcal{C}}_{\widetilde{\mathscr{S}}}(0)|  \ge  \e 9^k)\le      \P( |\tilde{\mathcal{C}}(0)|  \ge  \e 9^k)\le 
 \frac{C}{\sqrt{\e} 3^k}
\end{align}
{(recall that $|\tilde{\mathcal{C}}(0)|  = |{\mathcal{C}}(0)| $ denotes the number of lattice points inside $\tilde{\mathcal{C}}(0)$ or ${\mathcal{C}}(0)$).}
We now move on to bounding the first term of RHS in \eqref{411}. It is here we employ the novel averaging argument outlined in Section \ref{iop}. Towards this, define the set of levels for which the cardinality of the corresponding sphere (recall \eqref{ballsphere}) is not large:
\begin{align*}
    I:= \{ i\in [ 3^{k-1} ,3^{k}/2]  :|\partial B(0,i; \bwt \cL_{\widetilde{\mathscr{S}}})| \le  12\e\cdot   3^k  \}.
\end{align*} 
As
\begin{align*}
    |\tilde{\mathcal{C}}_{\widetilde{\mathscr{S}}}(0)| \ge \sum_{i=3^{k-1}}^{3^k/2}|\partial B(0,i; \bwt \cL_{\widetilde{\mathscr{S}}})|,
\end{align*}
we have the implication
\begin{align} \label{422}
     |\tilde{\mathcal{C}}_{\widetilde{\mathscr{S}}}(0)| < \e 9^k \Rightarrow |I| \ge 3^{k-1}/4.
\end{align}

 Under the event $\partial B(0,3^k; \bwt \cL_{\widetilde{\mathscr{S}}})  \neq   \emptyset$,  
 let $\ell=(e_1,\cdots,e_{3^k})$ {be a geodesic  in $\bwt \cL_{\widetilde{\mathscr{S}}}$ from} the origin to $\partial B(0,3^k; \bwt \cL_{\widetilde{\mathscr{S}}}) $. Note that there may be several geodesics and we take any of them.
{For each level $i\in [ 3^{k-1} ,3^{k}/2]$, set
\begin{align} \label{421}
    J(i):= \max \{1\le m\le {3^k}: e_m  \text{ is within $\widetilde{\mathscr{S}}$-level $i$}  \}.
\end{align}
Note that both of the vertices of $e_{J(i)}$ are at a distance of at least $i$ from the origin.}
Then define 
{\begin{align} \label{t}
    T_i:= \{ \text{Loops in $\bwt \cL_{\widetilde{\mathscr{S}}}^i$ {intersecting} the edge $e_{J(i)}$}\}. 
    \end{align}
{Note also that $T_i$s may \emph{not} be disjoint because there might be a large loop intersecting several spheres.}
We will argue that this is indeed nonempty shortly but first remark that the nature of the bound obtained through an inductive step will depend on the size of the loops in $T_i$.}
 {We emphasize that in this definition, $T_i$ only includes the  loops in $\bwt \cL_{\widetilde{\mathscr{S}}}^i$ (\emph{not} in $\bwt \cL_{\widetilde{\mathscr{S}}}^{\le i}$).}
{We now show that $T_i$ is indeed not empty. {This is roughly because any loop in  $\bwt \cL_{\widetilde{\mathscr{S}}}^{\le i}$, intersecting $e_{J(i)}$, belongs to $\bwt \cL_{\widetilde{\mathscr{S}}}^i$. To show this formally,} we will use the fact that for any loop $\tiloop$ in $  \bwt \cL_{\widetilde{\mathscr{S}}}$, 
{\begin{align} \label{imply}
    \exists w_1,w_2\in \tiloop \cap \Z^d  \text{ 
s.t. }d(0,w_1;\bwt \cL_{\widetilde{\mathscr{S}}}) \le i, \  d(0,w_2;\bwt \cL_{\widetilde{\mathscr{S}}}) \ge i \Rightarrow \exists w'\in \tiloop \cap \Z^d \text{ 
s.t. }  d(0,w';\bwt \cL_{\widetilde{\mathscr{S}}})=i.
\end{align}}
This is because for any edge $e \in E(\Z^d)$ in the loop $\tiloop$, distances from $0$ to the two adjacent vertices of $e$ differ only by $1.$ As $e_{J(i)}$ is  within $\widetilde{\mathscr{S}}$-level $i$, there exists a loop $\tiloop\in \bwt \cL_{\widetilde{\mathscr{S}}}^{\le i}$ {intersecting} $e_{J(i)}$. Also by \eqref{421}, $$d(0,\textsf{v}^{\text{out}}_{e_{J(i)}};\bwt \cL_{\widetilde{\mathscr{S}}}) > i.$$
Hence  by \eqref{imply}, $\tiloop  $ should share some vertex with $\partial B(0,i;\bwt \cL_{\widetilde{\mathscr{S}}})$, i.e. $\tiloop\in \bwt \cL_{\widetilde{\mathscr{S}}}^i$. This shows that $\tiloop \in T_i$ which finishes the proof.} 

{Now, as sketched in Section \ref{iop} before, we implement the averaging argument to bound the contributions arising from the levels $i$ such that the loops in $T_i$ are small or large separately.} Beyond the $\ble,$ this is the other major tool we employ to handle the long range nature of $\bwt\cL.$
For a large constant $L>0$ which will be chosen later,
partition the set $I$ into  
{\begin{align*}
    I_1 &:= \{i\in I:  |\tiloop| <L,  \  \forall\,\, \tiloop \in        T_i  \},\\
    I_2 &:=    \{i\in I: \exists \,\, \tiloop \in        T_i  \text{ such that } |\tiloop|  \ge L  \} 
\end{align*} 
(recall that $|\tiloop|:= | \textsf{Trace}(\tiloop)|$, i.e.  {the size of the corresponding discrete loop)}.}
In other words, $I_1$ is the collection of levels in $I$ such that \emph{no} loop in $\bwt \cL_{\widetilde{\mathscr{S}}}^i$, {intersecting} the edge  $e_{J(i)}$, is large.
Setting the event 
\begin{align}
    \mathcal{G}_{\widetilde{\mathscr{S}}}:= \Big\{\text{Every {loop} in $\bwt \cL_{\widetilde{\mathscr{S}}}^{\le 3^k}$ has a size at most $\frac{1}{2} \cdot 3^{k-1} - 1$}\Big\},
\end{align}
by \eqref{422}, we have 
\begin{align} \label{412}
     \P(  \partial B(0,3^k; \bwt \cL_{\widetilde{\mathscr{S}}}) \neq   \emptyset,  |\tilde{\mathcal{C}}_{\widetilde{\mathscr{S}}}(0)| < \e 9^k) & \le {\P( \partial B(0,3^k; \bwt \cL_{\widetilde{\mathscr{S}}})  \neq   \emptyset, |I_1 | \ge 3^{k-1}/8, \mathcal{G}_{\widetilde{\mathscr{S}}})}  \nonumber \\ 
     & +   \P( \partial B(0,3^k; \bwt \cL_{\widetilde{\mathscr{S}}})  \neq   \emptyset, |I_2 | \ge 3^{k-1}/8, \mathcal{G}_{\widetilde{\mathscr{S}}})  + \P(\mathcal{G}_{\widetilde{\mathscr{S}}}^c).
\end{align}
{In fact, even through the event $\cG_{\widetilde{\mathscr{S}}}$ appearing in the first term, will be dropped in our analysis, here we keep it.}

\noindent
\textbf{Step 2. {Control on the first term in \eqref{412}.}}
We bound the first term above. For each level $i$, setting  
 \begin{align*}
 \bwt \cL_{\widetilde{\mathscr{S}}}^{i,<L}      := \{\tiloop \in \bwt \cL_{\widetilde{\mathscr{S}}}^i     : |\tiloop| < L\},
 \end{align*}
 we have  
 \begin{align} \label{4105}
     |V( \bwt \cL_{\widetilde{\mathscr{S}}}^{i,<L}       )| \le  CL^{d-1}|\partial B(0,i; \bwt \cL_{\widetilde{\mathscr{S}}})|
 \end{align}
 (see \eqref{vs} for the definition of $V(\cdot)$).
{This is because every {loop} in $\bwt \cL_{\widetilde{\mathscr{S}}}^{i,<L}   $ shares a vertex   with $\partial B(0,i; \bwt \cL_{\widetilde{\mathscr{S}}})$ and  the corresponding discrete loop has a length at most $L$.}
Next, we claim that 
\begin{align} \label{imply2}
    &\{i\in     I_1\} \cap \{\partial B(0,3^k; \bwt \cL_{\widetilde{\mathscr{S}}})  \neq   \emptyset\} \cap \mathcal{G}_{\widetilde{\mathscr{S}}}  \nonumber \\
    &\Rightarrow 1\le  |\partial B(0,i; \bwt \cL_{\widetilde{\mathscr{S}}} )| \le  12\e  \cdot 3^k,\quad \exists v\in V( \bwt \cL_{\widetilde{\mathscr{S}}}^{i,<L}       ) \text{ s.t. }\partial B(v,3^{k-1};   \bwt \cL_{\widetilde{\mathscr{S}}} \setminus 
 \bwt \cL_{\widetilde{\mathscr{S}}}^{\le i}      ) \neq \emptyset.
\end{align}
The first property follows from  the definition of $I$ along with the facts $I_1 \subseteq I$ and $i< 3^k$. To see the second property,  by \eqref{421}, {the sub-path of $\ell$ from $v:=\textsf{v}^{\text{out}}_{e_{J(i)}}$ to $\partial B(0,3^k; \bwt \cL_{\widetilde{\mathscr{S}}})$, which we call $\ell_2$,} does \emph{not} use loops in $\bwt \cL_{\widetilde{\mathscr{S}}}^{\le i}.$ Take any loop  $\tiloop \in T_i $ (recall that $T_i$ is not empty), and
set $u$ to be any point in $\tiloop$ such that  $ d(0,u; \bwt \cL_{\widetilde{\mathscr{S}}}) = i$ (which exists since  $\tiloop \in T_i \subseteq \bwt \cL_{\widetilde{\mathscr{S}}}^i $). Then under the event $\mathcal{G}_{\widetilde{\mathscr{S}}},$ 
\begin{align} \label{500}
    d(0,v; \bwt \cL_{\widetilde{\mathscr{S}}}) \le  d(0,u; \bwt \cL_{\widetilde{\mathscr{S}}})+ d(u,v; \bwt \cL_{\widetilde{\mathscr{S}}})\le  i  +  |\tiloop| +1 \le \frac{3^k}{2}+   \frac{3^{k-1}}{2} =2\cdot 3^{k-1},
\end{align}
 where the second inequality follows from the fact that $u\in \tiloop$ and $d^\extr(\tiloop,v) \le 1$ (since $\tiloop\in T_i$ intersects the edge $e_{J(i)}$, see \eqref{t}). This implies that 
\begin{align} \label{402}
    |\ell_2| \ge 3^k-2\cdot 3^{k-1} = 3^{k-1}.
\end{align}  
As $i\in I_1$ and thus $|\tiloop| < L$, recalling that  $e_{J(i)}$ is within $\widetilde{\mathscr{S}}$-level $i$, we have $v\in V( \bwt \cL_{\widetilde{\mathscr{S}}}^{i,<L}       )$,  verifying the implication \eqref{imply2}.

~

We next turn to bounding  the probability of the event in  \eqref{imply2}. To accomplish this, we state the following general claim, which will be useful later as well: 
If the collection of events  $\{F_v\}_{v\in V( \bwt \cL_{\widetilde{\mathscr{S}}}^{i,<L}       )}$ satisfies that  uniformly in  $\bwt \cL_{\widetilde{\mathscr{S}}}^{\le i}$,
\begin{align} \label{universal0}
 \sup_{v\in V( \bwt \cL_{\widetilde{\mathscr{S}}}^{i,<L}       )}\P(F_v \mid \bwt \cL_{\widetilde{\mathscr{S}}}^{\le i}) \le p
\end{align}
{(note that the events $F_v$s are indexed by a random set but this  will not affect our argument),}
 then for any $M \ge 1$ and an event $E$ measurable w.r.t. $\bwt \cL_{\widetilde{\mathscr{S}}}^{\le i}$,
\begin{align} \label{universal1}
    \P(E, \   1\le |\partial B     (0,i; \bwt \cL_{\widetilde{\mathscr{S}}})| \le M, & \quad   \exists v\in V( \bwt \cL_{\widetilde{\mathscr{S}}}^{i,<L}       ) \text{ s.t. }  F_v \text{ occurs})    \nonumber \\
    &\le CL^{d-1}Mp  \cdot  \P(E, \  \partial B     (0,i; \bwt \cL_{\widetilde{\mathscr{S}}}) \neq \emptyset) .
\end{align}
Indeed, by conditioning on $\bwt \cL_{\widetilde{\mathscr{S}}}^{\le i}$ and then using a union bound, the above probability is  written as
\begin{align*}
    \E \Big[\P( \exists v\in V( \bwt \cL_{\widetilde{\mathscr{S}}}^{i,<L}       ) & \text{ s.t. }  F_v \text{ occurs} \mid \bwt \cL_{\widetilde{\mathscr{S}}}^{\le i} )\1_E  \1_{ 1\le |\partial B     (0,i; \bwt \cL_{\widetilde{\mathscr{S}}})| \le M}  \Big] \\
     &  \overset{\eqref{4105} , \eqref{universal0} }{\le}  CL^{d-1}Mp\cdot \P(E, \  1\le |\partial B     (0,i; \bwt \cL_{\widetilde{\mathscr{S}}})| \le M),
\end{align*}
yielding \eqref{universal1}.

The following lemma characterizes the measure induced on the loops by the conditioning as in \eqref{universal0}, whose proof we postpone to the end of this section.
{
\begin{lemma} \label{forbid} Let $\widetilde{\mathscr{S}}$ be a collection of loops satisfying Assumption \ref{ass123}. Then 
for any $i\in \N$ and {a collection $\widetilde{\mathscr{A}}$ of loops in the support of $\widetilde\cL_{\widetilde{\mathscr{S}}}^{\le i}$,  let $\widetilde{\mathscr{A}}^{(i)}$ be the collection of  loops which intersect the closure of $  \widetilde B(0,i;\widetilde{\mathscr{A}})$} (as a subset of $\cable$). Then, for any $v\in \Z^d$ and $j \in \N$,
\begin{align}  \label{thin}
  \P  ( \partial B(v,j;    \bwt \cL_{\widetilde{\mathscr{S}}} \setminus 
 \bwt \cL_{\widetilde{\mathscr{S}}}^{\le i}      ) \neq \emptyset  \mid \bwt \cL_{\widetilde{\mathscr{S}}}^{\le i} = \widetilde{\mathscr{A}}) &=\P  ( \partial B(v,j;         \bwt \cL_{\widetilde{\mathscr{S}}}  \setminus {\widetilde{\mathscr{A}}^{(i)}} ) \neq \emptyset ).
\end{align}
\end{lemma}
} By the above lemma,
\begin{align} \label{4111}
    \P  ( \partial B(v,3^{k-1};    \bwt \cL_{\widetilde{\mathscr{S}}} \setminus 
 \bwt \cL_{\widetilde{\mathscr{S}}}^{\le i}      ) \neq \emptyset  \mid \bwt \cL_{\widetilde{\mathscr{S}}}^{\le i} = \widetilde{\mathscr{A}}) &=\P  ( \partial B(v,3^{k-1};         \bwt \cL_{\widetilde{\mathscr{S}}}  \setminus  {\widetilde{\mathscr{A}}^{(i)}} ) \neq \emptyset  )\nonumber \\
 &=\P  ( \partial B(v,3^{k-1};        {  \widetilde \cL_{ {\widetilde{\mathscr{S}}}  \setminus  {\widetilde{\mathscr{A}}^{(i)}} }} ) \neq \emptyset  ) \le \Lambda(3^{k-1}),
\end{align}
where the second identity is obtained by  $ \bwt \cL_{\widetilde{\mathscr{S}}} \setminus \widetilde{\mathscr{A}}^{(i)} \overset{\text{d}}{=}   \widetilde \cL_{ \widetilde{\mathscr{S}} \setminus \widetilde{\mathscr{A}}^{(i)}}$.
Note that the final inequality which invokes the induction step only holds once one establishes the following lemma. 
\begin{lemma} \label{topverify} If $\widetilde{\mathscr{S}}$ satisfies  Assumption \ref{ass123}, then almost surely, so does $ \widetilde{\mathscr{S}} \setminus \widetilde{\mathscr{A}}^{(i)},$ where $\widetilde{\mathscr{A}}=\bwt \cL_{\widetilde{\mathscr{S}}}^{\le i}.$
\end{lemma}
The proof of this is postponed to the end of the section. 

{Thus as $\partial B(0,i;\bwt \cL_{\widetilde{\mathscr{S}}})$ is determined \emph{only} by information of   $\bwt \cL_{\widetilde{\mathscr{S}}}^{\le i}$,   by  \eqref{universal1} with $E = \{\text{Entire sample space}\}$,
for each level $i\in [3^{k-1}, 3^k/2]$, 
\begin{align} \label{507}
    \P( &1\le  |\partial B     (0,i; \bwt \cL_{\widetilde{\mathscr{S}}})| \le  12\e  \cdot 3^k, \quad    \exists v \in V( \bwt \cL_{\widetilde{\mathscr{S}}}^{i,<L}       )\text{ s.t. } \partial B(v,3^{k-1};  \bwt \cL_{\widetilde{\mathscr{S}}}\setminus 
 \bwt \cL_{\widetilde{\mathscr{S}}}^{\le i}     ) \neq \emptyset)  \nonumber   \\
&\le CL^{d-1}  ( 12\e  \cdot 3^k )   \Lambda(3^{k-1})\P( \partial B     (0,i; \bwt \cL_{\widetilde{\mathscr{S}}}) \neq \emptyset)\le CL^{d-1} \e \cdot 3^k  (\Lambda(3^{k-1}))^2.
\end{align}}
Therefore using \eqref{inequality}, the first term in \eqref{412} is bounded by
\begin{align}   \label{413}
 & \frac{8}{3^{k-1}}   \sum_{i=3^{k-1}}^{ 3^{k}/2  } \P(i\in     I_1,\  \partial B(0,3^k; \bwt \cL_{\widetilde{\mathscr{S}}})  \neq   \emptyset   , \ \mathcal{G}_{\widetilde{\mathscr{S}}}  )  
 \nonumber \\
  &\overset{\eqref{imply2}}{\le} \frac{8}{3^{k-1}}   \sum_{i=3^{k-1}}^{ 3^{k}/2  } \P({ 1\le  |\partial B     (0,i; \bwt \cL_{\widetilde{\mathscr{S}}})| \le  12\e  \cdot 3^k, \    \exists v \in V( \bwt \cL_{\widetilde{\mathscr{S}}}^{i,<L}       )\text{ s.t. } \partial B(v,3^{k-1};  \bwt \cL_{\widetilde{\mathscr{S}}}\setminus 
 \bwt \cL_{\widetilde{\mathscr{S}}}^{\le i}     ) \neq \emptyset}) \nonumber \\
 &\overset{\eqref{507}}{\le}   CL^{d-1} \e \cdot 3^k  (\Lambda(3^{k-1}))^2.
\end{align}

\noindent
\textbf{Step 3.}
Next, we bound the second term in \eqref{412}. Let $\mathcal{L}_{\widetilde{\mathscr{S}}}$  be the collection of discrete loops induced by the loop soup $\bwt \cL_{\widetilde{\mathscr{S}}},$ and define {${\mathcal{L}}_{\widetilde{\mathscr{S}}}^{\le i}$} to be  the collection of discrete loops  in  $\mathcal{L}_{\widetilde{\mathscr{S}}}$ which share some vertex with $ B(0,i;\bwt \cL_{\widetilde{\mathscr{S}}})$.  
We claim  that 
\begin{align}  \label{400}
&\{i  \in I_2\} \cap \{\partial B(0,3^k; \bwt \cL_{\widetilde{\mathscr{S}}})  \neq   \emptyset\} \cap \mathcal{G}_{\widetilde{\mathscr{S}}}  \nonumber \\
&\Rightarrow 
    \exists  \ {\text{discrete loop}} \  \Gamma\in {\mathcal{L}}_{\widetilde{\mathscr{S}}}^{\le i}, \  \exists u\in \Gamma,  \ {\exists v\in \Z^d \text{ with } d^\extr(\Gamma,v) \le 1}  \text{ s.t. }   \nonumber \\
    &\qquad\qquad\{|\Gamma| \ge L\} \cap  \{d(0,u ;\bwt \cL_{\widetilde{\mathscr{S}}}  )=i\}  \cap \{  \partial B(v,     3^{k-1}; \bwt \cL_{\widetilde{\mathscr{S}}} \setminus \bwt \cL_{\widetilde{\mathscr{S}}}^{\le i}) \neq \emptyset \}.
\end{align}
To see this, as $i  \in   I_2$, there exists a {loop} $\tiloop  \in T_i$ with $|\tiloop| \ge  L$, {intersecting} $e_{J(i)}$. As $\tiloop \in T_i$,  there exists $u\in \tiloop$ such that $d(0,u;\bwt \cL_{\widetilde{\mathscr{S}}} )=i,$ implying $u\in \Gamma:=  \textsf{Trace}(\tiloop).$    Also the last event above holds for $v:=\textsf{v}^{\text{out}}_{e_{J(i)}}$ (which satisfies $d^\extr(\Gamma,v) \le 1 $), see the discussion after 
 \eqref{imply2}. 

Hence {using the above implication,} by a union bound along with the inequality \eqref{inequality}, the second term in \eqref{412} is bounded by  
  \begin{align}   \label{405}
      &\frac{8}{3^{k-1}} \sum_{i=3^{k-1}}^{3^k/2}  \P( i \in I_2, \  \partial B(0,3^k; \bwt \cL_{\widetilde{\mathscr{S}}})  \neq   \emptyset, \  \mathcal{G}_{\widetilde{\mathscr{S}}}) \nonumber \\
      &\le \frac{8}{3^{k-1}} \sum_{i=3^{k-1}}^{3^k/2} \sum_{u\in \Z^d}\sum_{ \substack{ \Gamma \ni u\\ |\Gamma|  \ge L }}\sum_{ \substack{v\in \Z^d \\ d^\extr(\Gamma,v) \le 1} } \P( d(0,u;\bwt \cL_{\widetilde{\mathscr{S}}} )=i, \  \partial B(v, {     3^{k-1}      }; \bwt \cL_{\widetilde{\mathscr{S}}} \setminus \widetilde \cL^{\le i}_{\widetilde{\mathscr{S}}}) \neq \emptyset , \  \Gamma \in {\mathcal{L}}_{\widetilde{\mathscr{S}}}^{\le i}) .
    \end{align}
    
{As the events $\{d(0,u;\bwt \cL_{\widetilde{\mathscr{S}}} ) =i \} $ and $\{\Gamma \in  {\mathcal{L}}_{\widetilde{\mathscr{S}}}^{\le i}\}$ are determined \emph{only} by $ \bwt \cL_{\widetilde{\mathscr{S}}}^{\le i}$,}  by \eqref{universal2} and \eqref{4111}, 
\begin{align} \label{step3}
    \P(d(0,u;\bwt \cL_{\widetilde{\mathscr{S}}} )=i,\  \partial B(v, {     3^{k-1}    }; \bwt \cL_{\widetilde{\mathscr{S}}} \setminus \widetilde \cL^{\le i}_{\widetilde{\mathscr{S}}}) \neq \emptyset &,\  \Gamma \in  {\mathcal{L}}_{\widetilde{\mathscr{S}}}^{\le i})  \nonumber \\
 &\le  \Lambda(3^{k-1}) {\P(d(0,u;\bwt \cL_{\widetilde{\mathscr{S}}} )=i  ,  \Gamma \in  {\mathcal{L}}_{\widetilde{\mathscr{S}}}^{\le i})} .
\end{align}
Plugging this into \eqref{405}, we deduce that the second term in \eqref{412} is bounded by  
    \begin{align} \label{step 33}
      &\frac{8}{3^{k-1}}\Lambda(3^{k-1}) \sum_{i=3^{k-1}}^{3^k/2} \sum_{u\in \Z^d} \sum_{ \substack{ \Gamma \ni u\\ |\Gamma|  \ge L }}  \sum_{ \substack{v\in \Z^d \\ d^\extr(\Gamma,v) \le 1} }\P(d(0,u;\bwt \cL_{\widetilde{\mathscr{S}}} )=i , \Gamma \in  {\mathcal{L}}_{\widetilde{\mathscr{S}}}^{\le i})  \nonumber \\
      &\le   \frac{C}{3^{k-1}}\Lambda(3^{k-1})\sum_{u\in \Z^d} \sum_{ \substack{ \Gamma \ni u\\ |\Gamma|  \ge L }} \sum_{i=3^{k-1}}^{3^k/2}
 |\Gamma| \P(d(0,u;\bwt \cL_{\widetilde{\mathscr{S}}} )=i , \Gamma \in  {\mathcal{L}}_{\widetilde{\mathscr{S}}})  \nonumber  \\
      &\le  \frac{C}{3^{k-1}}\Lambda(3^{k-1})  \sum_{u\in \Z^d} \sum_{ \substack{ \Gamma \ni u\\ |\Gamma|  \ge L }} |\Gamma| \P(0  
 \overset{\bwt \cL_{\widetilde{\mathscr{S}}},3^k/2}{\longleftrightarrow}  u,\Gamma \in  {\mathcal{L}}_{\widetilde{\mathscr{S}}})   \nonumber \\
  & \le    \frac{C}{3^{k-1}}\Lambda(3^{k-1})  \sum_{u\in \Z^d} \sum_{ \substack{ \Gamma \ni u\\ |\Gamma|  \ge L }} |\Gamma| \P(0  
 \overset{3^k/2}{\longleftrightarrow}  u,\Gamma \in \mathcal{L}) \le  C L^{3-d/2}\Lambda(3^{k-1}) ,
    \end{align}
    where we used Lemma  \ref{no big2} (in particular, \eqref{522} with $i=1$) in the last inequality.

~

\noindent
\textbf{Step 4. Conclusion.}
Since $ B(0,r;\bwt \cL_{\widetilde{\mathscr{S}}}) \subseteq B(0,r;\widetilde \cL)$ and thus $\bwt \cL_{\widetilde{\mathscr{S}}}^{\le r} \subseteq \widetilde \cL^{\le r}$, by Proposition \ref{no big},   {$$\P(\mathcal{G}_{\widetilde{\mathscr{S}}}^c) \le \P(\text{$\exists \Gamma \in {\mathcal{L}}_{\widetilde{\mathscr{S}}}^{\le 3^k}$ s.t.  $|\Gamma| \ge 3^{k-1}/2  - 1 $}) \le  C3^k\cdot (3^{k-1}/2)^{1-d/2} \le C(3^{k})^{2-d/2} .$$}
Therefore, plugging the previously obtained bounds into \eqref{411} and then taking {the  supremum over all collection of loops $\widetilde{\mathscr{S}}$}, we deduce that there exists a constant $C_0>0$ {(only depending on $d$)}, independent of $L$ and $\e$, such that  for any $L,\e>0$ and $k\in \N$, 
\begin{align} \label{4100}
     \Lambda(3^k) \le  \frac{C_0}{\sqrt{\e}3^k}+    C_0L^{d-1} \e \cdot 3^k  (\Lambda(3^{k-1}))^2+  C_0 L^{3-d/2} 
 \Lambda(3^{k-1})+ {C_0 (3^k)^{2-d/2}}.
\end{align}
Setting  $L$ to be any large enough constant satisfying
\begin{align} \label{4101}
     3C_0  L^{3-d/2}   < \frac{1}{4},
\end{align}
and then take a constant $D>0$ such that
\begin{align} \label{4102}
6C_0^{3/2}D^{1/2}L^{(d-1)/2} < \frac{D}{4},\quad C_0<\frac{D}{4}.
\end{align}
We claim that 
\begin{align}
     \Lambda(3^{k}) \le \frac{D}{3^{k}}.
\end{align}
We verify this inductively in $k$. Assuming $     \Lambda(3^{k-1}) \le \frac{D}{3^{k-1}},$
  by \eqref{4100} and the fact $d>20$,
\begin{align} \label{418}
    \Lambda(3^k) \le \Big( \frac{C_0}{\sqrt{\e}}+  9C_0D^2L^{d-1}\e  +  3C_0 D L^{3-d/2}  
  + C_0 \Big)\frac{1}{3^k}.
\end{align} 
Setting $\e := (36C_0DL^{d-1})^{-1}>0,$
using the conditions \eqref{4101} and \eqref{4102}, we deduce that the quantity \eqref{418} is at most $\frac{D}{3^k},$ finishing the proof of the inductive step.

Therefore, we conclude the proof,  as for any $3^{k-1} \le r < 3^k,$
\begin{align*}
    \Lambda(r) \le \Lambda(3^{k-1} ) \le \frac{D}{3^{k-1}} \le \frac{3D}{r}.
\end{align*}

\end{proof}

We now provide the outstanding proofs of Lemmas \ref{forbid} and \ref{topverify}.

\begin{proof}[Proof of Lemma \ref{forbid}]
We claim that 
{for any collection of loops $\widetilde{\mathscr{A}}' \subset \widetilde{\mathscr{S}}$  such that $\widetilde{\mathscr{A}}'\cap \widetilde{\mathscr{A}}^{(i)} = \emptyset$,
\begin{align} \label{1212}
\widetilde B(0,i;\widetilde{\mathscr{A}}) =  \widetilde B(0,i;\widetilde{\mathscr{A}}\cup \widetilde{\mathscr{A}}')
\end{align}
and
\begin{align} \label{1313}
  (\widetilde{\mathscr{A}}\cup \widetilde{\mathscr{A}}')^{(i)} =\widetilde{\mathscr{A}}^{(i)} .
\end{align}
We first verify \eqref{1212}.
 Since $d(0,w; \widetilde{\mathscr{A}}\cup \widetilde{\mathscr{A}}') \le d(0,w; \widetilde{\mathscr{A}}) $ for any $w\in \cable$ (not necessarily a lattice point), it suffices to show  that 
\begin{align} \label{1414}
   d(0,w; \widetilde{\mathscr{A}}) > i \Rightarrow d(0,w; \widetilde{\mathscr{A}}\cup \widetilde{\mathscr{A}}') > i.
\end{align}
Suppose that  $d(0,w; \widetilde{\mathscr{A}}\cup \widetilde{\mathscr{A}}') \le i$. Let $\ell$ be an $\widetilde{\mathscr{A}}\cup \widetilde{\mathscr{A}}'-$ geodesic, i.e. all its edges and partial edges are covered by the union of ranges of the loops in $\widetilde{\mathscr{A}}\cup \widetilde{\mathscr{A}}'$, from $0$ to $w.$ Let $w'$ be a point at which $\ell$ intersects $\partial B(0,i;\widetilde{\mathscr{A}})$. Such an intersection point must exist since the distance from the origin in $\widetilde{\mathscr{A}}$ along $\ell$ jumps by $\pm 1$ and starts at $0$ and ends at a value greater than $i.$  In fact, there may be several intersection points and we choose any of them. Note that
 \begin{align*}
  d(0,w'; \widetilde{\mathscr{A}}\cup \widetilde{\mathscr{A}}')< d(0,w; \widetilde{\mathscr{A}}\cup \widetilde{\mathscr{A}}') \le i=d(0,w'; \widetilde{\mathscr{A}}),
\end{align*}
where the first inequality follows from the fact  $w\notin   \widetilde B(0,i; \widetilde{\mathscr{A}})$ and the definition of $w'$. Now, consider the initial subpath of $\ell$, say called $\ell_{0\ar w'},$ from $0$ to $w',$ which by properties of geodesics, is also a geodesic from $0$ to $w'$ in the environment $\widetilde{\mathscr{A}}\cup \widetilde{\mathscr{A}}'.$ Let it be  the sequence of lattice points $(0=w_0,w_1,\ldots, w_j=w')$ for some $j<i$. Since the distance to $w'$ from the origin in the two environments differ, $\widetilde{\mathscr{A}}'$ must intersect $\ell_{0\ar w'}.$ Let $e=[w_k,w_{k+1}]$ be the first edge that $\widetilde{\mathscr{A}}'$ intersects. Note first of all, that $\{w_0,w_1,\ldots w_k\}\subset \widetilde B(0,j; \widetilde{\mathscr{A}}).$ Thus, because of the type of $\widetilde{\mathscr{S}}$ we are considering, either the glued edge loop $\overline \gamma_{e}$ also is a member of $\widetilde{\mathscr{A}}$ or no such glued edge loop is present in $\widetilde\cL_{\widetilde{\mathscr{S}}}$. In the former case,  the intersection of $e$ with  $\widetilde B(0,j; \widetilde{\mathscr{A}})$ must be a semi-closed interval of type $[w_k, \alpha)$ for some $\alpha \in (w_k,w_k+1)$ where the latter denotes the interior of the line segment corresponding to the edge $e=[w_k,w_{k+1}],$ while if the glued edge loop is not present, the intersection will be a closed interval $[w_k,\alpha].$  In the latter case, since $w_{k}$ is covered by some loop, no loop can cover $w_{k+1}$ (by our assumption) and hence $e$ cannot be covered leading to a contradiction, since $e$ is assumed to be a part of $\ell.$ In case, a glued edge loop corresponding to $e$ is an element of  $\widetilde{\mathscr{A}}$, 
any loop, say $\widetilde \Gamma$ in $\widetilde{\mathscr{A}}'$ that intersects $e$ must be of the point or fundamental type and hence contains a partial edge of the form $[\beta,w_{k+1}]$ for some $\beta \in [w_{k}, w_{k+1}].$ Since the entire geodesic $\ell$ and hence the edge $e$ is covered by loops in 
$\widetilde{\mathscr{A}}\cup \widetilde{\mathscr{A}}',$ $\alpha$ must be in some partial edge $[\beta,w_{k+1}].$ This is a contradiction, since all such partial edges are subsets of loops of $\widetilde{\mathscr{A}}^{(i)}$, while $\alpha$ belongs to the closure of $\widetilde B(0,i; \widetilde{\mathscr{A}})$.

Next, \eqref{1313} is an immediate consequence of \eqref{1212}.

{Therefore noting that on conditioning on $ \bwt \cL_{\widetilde{\mathscr{S}}}^{\le i} = \widetilde{\mathscr{A}}$, all loops in $ \widetilde{\mathscr{A}}^{(i)}$  are forbidden  in the one-arm event  $\{\partial B(v,j;    \bwt \cL_{\widetilde{\mathscr{S}}} \setminus 
 \bwt \cL_{\widetilde{\mathscr{S}}}^{\le i}      ) \neq \emptyset \}$ by \eqref{1212} and \eqref{1313}, it follows that the conditional law of $\widetilde\cL_{\widetilde{\mathscr{S}}}$ given $\bwt \cL_{\widetilde{\mathscr{S}}}^{\le i} = \widetilde{\mathscr{A}}$ is nothing but simply the Poisson process restricted to the complement of $\widetilde{\mathscr{A}}^{(i)}$. The thinning property then establishes  \eqref{thin}.}
 }

\end{proof}

We end this section with the short proof of Lemma \ref{topverify}.

 \begin{proof}[Proof of Lemma \ref{topverify}] This follows by the observation that if an edge $e$ is such that the support of the {glued edges loop} corresponding to $e$ is contained in $\widetilde{\mathscr{S}}$, then the entire support is forbidden, i.e. contained in $\widetilde{\mathscr{A}}^{(i)},$ iff $d(0,v; \bwt \cL_{\widetilde{\mathscr{S}}})< i$ where one of its incident vertices is $v.$  Further, in this case, the point loops and any fundamental loop involving $v$ will also be forbidden as each of them intersects $v$ and hence intersects the closure of $\tilde B(0,i; \widetilde{ \mathscr{A}})$. Also, the number of edges for which the corresponding glued edge loop gets forbidden will continue to be finite since they will all be in the extrinsic ball of radius $i$ around the origin. 
 \end{proof}

The last remaining input we need to develop is a resistance estimate {delivering the bounds outlined in Section \ref{iop}.}
\section{Effective resistance} \label{section 6}
For $v\in \Z^d$ and $U\subseteq \Z^d$, recall that $R_{\textup{eff}} (v,U) $ denotes 
the effective resistance between $\{v\}$ and $U$ with respect to $\cC(0)$.
Note that the resistance from $0$ to $\partial B (0,r)$, denoted $R_{\textup{eff}}(0, \partial B (0,r))$ is always at most $r$ on the event that $\partial B     (0,r)\neq \emptyset$ and on the event that the latter is empty, $R_{\textup{eff}}(0, \partial B (0,r))=\infty.$ The main estimate of this section proves an upper bound on the probability that $R_{\textup{eff}}(0, \partial B (0,r))$ is much smaller than $r.$
\begin{proposition} \label{resist}
Let $d>20$. There exists $C, r_0>0$ such that for any $\lambda>1$ and $r_0<r\in \N$, 
    \begin{align*}
     \P( R_{\textup{eff}}(0, \partial B     (0,r)) \le \lambda^{-1} r) \le   C( \lambda^{-\frac{d-6}{d-2}} 
 r^{-1} + r^{-5})  .
    \end{align*}
\end{proposition}
Before diving into the proof, let us briefly elaborate on the idea which we believe to be of independent interest. Recall the classical  Nash-Williams (NW)  inequality to lower bound resistance (see \cite[Chapter 9]{markov}).
The inequality broadly says that the resistance between two points in a graph is not too small if one can construct many disjoint edge cutsets separating the two points, which are, themselves not too large. 

In our setting, we will describe a natural collection of cutsets. Unfortunately, however, they will not be disjoint. Very briefly, with the precise definition following shortly, the $i^{th}$ cutset will be formed of edges of the form \eqref{421}, i.e., the last edge within level $i$ on some path from $0$ to $\partial B(0,r)$. As in \cite{kn2}, we will term such an edge as a \emph{lane} at level $i$  (abbreviated to $i$-lane), i.e. an edge {within level $i$} which admits
an emanating path to  $\partial B(0,r)$ that does not return
to level $i$. 

However, an edge might participate in multiple cutsets if it belongs to a large loop. This necessitates a generalization of the NW inequality which yields useful lower bounds provided one can establish the multiplicity of the edges appearing across the cutsets, at least in some averaged sense (see Proposition \ref{resist bound} for the statement).  We expect this statement to be more broadly useful.

To bound the multiplicity of edges, we partition the levels depending on the sizes of loops the corresponding lanes intersect. We say that the level $i$ is \emph{good} if all loops, containing an $i$-lane, are small. It follows that an edge appearing in the cutset for some good level $i$ has bounded multiplicity. We then rely on our averaging method to argue that most levels are good. Finally, we bound the total number of lanes, using as input the expected volume bound from Proposition \ref{ball}.

\subsection{Generalized  NW inequality}
As indicated above, a key ingredient in the proof of Proposition \ref{resist} is a
generalization of NW
inequality which we state now.  Let $G  = (V,E)$ be a graph, and suppose that  $v\in V$ and $U \subseteq V$.   The standard  NW
inequality states that  if  $\{\Pi_1,\cdots,\Pi_m\}$ are \emph{disjoint} edge-cutsets separating $\{v\}$ from $U$ (with a unit conductance on each edge), then 
\begin{align} \label{nw}
     R_{\textup{eff}} (v,U)  \ge  \sum_{i=1}^m \frac{1}{|\Pi_i|}.
\end{align}

The following  is the generalized version of the
inequality which we will be reliant upon.
\begin{proposition}\label{resist bound}
  Let $G  = (V,E)$ be a graph. Suppose that  $v\in V$ and $U \subseteq V$.  Let $\{A_1,\cdots,A_m\}$ be the collection of edge-cutsets of $G$ (that are \emph{not} necessarily disjoint) which separate $\{v\}$ and $U$.  For $e\in E$, define $n_e := |\{i=1,\cdots,m: e\in A_i\}|.$ Then, 
    \begin{align} \label{prop62}
         R_{\textup{eff}} (v,U) \ge \frac{m^2}{\sum_{e\in E} n_e^2}.
    \end{align}
\end{proposition}
The following remark is instructive.
\begin{remark}
Let us consider the simple case when  $\{A_1,\cdots,A_m\}$ above are  disjoint. Then 
\begin{align*}
    n_e = \begin{cases}
        1 \qquad e\in A_1\cup \cdots \cup A_m, \\
        0 \qquad \text{otherwise}.
    \end{cases}
\end{align*}
Thus the bound \eqref{prop62} reads as
\begin{align} \label{80}
     R_{\textup{eff}} (v,U) \ge \frac{m^2}{\sum_{i=1}^m |A_i|}.
\end{align}
On the other hand,  the standard NW inequality \eqref{nw}  reads as 
\begin{align} \label{81}
    R_{\textup{eff}} (v,U)  \ge  \sum_{i=1}^m \frac{1}{|A_i|}.
\end{align}
By the Cauchy-Schwarz inequality, \eqref{81} is a stronger inequality than \eqref{80}. {In summary, Proposition \ref{resist bound} is weaker than the  standard NW inequality. However if average size of the $A_i$s is $O(1)$, which will indeed turn out to be the case for us, the two bounds are comparable.}
\end{remark}
\begin{proof}[Proof of Proposition \ref{resist bound}] 
    Let $\theta = ( \theta(\overrightarrow{e}))_{\overrightarrow{e}\in \overrightarrow{E}}$ be any unit flow from $\{v\}$ to $U$ (for more on flows and their relation to resistances, refer to \cite[Chapter 9]{markov}). Since $A_i$ is an edge-cutset  for each $i=1,\cdots,m,$ 
    \begin{align*}
        \sum_{e\in A_i} |\theta (e)| \ge 1.
    \end{align*}
    Note that in the above expression, we use unoriented edges and thus each edge is considered only once. As $\theta(\cdot)$ is anti-symmetric on oriented edges, the quantity $|\theta(e)|$ is unambiguous.
 Summing the above inequality over $i=1,\cdots,m,$  
    \begin{align} \label{505}
        m \le \sum_{i=1}^m  \sum_{e\in A_i} |\theta (e)|  = \sum_{e\in E} n_e  |\theta (e)|   \le  \Big(\sum_{e\in E} n_e^2 \Big)^{1/2}\Big(\sum_{e\in E}  |\theta (e)| ^2 \Big)^{1/2}  .
     \end{align}
Note that Thomson's principle says that
     \begin{align*}
          R_{\textup{eff}} (v,U)  = \inf \Big\{ \sum_{e\in E} |\theta (e)| ^2 : \theta \text{ is a unit flow from $\{v\}$ to $U$}\Big\}.
     \end{align*}
    Therefore by taking the infimum over all unit flows  $\theta$ from $\{v\}$ to $U$  in  \eqref{505}, we conclude the proof.
\end{proof}

\subsection{Effective resistance bound}
Similarly as in Section \ref{section 5}, for $i\in \N$,
 let $\widetilde \cL^i     $  (resp. $\widetilde \cL^{\le i}$) be the collection of loops in $\widetilde \cL$ intersecting with  $\partial B(0,i)$ (resp.  $ \widetilde B(0,i)$).
 
An edge $e$  in $E(\Z^d)$ is called \emph{within level $i$} if it intersects with some loop in $ \widetilde \cL^{\le i}$.

{Recalling that $ \mathcal{L}^{\le i}$ denotes the collection of {discrete loops} in $\mathcal{L}$ which share some vertex with  $ B(0,i)$,} define $ u \Join_i   v$ (resp. $ u \Join   v$) to be  the event that  there exists a loop $\Gamma \in \mathcal{L}^{\le i}$ (resp. $\Gamma \in \mathcal{L}$)  such that $u,v\in \Gamma$. Then, obviously for any $i$,
\begin{align} \label{obvious}
    u \Join_i   v \Rightarrow u \Join  v.
\end{align}
\begin{proof}[Proof of Proposition \ref{resist}]

The proof consists of four steps. Note that if $\partial B(0,r) = \emptyset$, then  $R_{\textup{eff}}(0, \partial B     (0,r))=\infty.$ Therefore, we assume  that $\partial B(0,r) \neq  \emptyset$   throughout the proof.\\

\noindent
\textbf{Step 1. Lanes and pivotal loops.}
 For  each level $i \in [r/4,r/2]$, we say that an edge $e$ present in the loop soup $\widetilde \cL,$ whose two vertices are suitably labelled  by $\textsf{v}^{1}_e$ and  $\textsf{v}^{2}_e$,   is an \emph{$i$-lane} if
 \begin{enumerate}
     \item $\textsf{v}^{1}_e$ intersects with some loop in $ \widetilde \cL^{\le i}$ {(i.e. $e$ is within level $i$).}
     \item  There exists a $\widetilde \cL \setminus \widetilde \cL^{\le i}$-path from 
$\textsf{v}^{2}_e$ 
     to $\partial B(0,r)$. 
 \end{enumerate}
Let $\mathsf{lane}(i)$ be the collection of $i$-lanes. Obviously
{\begin{align} \label{511}
    e\in \mathsf{lane}(i)  \Rightarrow e \notin E(B    (0,i)),
\end{align}}
since if $e\in E(B    (0,i))$, then any path from the vertices of $e$ to $\partial B    (0,r)$ must use a loop in $\widetilde \cL^{\le i}$.
 
We claim that for each $i$, $\mathsf{lane}(i)$ is an edge-cutset separating $\{0\}$ from $\partial B(0,r)$.  To see this,  let $\ell = (e_1,\cdots,e_m)$ be any path from $0$ to $\partial B(0,r)$. {Similarly as in \eqref{421},} setting
{\begin{align*}
    J(i):= \max \{1\le k\le m: e_k  \text{ is within level $i$}  \},
\end{align*}}
a path from the one of vertices of $e$ (i.e. $\textsf{v}^{2}_e:= \text{the common vertex of } {e_{J(i)}} \text{ and } {e_{J(i)+1}} $) to  $\partial B(0,r)$, as a sub-path of $\ell$, is a $\widetilde \cL \setminus \widetilde \cL^{\le i}$-path. Also the another vertex of $e$, say $\textsf{v}^{1}_e,$ should intersect with some loop (of fundamental or point type) in  $ \widetilde \cL^{\le i}$. Hence $e_{J(i)}$ is an $i$-lane, i.e. any path  from $0$ to $\partial B(0,r)$ should contain   an $i$-lane.
{Note that since we have assumed  $\partial B(0,r) \neq  \emptyset$, for large enough $r$, if no loop in $\widetilde \cL^{\le r}$ is large (say, everything is at most $r^{0.9}$), then for every  level $i\in [r/4,r/2],$ there exists an $i$-lane edge, i.e. $\mathsf{lane}(i) \neq \emptyset$.}
~

 For  each  $i \in [r/4,r/2]$,  {a loop in 
 $\widetilde \cL^{\le i}$ is called an \emph{$i$-pivotal loop} if it {intersects}  an $i$-lane edge. Also, similarly as in the proof of Proposition \ref{one arm}, in order to implement the averaging argument, for a constant $L>0$ which will be chosen later, define}  
 \begin{align*}
     I_1&:= \{ i \in [r/4,r/2]: \text{ all $i$-pivotal loops  have  a size less than $L$}\}, \\
      I_2&:= \{ i \in [r/4,r/2]: \text{ there exists an  $i$-pivotal loop of size at least $L$}\}.
 \end{align*}
As $\mathsf{lane}(i) \neq \emptyset$ for each level $i$, by the definition (1) of the lane, there always exists an $i$-pivotal loop.
Now setting the event
\begin{align}
    \mathcal{G}:= \{\text{Every loop in $\widetilde \cL^{\le r}$ has a size at most $r^{0.9}$}\},
\end{align}
we claim that the following implications hold for large $r$: For any edge $e$ and a level $i \in [r/4,r/2]$,  
\begin{align} \label{600}
 \{e  \in  & \mathsf{lane}(i)  \}  \cap \mathcal{G} \Rightarrow 
    \exists u\in \Z^d \text{ s.t. } \{d(0,u)=i\} \cap \{ u \Join_i \textsf{v}^{1}_e\} \cap \{  \partial B(\textsf{v}^{2}_e, {     r/4       }; \widetilde \cL \setminus \widetilde \cL^{\le i}) \neq \emptyset \} 
  \end{align} 
  and 
\begin{align} \label{601}
\{i  \in  I_2\} \cap \mathcal{G}  &\Rightarrow 
    \exists \Gamma\in \mathcal{L}^{\le i}  , \  \exists u\in \Gamma  , \ {\exists v\in \Z^d \text{ with } d^\extr(\Gamma,v) \le 1}     \text{ s.t. } \nonumber \\
    &\qquad \qquad\{|\Gamma| \ge L\} \cap  \{d(0,u)=i\}  \cap \{  \partial B(v,     r/4; \widetilde \cL \setminus \widetilde \cL^{\le i}) \neq \emptyset \}.
 \end{align}
{We will use \eqref{600} when controlling the contributions from $i\in I_1$ and rely on \eqref{601} for $i\in I_2$.}
To see \eqref{600}, as $e$ is an $i$-lane, there exists a loop $\tiloop \in \widetilde \cL^{\le i}$ {intersecting with}  $\textsf{v}^{1}_e$.  As   $\tiloop  \in \widetilde \cL^{\le i}$  shares a vertex with $  B(0,i)$, using \eqref{511} and the reasoning as in \eqref{imply}, there exists $u\in \tiloop$ such that $d(0,u)=i$.

Note that {$\Gamma:=  \textsf{Trace}(\tiloop)\in \mathcal{L}^{\le i}$} and $u, \textsf{v}^{1}_e\in \Gamma$, implying that $u \Join_i \textsf{v}^{1}_e$. In addition similarly as in \eqref{500},  for large $r$, under the event $\mathcal{G},$
\begin{align*}
   d(0,\textsf{v}^{2}_e) \le d(0,u ) + d(u,\textsf{v}^{2}_e) \le   i + |\tiloop|+1 \le  r/2+  r^{0.9} +1 \le 3r/4.
 \end{align*}
 {Hence 
we verify \eqref{600}.  \eqref{601} follows similarly as in \eqref{400}.}

~

\noindent
\textbf{Step 2. Control on $I_2$.} We claim that
there exists $C>0$ such that for any $L>0$ and $r\in \N,$ 
    \begin{align} \label{510}
\E [| I_2|\1_{\mathcal{G}}] \le  C   L^{3-d/2}.
\end{align}
To see this, by \eqref{601} and a union bound,
    \begin{align}   \label{602}
      \E [| I_2|\1_{\mathcal{G}}] &= \sum_{i=r/2}^{r/4} \P( i\in I_2, \  \mathcal{G})\nonumber \\
      &\le \sum_{i=r/2}^{r/4} \sum_{u\in \Z^d}\sum_{ \substack{ \Gamma \ni u\\ |\Gamma|  \ge L }}\sum_{ \substack{v\in \Z^d \\ d^\extr(\Gamma,v) \le 1} }  \P( d(0,u)=i,  \ \partial B(v, {     r/4       }; \widetilde \cL \setminus \widetilde \cL^{\le i}) \neq \emptyset, \  \Gamma \in \mathcal{L}^{\le i}).
          \end{align}
Let us bound the above probability. By the argument in \eqref{4111} and the result of Proposition \ref{one arm}, 
\begin{align} \label{606} 
    \P( \partial B(v, {     r/4       }; \widetilde \cL \setminus \widetilde \cL^{\le i}) \neq \emptyset \mid  \widetilde \cL^{\le i})\le \frac{C}{r}.
\end{align}
Hence as $\{d(0,u)=i\}$ and $\{\Gamma \in \mathcal{L}^{\le i}\}$ are only determined by the information of  $\widetilde \cL^{\le i}$,   by \eqref{universal2} (with $\widetilde{\mathscr{S}}=$ the collection of all loops on $\cable$),
\begin{align} \label{603}
     \P( d(0,u)=i,  \ \partial B(v, {     r/4       }; \widetilde \cL \setminus \widetilde \cL^{\le i}) \neq \emptyset, \  \Gamma \in \mathcal{L}^{\le i}) \le \frac{C}{r}\P(d(0,u)=i , \Gamma \in   \mathcal{L}^{\le i}) .
\end{align} 
Plugging this into  \eqref{602}, {following the string of inequalities in \eqref{step 33} and using  Lemma  \ref{no big2} again,}
          \begin{align*}
       \E [| I_2|\1_{\mathcal{G}}]&\le  \frac{C}{r} \sum_{i=r/2}^{r/4} \sum_{u\in \Z^d} \sum_{ \substack{ \Gamma \ni u\\ |\Gamma|  \ge L }} \sum_{ \substack{v\in \Z^d \\ d^\extr(\Gamma,v) \le 1} }  \P(d(0,u)=i , \Gamma \in  \mathcal{L}^{\le i}) \\
      & \le      \frac{C}{r}   \sum_{u\in \Z^d} \sum_{ \substack{ \Gamma \ni u\\ |\Gamma|  \ge L }} |\Gamma| \P(0\arr u,\Gamma \in \mathcal{L})  \le  C  L^{3-d/2},
    \end{align*}
    yielding \eqref{510}. Hence by Markov's inequality,  
\begin{align} \label{517}
   \P( |I_1| \le  r/8,  \ \partial B(0,r) \neq \emptyset , \  \mathcal{G})  &= \P( |I_2| \ge  r/8, \  \partial B(0,r) \neq \emptyset , \  \mathcal{G} ) \le   C   L^{3-d/2}r^{-1}.
\end{align}

\noindent
\textbf{Step 3. Bounds on multiplicity of lanes.} Now, we apply Proposition \ref{resist bound} for edge-cutsets  $\{\mathsf{lane}(i)\}_{i\in I_1}$ separating $\{0\}$ from $\partial B(0,r)$.
As has already been emphasized before,  these cutsets are not disjoint in general, i.e. each edge $e$ can belong to  $\mathsf{lane}(i)$ for multiple levels $i\in I_1$. We next bound the multiplicity of the edges. Define
\begin{align*}
   I^{L}(e) :=\{i\in I_1 : e\in \mathsf{lane}(i)\}.
\end{align*}
We claim that 
\begin{align} \label{512} 
    |I^{L}(e)| \le L .
\end{align}
We only consider the case when  $I^{L}(e)$ is not empty.
Setting $t^{L}(e)$ to be the smallest element in $    I^{L}(e) ,$ we obtain \eqref{512} as a consequence of
\begin{align} \label{513}
    I^{L}(e) \subseteq \{t^{L}(e),t^{L}(e)+1,\cdots,t^{L}(e)  +L-1\}.
\end{align}
To see this,  as $e$ is within  level  $t^{L}(e)$, there exists a loop $\tiloop\in \widetilde \cL^{\le t^{L}(e)}$ {intersecting}  $e$. Also since $\tiloop$ is a $t^{L}(e)$-pivotal loop (as $e$ is $t^{L}(e)$-lane) and $t^{L}(e)\in      I_1$, we have  $|\tiloop| < L$. As $\tiloop\in  \widetilde \cL^{\le t^{L}(e)} $ shares a vertex with $B(0,t^{L}(e) )$,  we deduce $e\in E(B(0,t^{L}(e) + L ))$, implying \eqref{513} due to \eqref{511}. 

    ~
Analogous to our notation $0\arr v$ for a vertex $v,$ let us for convenience say that $0 \arr e$ if $e$ is contained  in $\widetilde \cL$ and  two vertices of $e$ are connected to $0$ with a length at most $r$. 
Now, in order to apply Proposition \ref{resist bound}, 
we  upper bound $$\sum_{e\in E(B(0,r) )}  |I^{L}(e)|^2 =  \sum_{e\in E(\Z^d )} \1_{0\arr e} |I^{L}(e)|^2 $$ 
As $I^{L}(e) = \emptyset$ for $e \notin  \mathsf{lane} := \cup_{i=r/4}^{r/2}  \mathsf{lane}(i),$  using \eqref{512},  
\begin{align} \label{516}
 \sum_{e\in E(\Z^d )} \1_{0\arr e} |I^{L}(e)|^2   & =  \sum_{e\in E(\Z^d )} \1_{e \in  \mathsf{lane}} \1_{0\arr e} |I^{L}(e)|^2   +  \sum_{e\in E(\Z^d )} \1_{e \notin  \mathsf{lane}} \1_{0\arr e} |I^{L}(e)|^2  
    \nonumber \\
    &\le        L^2   \sum_{e \in   E(\Z^d )}  \1_{e \in  \mathsf{lane}}  .
\end{align}
For each edge $e\in E(\Z^d )$,  by \eqref{600},  
\begin{align} \label{612}
\1_{e \in  \mathsf{lane}} \1_\mathcal{G}   \le  \sum_{i=r/2}^{r/4} \1_{e \in  \mathsf{lane}(i)}  \1_\mathcal{G}   &\le \sum_{i=r/2}^{r/4}\sum_{u\in \Z^d}  
\1_{ d(0,u)=i} \1_{u\Join_i  
\textsf{v}^{1}_e} \1_{ \partial  B(\textsf{v}^{2}_e, {     r/4       }; \widetilde \cL \setminus \widetilde \cL^{\le i}) \neq \emptyset} .
\end{align}
As {$\{d(0,u)=i\} $ and  $\{u\Join_i \textsf{v}^{1}_e \} $ are measurable w.r.t. $ \widetilde \cL^{\le i}$}, by \eqref{universal2} and \eqref{606},
\begin{align*}
\P(d(0,u)=i,& \ u\Join_i  
\textsf{v}^{1}_e,  \    \partial  B(\textsf{v}^{2}_e, {     r/4       };   \widetilde \cL \setminus \widetilde \cL^{\le i}) \neq \emptyset )   \le  \frac{C}{r}  \P( d(0,u)=i, \  {u\Join_i \textsf{v}^{1}_e }). 
\end{align*} 
 Thus summing \eqref{612} over all $e \in   E(\Z^d )$ and then taking the expectation, 
\begin{align*}
  \E  \Big[ \sum_{e \in   E(\Z^d )}   \1_{e \in  \mathsf{lane}}   \1_\mathcal{G}   \Big] &\le \frac{C}{r} \sum_{e\in E(\Z^d )} \sum_{i=r/2}^{r/4}    \sum_{u\in \Z^d} 
  \P( d(0,u)=i, \  {u\Join_i \textsf{v}^{1}_e }) \\
 & \overset{\eqref{obvious}}{\le}    \frac{C}{r}   \E \Big[  \sum_{e\in E(\Z^d )}  \sum_{u\in \Z^d}
 \sum_{i=r/2}^{r/4}  \1_{ d(0,u)=i} \1_{u\Join  \textsf{v}^{1}_e}\Big] \\
  &\le  \frac{C}{r}  \E \Big[  \sum_{e\in E(\Z^d )} \sum_{u\in \Z^d }
  \1_{0 \arr u} \1_{u\Join   \textsf{v}^{1}_e} \Big] \le  \frac{C}{r}  \E \Big[  \sum_{w\in \Z^d } \sum_{u\in \Z^d } \1_{0 \arr u} \1_{u\Join 
w } {     }\Big] \le C,
\end{align*}
where the last inequality is obtained by  Lemma \ref{modified second}. Applying this to  \eqref{516},
\begin{align*}
     \E\Big[ \sum_{e\in E(\Z^d )} \1_{0\arr e} |I^{L}(e)|^2  \cdot \1_\mathcal{G}   \Big] \le   CL^2,
\end{align*}
which implies by Markov's inequality that  for any $\lambda>0,$
\begin{align} \label{518}
     \P\Big(\sum_{e\in E(\Z^d )} \1_{0\arr e} |I^{L}(e)|^2  \ge 64 \lambda r, \ \mathcal{G}\Big) \le C L^2 \lambda^{-1}r^{-1}.
\end{align}

\noindent
\textbf{Step 4. Conclusion.}
Applying Proposition \ref{resist bound} with cutsets $\{\mathsf{lane}(i)\}_{i\in I_1}$,
\begin{align*}
    R_{\textup{eff}} (0,B     (0,r) ) \ge \frac{|I_1|^2}{\sum_{e\in E(B(0,r) )}  |I^{L}(e)|^2} .
\end{align*}
 This gives 
\begin{align*}  
     \P( R_{\textup{eff}}(0, \partial B     (0,r)) &\le \lambda^{-1} r) = 
           \P(  R_{\textup{eff}}(0, \partial B     (0,r)) \le \lambda^{-1} r ,  \partial B(0,r) \neq \emptyset, \mathcal{G} )  
 + \P(\mathcal{G}^c) \nonumber \\
           &\le 
      \P\Big( \sum_{e\in E(\Z^d )} \1_{0\arr e} |I^{L}(e)|^2 \ge 64 \lambda r ,  \ \mathcal{G}\Big)  + \P( |I_1| \le  r/8, \partial B(0,r) \neq \emptyset ,   \mathcal{G} )  + \P(\mathcal{G}^c) .
\end{align*}
Using \eqref{517}, \eqref{518} and Proposition \ref{no big}, we deduce that 
\begin{align*}
    \P( R_{\textup{eff}}(0, \partial B     (0,r)) \le \lambda^{-1} r)   \le C 
 ( L^2  \lambda^{-1}r^{-1} +   L^{3-d/2}r^{-1} + r^{-5}  ) \le  
 C( \lambda^{-\frac{d-6}{d-2}}  r^{-1} + r^{-5}) ,
\end{align*}
{where we take $L = \lambda^{\frac{2}{d-2}} $.} Therefore  we
 conclude the proof.
\end{proof}

\section{Alexander-Orbach  conjecture: Proof of Theorem \ref{main}} \label{section 7}
As indicated in Section \ref{iop}, we will rely on Proposition \ref{barlow}. Thus we need to prove bounds on $\P_{\IIC}(|B(0,r)| \notin [\e r^2, \frac{1}{\e} r^2])$ as well as 
$\P_{\IIC}(R_{\textup{eff}}(0, \partial  B     (0,r) ) \le \e r).$
We have already obtained relevant bounds in the unconditional set up and so it remains to translate them into conditional bounds. We will prove bounds conditioned on the origin being connected to the infinity.  The nature of the arguments are somewhat similar to what has already appeared in the preceding sections. {Indeed, we rely on $\ite$ and $\ete$ as well as averaging arguments.}  Finally, as mentioned right after the statement of Theorem \ref{main}, the proof also holds for the sequences of measures $\nu_n$ defined in \eqref{limitiic}. While for the moment we will obtain estimates for the measures $\nu^x$, it will be quite apparent that the same estimates continue to hold even for the measures $\nu_n$. We comment on this at the end in Remark \ref{difflimit}.

~

The first two lemmas provide the necessary volume estimates. The third lemma is a general result which allows one to transfer unconditional statements into conditional ones.
\begin{lemma} \label{condition ball upper}
There exists  $C>0$ such that for any $r\in \mathbb{N}$ and $x\in \Z^d$ with   sufficiently large $|x|$  (depending on $r$),
    \begin{align*}
        \E [ |B     (0,r)| \1_{0 \ar x} ] \le C r^2 |x|^{2-d}.
    \end{align*}
\end{lemma}

The next lemma provides a complementary  lower bound on  the volume of balls, conditioned that the origin is connected to $x\in \Z^d$ (with large $|x|$).
\begin{lemma} \label{condition volume lower}
There exists $C>0$ such that
for any $\e>0$ and $r\in \mathbb{N}$, for any $x\in \Z^d$ with sufficiently large $|x|$  (depending on $r$ and $\e$),
    \begin{align*}
        \P(|B     (0,r)| \le  \e r^{2}, 0\ar x) \le C\e^{\frac{d-6}{3d-8}} |x|^{2-d}.
    \end{align*}
\end{lemma}

Finally, the following lemma provides a universal upper bound on $  \P(E \cap \{0 \ar x\})$ in terms of $  \P(E )$, for any event $E$. Recall that $\widetilde \cL^{\le r}$ denotes the collection of loops in $\widetilde \cL$ which intersect $ B(0,r)$.

\begin{lemma} \label{condition}
There exists $C>0$ such that for any $r \in \N$ and the event $E$ measurable w.r.t. $\widetilde \cL^{\le r}$,  for any $x\in \Z^d$ with sufficiently large $|x|$ (depending on  $r$ {and $\P(E)$}),
    \begin{align*}
        \P(E \cap \{0 \ar x\}) \le C (r \P(E))^{\frac{d-6}{2(2d-7)} }  |x|^{2-d}.
    \end{align*}
\end{lemma}

Given the above lemmas, we establish   Theorem \ref{main} with the aid of Proposition \ref{barlow}.
\begin{proof}[Proof of Theorem \ref{main}]
Let $\lambda>1.$
By  Lemma \ref{condition ball upper} and Markov's inequality,  
     \begin{align*}
        \P( |B(0,r) | \ge \lambda r^2
        , 0 \ar x) \le C\lambda^{-1} |x|^{2-d}.
    \end{align*}
    Hence by the two-point estimate \eqref{two point} (particularly the lower bound),
    \begin{align} \label{801}
        \P( |B(0,r) | \ge \lambda r^2 \mid 0 \ar x) \le C\lambda^{-1}.
    \end{align}
    In addition, by Lemma \ref{condition volume lower}   and the two-point estimate \eqref{two point},
      \begin{align} \label{802}
        \P( |B(0,r) | \le \lambda^{-1} r^2 \mid 0 \ar x) \le C \lambda^{-\frac{d-6}{3d-8}} .
    \end{align}
Finally, we control the effective resistance.
Note that we  have a deterministic lower bound
    \begin{align*}
        R_{\textup{eff}}(0, \partial B     (0,r))  = \frac{1}{\textsf{c}(0)\P(\tau_{\partial B     (0,r)} <\tau_0^+)} \ge  \frac{1}{2d},
    \end{align*}
    where $\textsf{c}(0)$ denotes the conductance at 0 which is at most $2d$, and $\tau_A$ (resp.  $\tau_A^+$) denotes the hitting time (resp. positive hitting time) at $A$ of a discrete-time simple random walk starting from 0 (see \cite[Chapter 9]{markov} for details).
    Hence  in the case $\lambda> 2dr,$ 
            \begin{align} \label{111}
       \{ R_{\textup{eff}}(0, \partial B     (0,r)) \le \lambda^{-1} r \} = \emptyset.
    \end{align}
    Also when $\lambda   \le  2dr,$
    as $ R_{\textup{eff}}(0, \partial B     (0,r))$ is measurable w.r.t. $\widetilde \cL^{\le r}$, by Proposition \ref{resist} and Lemma \ref{condition}, there exists $ r_0>0$ such that for any $\lambda>1$ and $r>r_0$,
        \begin{align*}
     \P( R_{\textup{eff}}(0, \partial B     (0,r)) \le \lambda^{-1} r , 0\ar x) \le   C( \lambda^{-\frac{d-6}{d-2}} 
 + r^{-4})^{\frac{d-6}{2(2d-7)}} |x|^{2-d}   \le   C  \lambda^{-\frac{d-6}{d-2} \cdot \frac{d-6}{2(2d-7)}} |x|^{2-d}    .
    \end{align*}
This and \eqref{111}, applied with  the two-point estimate \eqref{two point}, yield that for any $\lambda>1$ and $r>r_0$,
            \begin{align} \label{803}
     \P( R_{\textup{eff}}(0, \partial B     (0,r)) \le \lambda^{-1} r \mid 0\ar x) \le    C  \lambda^{-\frac{d-6}{d-2} \cdot \frac{d-6}{2(2d-7)}}.
    \end{align}
    Recall from the statement of Proposition \ref{barlow} that $U(\lambda)$ denotes the collection of $r\in \N$ such that the following conditions hold:
\begin{enumerate}
    \item $\lambda^{-1}r \le |B(0,r)| \le \lambda r$,
    \item $ R_{\textup{eff}}(0, \partial  B     (0,r) ) \ge \lambda^{-1}r.$
\end{enumerate}
Hence taking any sub-sequential limit of $\nu^{x}=\P(\cdot \subset \widetilde \cC(0)\mid 0\lbij x)$ as $|x|\rightarrow \infty,$ and using Lemma \ref{disconv} which shows that the configuration of edges converges in distribution as well, by \eqref{801}, \eqref{802} and \eqref{803},  there exist $C,c>0$ such that for any sub-sequential limit $\P_{\textup{IIC}},$
\begin{align*}
    \P_{\textup{IIC}}(r \in U(\lambda)) \ge 1-C\lambda^{-c}.
\end{align*}
Therefore we conclude the proof as a consequence of Proposition \ref{barlow}.
\end{proof}

In the remainder of the section we provide the proofs of Lemmas  \ref{condition ball upper}, \ref{condition volume lower} and \ref{condition}.
\begin{proof}[Proof of Lemma \ref{condition ball upper}]
Note that 
    \begin{align} \label{713}
        \E [ |B     (0,r)| \1_{0 \ar x} ] = \sum_{z\in \Z^d} \P( 0 \arr z, 0 \ar x).
    \end{align}

    \begin{figure}[h]
\centering
\includegraphics[scale=.5]{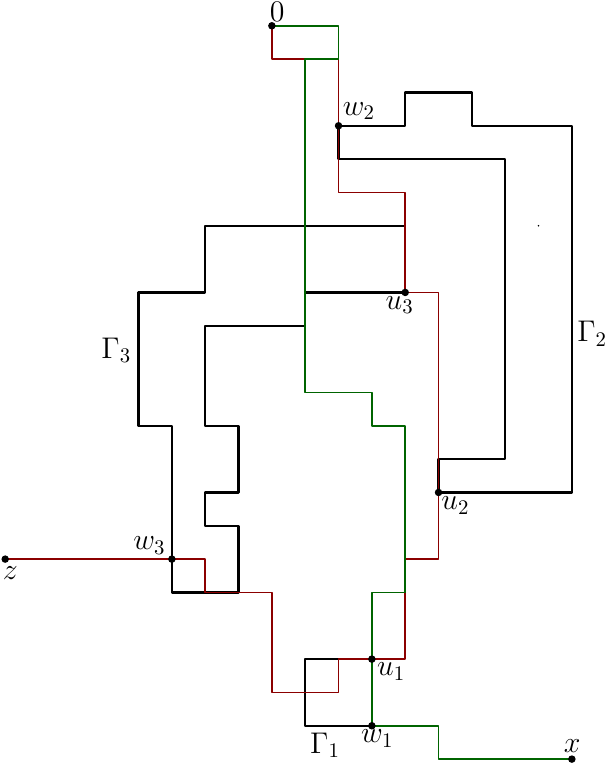}
\caption{Illustrating the $\ble$ extraction in \eqref{extraction100}. 
}
\label{7.1}
\end{figure}

By Lemma  \ref{geometry}, {with $\overline{\mathscr{S}} = \overline \cL$,} {the event $\{0 \arr z, 0 \ar x\}$ implies the existence} of $(\{\Gamma_i\}_{i=1}^3,\{u_i\}_{i=1}^3)\in \textup{\ble}$ along with $w_1,w_2,w_3\in \Z^d$ such that (see Figure \ref{7.1} for an illustration)
\begin{align*}
    \Gamma_i \in \mathcal{L},\quad d^\extr(\Gamma_i,w_i) \le 1,\quad \Gamma_i \cap B^\extr(0,r) \neq \emptyset\quad  \text{ for } i=1,2,3
\end{align*}
and  
\begin{align}\label{extraction100}
w_1\overset{\overline{\mathcal{L}} \setminus \{\overline{\Gamma}_1,\overline{\Gamma}_2,\overline{\Gamma}_3  \}}{\longleftrightarrow} x \circ w_2 \overset{\overline{\mathcal{L}} \setminus \{\overline{\Gamma}_1,\overline{\Gamma}_2,\overline{\Gamma}_3  \},r}{\longleftrightarrow}  0 \circ  w_{3} \overset{\overline{\mathcal{L}} \setminus \{\overline{\Gamma}_1,\overline{\Gamma}_2,\overline{\Gamma}_3  \},r}{\longleftrightarrow}  z \circ \{u_1\overset{\overline{\mathcal{L}} \setminus \{\overline{\Gamma}_1,\overline{\Gamma}_2,\overline{\Gamma}_3  \}}{\longleftrightarrow} u_2, u_2\overset{\overline{\mathcal{L}} \setminus \{\overline{\Gamma}_1,\overline{\Gamma}_3  \}}{\longleftrightarrow} u_3\}.
\end{align} 
Hence by a union bound  and the BKR inequality, the quantity \eqref{713} is bounded by  
\begin{align*}     
&\sum_{z\in \Z^d}\sum_{\substack{ (\{\Gamma_i\}_{i=1}^3,\{u_i\}_{i=1}^3) \in \textup{\ble} \\ \Gamma_i \cap B^\extr(0,r) \neq \emptyset, \  i=1,2,3 } }  \sum_{\substack{ w_1,w_2,w_3\in \Z^d \\ d^\extr(\Gamma_i,w_i) \le 1,\ i=1,2,3 }} \P( w_1 \ar x)  \P( w_2\arr 0 )  \P(w_3 \arr z ) \nonumber  \\
    &\qquad \qquad\qquad \qquad\qquad  \cdot \P( u_1\overset{\overline{\mathcal{L}} \setminus \{\overline{\Gamma}_1,\overline{\Gamma}_2,\overline{\Gamma}_3  \}}{\longleftrightarrow} u_2, u_2\overset{\overline{\mathcal{L}} \setminus \{\overline{\Gamma}_1,\overline{\Gamma}_3  \}}{\longleftrightarrow} u_3, \ {\Gamma_1,\Gamma_2,\Gamma_3  \in \mathcal{L}}).
\end{align*}
{As before, the degenerate cases can be dealt with in an easier way, which we will refrain from providing the details for, to contain the length of the article. 
}

{Interchanging the summation and then using the volume bound (Proposition \ref{ball}) for the summation $\sum_{z\in \Z^d} \P(w_3 \arr z ) $, the above quantity is bounded by}
\begin{align} \label{717}
      &Cr\sum_{\substack{ (\{\Gamma_i\}_{i=1}^3,\{u_i\}_{i=1}^3) \in \textup{\ble} \\ \Gamma_i \cap B^\extr(0,r) \neq \emptyset , \ i=1,2,3 } } \sum_{\substack{ w_1,w_2\in \Z^d \\ d^\extr(\Gamma_i,w_i) \le 1,\ i=1,2 }}  |\Gamma_3|  \P(w_1 \ar x)  \P( w_2\arr 0 )   \nonumber \\
     &\qquad \qquad\qquad \qquad\qquad  \cdot \P( u_1\overset{\overline{\mathcal{L}} \setminus \{\overline{\Gamma}_1,\overline{\Gamma}_2,\overline{\Gamma}_3  \}}{\longleftrightarrow} u_2, u_2\overset{\overline{\mathcal{L}} \setminus \{\overline{\Gamma}_1,\overline{\Gamma}_3  \}}{\longleftrightarrow} u_3, \ {\Gamma_1,\Gamma_2,\Gamma_3  \in \mathcal{L}}).
\end{align}
We decompose this sum into two  categories depending on the size of the {discrete loop} $\Gamma_1$, i.e., $|\Gamma_1|  < |x|^{0.9}$ and $|\Gamma_1|  \ge  |x|^{0.9}$. 
In the first case $|\Gamma_1|  < |x|^{0.9}$, recalling that $\Gamma_1\cap B^\extr(0,r) \neq \emptyset$,  for sufficiently large $|x|$ (depending on $r$),  we have    $|w_1-x| > |x|/2$ for any $w_1\in \Gamma_1$. Thus  the contribution of the  {discrete loops} $\Gamma_1$ with $|\Gamma_1|  < |x|^{0.9}$ in \eqref{717} is bounded by
\begin{align} \label{715}
     & Cr|x|^{2-d} \sum_{\substack{ (\{\Gamma_i\}_{i=1}^3,\{u_i\}_{i=1}^3) \in \textup{\ble} \\  \Gamma_i \cap B^\extr(0,r) \neq \emptyset , \  i=1,2,3 \\
     |\Gamma_1|  < |x|^{0.9}  } }  |\Gamma_1| |\Gamma_3| \sum_{\substack{ w_2\in \Z^d \\ d^\extr(\Gamma_2,w_2) \le 1}}   \P(w_2\arr 0 ) \nonumber   \\
     &\qquad \qquad\qquad \qquad\qquad \qquad\qquad \qquad\cdot \P( u_1\overset{\overline{\mathcal{L}} \setminus \{\overline{\Gamma}_1,\overline{\Gamma}_2,\overline{\Gamma}_3  \}}{\longleftrightarrow} u_2, u_2\overset{\overline{\mathcal{L}} \setminus \{\overline{\Gamma}_1,\overline{\Gamma}_3  \}}{\longleftrightarrow} u_3, \ {\Gamma_1,\Gamma_2,\Gamma_3  \in \mathcal{L}}).
\end{align}
 The above quantity \emph{without} the term $ Cr|x|^{2-d} $ is bounded by
\begin{align*}
    &\sum_{w_2\in \Z^d} \P( w_2\arr 0 )\sum_{\substack{ (\{\Gamma_i\}_{i=1}^3,\{u_i\}_{i=1}^3) \in \textup{\ble} \\ d^\extr(\Gamma_2,w_2) \le 1 } }  |\Gamma_1| |\Gamma_3| \P( u_1\overset{\overline{\mathcal{L}} \setminus \{\overline{\Gamma}_1,\overline{\Gamma}_2,\overline{\Gamma}_3  \}}{\longleftrightarrow} u_2, u_2\overset{\overline{\mathcal{L}} \setminus \{\overline{\Gamma}_1,\overline{\Gamma}_3  \}}{\longleftrightarrow} u_3, \ {\Gamma_1,\Gamma_2,\Gamma_3  \in \mathcal{L}}) \\
    &\le C\sum_{w_2\in \Z^d} \P( w_2\arr 0)  \le Cr,
\end{align*}
where we used Proposition \ref{key2} and translation invariance in the above inequality. Therefore the contribution in \eqref{717}, arising from the loops $\Gamma_1$ with  $|\Gamma_1|<  |x|^{0.9}$, is bounded by $Cr^2 |x|^{2-d}.$
~

Next we bound the contribution from {discrete loops} $\Gamma_1$ with $|\Gamma_1|  \ge  |x|^{0.9}$ in \eqref{717}. By dropping the $\ble$ condition and the connection events among $u_i$s, the contribution is bounded by 
\begin{align} \label{714}
    &Cr\sum_{\substack{\Gamma_i \cap B^\extr(0,r) \neq \emptyset , \  i=1,2,3 \\ |\Gamma_1|  \ge |x|^{0.9} } } \sum_{\substack{ w_1\in \Z^d \\ d^\extr(\Gamma_1,w_1) \le 1}} \sum_{\substack{ w_2\in \Z^d \\ d^\extr(\Gamma_2,w_2) \le 1}}   |\Gamma_3|  \P(w_1 \ar x)  \P(w_2\arr 0 )    \P({\Gamma_1,\Gamma_2,\Gamma_3  \in \mathcal{L}}) \nonumber \\
    &\le Cr \Big[ \sum_{ \Gamma_2 \cap B^\extr(0,r) \neq \emptyset } \sum_{\substack{ w_2\in \Z^d \\ d^\extr(\Gamma_2,w_2) \le 1}} \P(w_2  \arr  0)  \P(\Gamma_2 \in \mathcal{L})\Big]\Big[\sum_{ \Gamma_3 \cap B^\extr(0,r) \neq \emptyset }  |\Gamma_3| \P(\Gamma_3 \in \mathcal{L})  \Big]  \nonumber \\
    & \quad  \quad  \quad   \cdot \Big[ \sum_{\substack{ \Gamma_1 \cap B^\extr(0,r) \neq \emptyset \\ |\Gamma_1| \ge  |x|^{0.9}}} \sum_{\substack{ w_1\in \Z^d \\ d^\extr(\Gamma_1,w_1) \le 1}} 
 \P (  w_1  \ar x)\P(\Gamma_1 \in \mathcal{L}) \Big].
\end{align}
We next bound each term above.
By Lemma \ref{second moment}, 
\begin{align} \label{11111}
     \sum_{ \Gamma_2 \cap B^\extr(0,r) \neq \emptyset }\sum_{\substack{ w_2\in \Z^d \\ d^\extr(\Gamma_2,w_2) \le 1}} \P(w_2  \arr  0)\P(\Gamma_2 \in \mathcal{L}) &\le  \sum_{w_2 \in \Z^d} \P(w_2  \arr  0)\sum_{\substack{ d^\extr(\Gamma_2,w_2) \le 1}}  \P(\Gamma_2 \in \mathcal{L}) \nonumber \\
     &\le C   \sum_{w_2 \in \Z^d } \P(w_2  \arr  0) \le Cr
\end{align}
and
\begin{align} \label{11112}
      \sum_{ \Gamma_3 \cap B^\extr(0,r) \neq \emptyset }|\Gamma_3| \P(\Gamma_3 \in \mathcal{L}) \le    \sum_{z\in 
B^\extr(0,r)   }     \sum_{ \Gamma_3 \ni z}|\Gamma_3| \P(\Gamma_3 \in \mathcal{L}) \le  Cr^d.
\end{align}
Finally we  bound  the last term in \eqref{714}.  For any  $z\in B^\extr(0,r),$ by translation invariance and Lemma \ref{lemma 3.1},
\begin{align*}
    \sum_{\substack{|\Gamma_1| \ge  |x|^{0.9} \\  \Gamma_1 \ni z}} \sum_{\substack{ w_1\in \Z^d \\ d^\extr(\Gamma_1,w_1) \le 1}} 
 \P (  w_1  \ar x)\P(\Gamma_1 \in \mathcal{L})  &=     \sum_{\substack{|\Gamma'_1| \ge  |x|^{0.9} \\  \Gamma_1' \ni 0}} \sum_{\substack{ w'_1\in \Z^d \\ d^\extr(\Gamma'_1,w'_1) \le 1}} 
 \P (  w_1'  \ar x-z)\P(\Gamma_1' \in \mathcal{L}) \\
 &\le C|x-z|^{3-d} (|x|^{0.9})^{2-d/2}    \le C |x|^{4.8-1.45d},
\end{align*}
{where the last inequality holds for sufficiently large $|x|$ (depending on $r$), since $z\in B^\extr(0,r)$.} Hence by a union bound, the last term in \eqref{714} is bounded by  
\begin{align} \label{729}
 \sum_{z\in B^\extr(0,r) } \sum_{\substack{|\Gamma_1| \ge  |x|^{0.9} \\  \Gamma_1 \ni z}} \sum_{\substack{ w_1\in \Z^d \\ d^\extr(\Gamma_1,w_1) \le 1}} 
 \P (  w_1  \ar x)\P(\Gamma_1 \in \mathcal{L})\le C r^d|x|^{4.8-1.45d}.
\end{align}
{Therefore, by taking the product of \eqref{11111}, \eqref{11112} and \eqref{729}, we deduce that for sufficiently large $|x|$ (depending on $r$), the quantity \eqref{714} is bounded by}
\begin{align*}  
 Cr^{2d+2} |x|^{4.8-1.45d} \le Cr^2 |x|^{2-d},
\end{align*}
which finishes the proof.

\end{proof}

\begin{proof} [Proof of Lemma \ref{condition volume lower}]

{The proof has a similar nature as that of Proposition \ref{one arm} and to help the reader draw a parallel, we employ similar notations which we briefly recall now: $\widetilde \cL^i$ (resp.  $\widetilde \cL^{\le i}$) denotes the collection of loops in $\widetilde \cL$ which intersect $\partial B(0,i)$ (resp. $\widetilde{B}(0,i)$), and an edge $e$ is within level $i$ if it intersects  some loop in $\widetilde \cL^{\le i}$. Also  ${\mathcal{L}}^{\le i}$ denotes the set of {discrete loops} in $\mathcal{L}$ intersecting  $ B(0,i)$.}

Throughout the proof, we assume that  the event $\{0\ar x\}$ occurs and let {$\ell=(e_1,\cdots,e_m)$ be a geodesic from $0$ to $x$ in $\widetilde \cL$. Note that there may be several geodesics and we take any one of them.}  For each  level $i$, setting
\begin{align} \label{723}
    J(i):= \max \{1\le k\le m: e_k  \text{ is within level $i$}  \},
\end{align}
define 
\begin{align} \label{724}
    T_i:= \{ \text{Loops in $\widetilde \cL^i$ {intersecting} the edge $e_{J(i)}$}\}.
\end{align}
Let
\begin{align*}
    I:= \{   i\in [    r/4    , r/2     ]   :|\partial B     (0,i)| \le  10 \e r \},
\end{align*}
and
for a  constant $L>0$ which will be chosen later,
partition the set $I$ into   
\begin{align} 
    I_1 &:= \{i\in I:  |\Gamma| <L,  \  \forall \, \Gamma \in        T_i  \}, \label{741}\\
    I_2 &:=    {\{i\in I:\exists\,  \Gamma \in        T_i  \text{ such that }  |\Gamma|  \ge L \} \cap  \{i\in I: |\Gamma| \le |x|^{0.9}, \  \forall \,  \Gamma \in        T_i      \},\label{742}}\\
      I_3 &:=  {\{i\in I: \exists\,  \Gamma \in        T_i  \text{ such that } |\Gamma| >|x|^{0.9}      \}} \label{743}.
\end{align} 
As $|I| \ge r/8$ if $|B     (0,r)| \le  \e r^{2},$ {under the condition $0\ar x$,}
\begin{align} \label{630}
     \P(|B     (0,r)| \le  \e r^{2}, 0\ar x) & \le  \sum_{k=1}^3   \P(|I_k| \ge  r/24, 0\ar x).
 \end{align}
Setting
$\widetilde \cL^{i,<L} := \{\Gamma \in \widetilde \cL^i: |\Gamma| < L\},$ 
similarly as in \eqref{imply2} and \eqref{400}, we have the implications
\begin{align} \label{726}
    \{i\in I_1\} \cap \{0\ar x\} \Rightarrow 1\le|\partial B     (0,i)| \le 10 \e r,\quad  \exists\,\, {v\in V( \widetilde \cL^{i,<L}  ) \text{ s.t. } v  \overset{\widetilde \cL\setminus \widetilde \cL^{\le i} }{\longleftrightarrow}   x}
\end{align}
{(i.e.  one can take $v:=\textsf{v}_{e_{J(i)}}^{\text{out}}$, a vertex of $e_{J(i)}$ whose intrinsic distance from 0 is greater than the other vertex)} and  
\begin{align} \label{727}
    \{i  \in I_2\} \cap \{0 \ar x\} \Rightarrow 
    &\,\, \exists \text{ discrete loop } \Gamma\in {\mathcal{L}}^{\le i}, \  \exists\,\, u\in \Gamma  ,   \ {\exists v\in \Z^d \text{ with } d^\extr(\Gamma,v) \le 1}  \nonumber \\
    &\text{ s.t. }     \{|\Gamma| \ge L\} \cap  \{d(0,u)=i\}  \cap \{v \overset{ \widetilde \cL\setminus   \widetilde \cL^{\le i}   }{\longleftrightarrow} x  \}.
\end{align}
Also for sufficiently large $|x|$ (depending on $r$),  the points $v$ in \eqref{726} and \eqref{727} satisfy
\begin{align*}
|v| \le d(0,v) \le r/2+ |x|^{0.9}+1  \le |x|/2,
\end{align*}
 where the second inequality follows from the facts $i\notin I_3$ and $i \le r/2$. Thus by the similar reasoning as in \eqref{4111}, {for any collection of loops $\widetilde{\mathscr{A}}$}
 \begin{align} \label{725}
   \P(v  \overset{\widetilde \cL\setminus   \widetilde \cL^{\le i}  }{\longleftrightarrow}   x \mid  \widetilde \cL^{\le i}  = \widetilde{\mathscr{A}} ) =  \P(v  \overset{\widetilde \cL\setminus {\widetilde{\mathscr{A}}}^{(i)} }{\longleftrightarrow}   x )\le \P(v \ar x) \le C|x|^{2-d}
\end{align}
(recall {from Lemma \ref{forbid}} that ${\widetilde{\mathscr{A}}}^{(i)}$ denotes the collection of all loops which share some vertex with the closure of $B(0,i;\widetilde{\mathscr{A}})$).
~

 Hence  by \eqref{726} and a union bound, the first term in the sum in \eqref{630} is bounded by
\begin{align}\label{731}  
    \P(|I_1| \ge  r/24, &\,\,0\ar x )
 \le   \frac{24}{r}   \sum_{i= r/4      }^{     r/2       }   \P({i\in I_1, 0\ar x} ) \nonumber \\
 &\le 
 \frac{24}{r}   \sum_{i= r/4      }^{     r/2       }   \P(1\le|\partial B     (0,i)| \le 10 \e r,  \ \exists v \in V( \widetilde \cL^{i,<L}      )\text{ s.t. } v  \overset{\widetilde \cL\setminus   \widetilde \cL^{\le i}   }{\longleftrightarrow}   x ) \nonumber \\
  &\le 
 \frac{24}{r}   \sum_{i= r/4      }^{     r/2       }   C L^{d-1}  \cdot 10\e r \cdot  C|x|^{2-d} \cdot \P(\partial  B    (0,      r/4      ) \neq \emptyset) \le CL^{d-1} \e |x|^{2-d},
\end{align}
where in the third inequality we used \eqref{universal1} (with  $S=$ the collection of all loops on $\cable$ {and the entire probability space $E$}) and \eqref{725}.

Next, we bound the second term   in \eqref{630}. Note that for any  $x,u,v\in \Z^d$,  $i\in \N$ and a discrete loop $\Gamma$, applying {\eqref{725} to the bound
\eqref{universal2} (with the conditioning  on $ \widetilde \cL^{\le i} $ and the event $F=\{v \overset{ \widetilde \cL \setminus  \widetilde \cL^{\le i} }  {\longleftrightarrow} x\}$)},
\begin{align}  
    \P(d(0,u)=i, v \overset{ \widetilde \cL \setminus  \widetilde \cL^{\le i} }  {\longleftrightarrow} x, \Gamma \in{\mathcal{L}}^{\le i})  \le  C|x|^{2-d} \P(d(0,u)=i,  \Gamma \in {\mathcal{L}}^{\le i}) .
\end{align}
Thus by \eqref{727} and a union bound, using the similar argument as in  \eqref{step 33}, 
    \begin{align}   \label{728}
     \P( |I_2| \ge r/24,  0 \ar x) 
      &\le \frac{24}{r} \sum_{i=r/4}^{r/2} \sum_{u\in \Z^d}\sum_{ \substack{ \Gamma \ni  u\\ |\Gamma|  \ge L }}\sum_{  \substack{v\in \Z^d \\ d^\extr(\Gamma,v) \leq 1}} \P( d(0,u)=i, v \overset{ \widetilde \cL \setminus   \widetilde \cL^{\le i} }  {\longleftrightarrow} x, \Gamma \in {\mathcal{L}}^{\le i})  \nonumber \\
 &\le\frac{C}{r}|x|^{2-d}   \sum_{i=r/4}^{r/2}\sum_{u\in \Z^d} \sum_{ \substack{ \Gamma \ni u\\ |\Gamma|  \ge L }}\sum_{  \substack{v\in \Z^d \\ d^\extr(\Gamma,v) \leq 1}} \P(d(0,u)=i  ,  \Gamma \in {\mathcal{L}}^{\le i}) \nonumber  \\
      & {\le  \frac{C}{r}|x|^{2-d}    \sum_{u\in \Z^d} \sum_{ \substack{ \Gamma \ni u\\ |\Gamma|  \ge L }} |\Gamma| \P(0  
 \overset{r/2}{\ar}  u,\Gamma \in \mathcal{L})  \le C  L^{3-d/2}|x|^{2-d}},
    \end{align}   
  {where the  last inequality is obtained by \eqref{522} in  Lemma \ref{no big2}.}

Finally,  we bound the last term in \eqref{630}.  We have the implication
\begin{align*}
\{i  \in I_3\} \cap \{0 \ar x\} \Rightarrow 
&\exists \text{ discrete loop } \Gamma\in \mathcal{L}, \  w\in \Z^d \text{ with } d^\extr(\Gamma,w) \leq 1  \\
& \text{ s.t. }    \Gamma \cap B^\extr(0,r) \neq \emptyset , \quad |\Gamma|  \ge |x|^{0.9},\quad  w\overset{ \widetilde \cL \setminus \Gamma^{\textup{cont}} 
 }{\longleftrightarrow} x.
\end{align*}
{To see this, by the definition of $I_3$, there exists a loop $\tilde{\Gamma} \in T_i$ such that $ |\tilde{\Gamma}|  \ge |x|^{0.9}.$ As $\tilde{\Gamma}$ is connected to 0 and $0\ar x$, we have $\tilde{\Gamma} \ar x.$ Let $\ell'$ be a geodesic from $\tilde{\Gamma}$ to $x$. Then the exit point $w\in \Z^d$ of $\ell'$ from $\tilde{\Gamma}$ along with the discrete loop $\Gamma= \textsf{Trace}(\tilde{\Gamma})$ satisfy the above condition.}

Hence by a union bound, 
\begin{align} \label{761}
    \P(i\in I_3, 0\ar x) &\le \sum_{\substack { \Gamma \cap B^\extr(0,r) \neq \emptyset \\ |\Gamma|  \ge |x|^{0.9}
 }   } \sum_{  \substack{w\in \Z^d \\ d^\extr(\Gamma,w) \leq 1}} \P( w \overset{ \widetilde \cL \setminus \Gamma^{\textup{cont}} 
 }{\longleftrightarrow} x, \Gamma \in \mathcal{L}) \nonumber  \\
 &\le  \sum_{\substack { \Gamma \cap B^\extr(0,r) \neq \emptyset \\ |\Gamma|  \ge |x|^{0.9}
 }   }  \sum_{  \substack{w\in \Z^d \\ d^\extr(\Gamma,w) \leq 1}} \P(w \ar  x ) \P( \Gamma \in \mathcal{L})  \nonumber \\
 &\le {\sum_{z\in B^\extr(0,r) } \sum_{\substack{|\Gamma| \ge  |x|^{0.9} \\  \Gamma \ni z}} \sum_{\substack{ w\in \Z^d \\ d^\extr(\Gamma,w) \le 1}}  \P(w \ar  x ) \P( \Gamma \in \mathcal{L})  
 \overset{ \eqref{729}}{\le} Cr^{d}|x|^{4.8-1.45d}}.
\end{align}
{This implies that}
\begin{align}
        \P(|I_3| \ge  r/24, 0\ar x) 
        \le   \frac{24}{r}   \sum_{i= r/4      }^{     r/2       }   \P({i\in I_3, 0\ar x} )  \le {Cr^{d}|x|^{4.8-1.45d}}.
        \end{align}
Therefore, applying the previously obtained bounds to \eqref{630},
{\begin{align*}
      \P(|B     (0,r)| \le  \e r^{2}, 0\ar x) \le C ( L^{d-1} \e  +   L^{3-d/2} +  r^d|x|^{2.8-0.45d})|x|^{2-d}.
\end{align*}
{By taking $L= \e^{-\frac{2}{3d-8}} $}, as  $d>20$ (which implies $2.8-0.45d<0$),  we establish the desired estimate for sufficiently large $|x|$ depending on $r$ and $\e$.}

\end{proof}

\begin{proof} [Proof of Lemma \ref{condition}]

Since the proof is rather similar to the one just completed, we will be brief.
For a constant $K>0$ which will be chosen later,
    \begin{align} \label{661}
        \P(E ,\  0 \ar x) &\le \P(E   , \ | B(0,2r)| >Kr^2 , \ 0\ar x)+ \P( E, \   | B(0,2r)| 
 \le K r^2, \ 0\ar x).
    \end{align}    
By  Lemma \ref{condition ball upper} 
and Markov's inequality, the first term  above is bounded by
\begin{align} \label{662}
    \P(   |B(0,2r)| > Kr^2 , \ 0\ar x)\le C {K^{-1}|x|^{2-d}}.
\end{align}
Let us bound the second term  in  \eqref{661}.  As in the proof of Lemma \ref{condition volume lower}, define
\begin{align*}
    I:= \{   i \in [r,2r] :|\partial B     (0,i)| \le 2K r \}
\end{align*}
and similarly set $I_1,I_2,I_3$ as  in  \eqref{741}-\eqref{743} for some $L>0$.
{Under the condition $0\ar x$,}
as $|I| \ge r/2$ if {$|B     (0,2r)| \le  K r^{2},$ } 
 \begin{align} \label{663}
    \P( E, \  | B(0,2r)| 
 \le K r^2  , \ 0\ar x) &\le  \sum_{k=1}^3     \P(E , \ |I_k| \ge r/6 , \ 0\ar x).
\end{align}
To bound the first term above, {as $E$ is measurable w.r.t.  $ \widetilde \cL^{\le i} $,}  similarly as in  \eqref{731},
\begin{align*} 
  \P(E , \ |I_1| \ge r/6 , \ 0\ar x) 
&\le 
 \frac{6}{r}  \sum_{i=r}^{2r}    \P(E,\ 1\le|\partial B     (0,i)| \le 2Kr,  \ \exists v \in V( \widetilde \cL^{i,<L} )\text{ s.t. } v  \overset{ \widetilde \cL \setminus   \widetilde \cL^{\le i}   }{\longleftrightarrow}   x ) \\
    & \le  \frac{6}{r}   \sum_{i=r}^{2r}  C L^{d-1}  \cdot 2K   r \cdot C|x|^{2-d}  \cdot  \P(E) \le C L^{d-1}K  r|x|^{2-d}  \cdot \P(E),
\end{align*}
where in the second inequality we used \eqref{universal1}.
Next, by the same reasoning as in \eqref{728},
 \begin{align}
     \P(|I_2| \ge r/6, 0\ar x) 
     \le  C L^{3-d/2}  |x|^{2-d} .
\end{align}
Furthermore,  by the argument  in \eqref{761}, for any constant  $\e>0$,  we deduce that    for sufficiently large $|x|$ {(depending on $\e$ and $r$),}
 \begin{align}
     \P(|I_3| \ge r/6, 0\ar x) 
     \le  C \e   |x|^{2-d} .
\end{align}
Hence plugging the above estimates  into \eqref{661},
\begin{align*}
     \P(E  , \  0\ar x)\le C ( K^{-1}+  L^{d-1}K  r \P(E) +  L^{3-d/2}  +\e)  |x|^{2-d}.
\end{align*}
By taking $K= (r \P(E))^{-\frac{d-6}{2(2d-7)} }$, $L=(r \P(E))^{-\frac{1}{2d-7} } $ and  {$\e = (r \P(E))^{\frac{d-6}{2(2d-7)} }$}, we conclude the proof.  
\end{proof}

As indicated, we end with a brief remark on how the results of this section carry over as is to the setting of the measures $\nu_n$ where one considers conditioning on the event $0\ar \partial B^{\textup{ext}}(0,n)$ instead of the event $0 \ar x.$
{ 
\begin{remark}\label{difflimit}
Replacing the two point estimate $\P(0\ar x)$ from \eqref{two point} by the  extrinsic one-arm bound proved in \cite{cd},
\begin{align*}
      C_1n^{-2} \le  \P(0 \ar  \partial B^{\textup{ext}}(0,n))  \le C_2n^{-2},
            \end{align*} we have the following version  of aforementioned lemmas:
    \begin{align*}
        \E [ |B     (0,r)| \1_{0 \ar  \partial B^{\textup{ext}}(0,n)} ] \le C r^2 n^{-2},
    \end{align*}
    \begin{align*}
        \P(|B     (0,r)| \le  \e r^{2}, 0\ar  \partial B^{\textup{ext}}(0,n)) \le C\e^{\frac{d-6}{3d-8}}n^{-2},
            \end{align*}
            and
    \begin{align*}
        \P(E \cap \{0 \ar  \partial B^{\textup{ext}}(0,n)\}) \le C (r \P(E))^{\frac{d-6}{2(2d-7)} }  n^{-2}.
            \end{align*}
This then implies the conditional results analogous to Lemmas \ref{condition ball upper}, \ref{condition volume lower}, and \ref{condition},  allowing us to conclude that any subsequential limit $\P_{\textup{IIC}}$ of $\P(\cdot \mid 0\ar \partial B^{\textup{ext}}(0,n))$, as $n\rightarrow \infty$, also satisfies the condition in Proposition \ref{barlow}.
\end{remark}}

We finish with the remaining proofs of some of the straightforward but technical bounds that  have repeatedly featured in our arguments.

\section{Appendix}\label{appendix}

We set $|0|^{-k} :=0$ for any non-negative integer $k$. Throughout the appendix, we assume that $d \in \N$ is an arbitrary dimension.

\begin{lemma} \label{lemma basic}
Suppose that $\alpha_1,\alpha_2 \le 0$ satisfies $  \alpha_1+\alpha_2 <-d$. Then, there exists $C>0$ such that for any $u,v\in \Z^d$, the following holds.

\begin{enumerate}
    \item  If $\min \{\alpha_1,\alpha_2\}>-d$,  then
\begin{align}   \label{case1}
     \sum_{z\in \Z^d} |u-z|^{\alpha_1} |v-z|^{\alpha_2} \le C |u-v|^{\alpha_1+\alpha_2+d}.
\end{align} 
\item  
 If $\min \{\alpha_1,\alpha_2\}< -d$, then 
\begin{align}   \label{case2}
     \sum_{z\in \Z^d} |u-z|^{\alpha_1} |v-z|^{\alpha_2} \le C |u-v|^{\max\{\alpha_1, \alpha_2\}}.
\end{align} 
\end{enumerate}

\end{lemma}

\begin{proof}
{By translation invariance, it suffices to  consider the case when $u$ and $v$ are ``almost'' antipodal, i.e. $u+v \in B^\extr(0,1)$. In the proof, we only consider the case $u=-v$ since other cases can be dealt with similarly}. We consider the contributions from  $z\notin B^\extr(0,2|v|)$ and $z\in B^\extr(0,2|v|)$ separately.

For $z\notin B^\extr(0,2|v|),$ we have $|z+v|,|z-v| \ge |z|/2.$ Since $\alpha_1,\alpha_2\le 0,$ 
\begin{align}   \label{711}
     \sum_{z\notin B^\extr(0,2|v|)} |z+v|^{\alpha_1} |z-v|^{\alpha_2} &\le   C \sum_{z\notin B^\extr(0,2|v|)}  |z|^{\alpha_1+\alpha_2}  \nonumber \\
     &\le C \sum_{r=2|v|}^\infty r^{d-1}r^{\alpha_1+\alpha_2} \le  C|v|^{\alpha_1+\alpha_2+d},
\end{align}
where we used $\alpha_1+\alpha_2<-d$ in the last inequality.

In addition, the contribution arising from $z\in B^\extr(0,2|v|)$ is bounded by
\begin{align} \label{712}
     &\sum_{z\in B^\extr(0,2|v|)}  |z+v|^{\alpha_1} |z-v|^{\alpha_2} \nonumber \\
     &\le \sum_{z\in B^\extr(0,2|v|) \cap (B^\extr(v,|v|))^c} |z+v|^{\alpha_1} |z-v|^{\alpha_2} +\sum_{z\in B^\extr(0,2|v|) \cap (B^\extr(-v,|v|))^c} |z+v|^{\alpha_1} |z-v|^{\alpha_2}\nonumber  \\
     &\le C|v|^{\alpha_2}\sum_{z\in B^\extr(-v,3|v|)} |z+v|^{\alpha_1}  + C |v|^{\alpha_1}\sum_{z\in B^\extr(-v,3|v|)} |z-v|^{\alpha_2}\nonumber \\
     &\le C   |v|^{\alpha_2} \sum_{r=0}^{3|v|}  r^{d-1}r^{\alpha_1} +C |v|^{\alpha_1} \sum_{r=0}^{3|v|}  r^{d-1}r^{\alpha_2}  .
\end{align}
If $\alpha_1>-d$ and $\alpha_2>-d$,  then \eqref{712} is bounded by
\begin{align*}
C( |v|^{\alpha_2}\cdot |v|^{\alpha_1+d} +|v|^{\alpha_1}\cdot |v|^{\alpha_2+d} ) \le C|v|^{\alpha_1+\alpha_2+d}.
\end{align*}
If exactly one of $\alpha_1$ and $\alpha_2$ is less than $d$ (without loss of generality, assume that $\alpha_1<-d$ and $\alpha_2>-d$), then \eqref{712} is bounded by
\begin{align*}
C( |v|^{\alpha_2}+|v|^{\alpha_1}\cdot |v|^{\alpha_2+d} ) \le C|v|^{\alpha_2}.
\end{align*} 
If $\alpha_1<-d$ and $\alpha_2<-d$, 
 then \eqref{712} is bounded by
\begin{align*}
C( |v|^{\alpha_2}+|v|^{\alpha_1})\le {C|v|^{\max\{\alpha_1, \alpha_2\}}.}
\end{align*} 
Combining these bounds with   \eqref{711}, we   conclude the proof.
    \end{proof}


\begin{lemma} \label{appendix1}
Suppose that $\alpha_1,\alpha_2,\alpha_3 \le 0$ satisfy the following three conditions:
\begin{enumerate}
    \item $\min \{\alpha_1,\alpha_2,\alpha_3\}>-d$,
    \item $-2d<\alpha_2+\alpha_3  < -d$ 
 or   $-2d<\alpha_1+\alpha_3  < -d$,
    \item $\alpha_1+\alpha_2+\alpha_3  < -2d.$
\end{enumerate} 
Then
there exists $C>0$ such that for any $u,v\in \Z^d$,  
    \begin{align} \label{720}
      \sum_{z_1,z_2\in \Z^d} |u-z_1|^{\alpha_1} |v-z_2|^{\alpha_2}  |z_1-z_2|^{\alpha_3} \le  C|u-v|^{\alpha_1+\alpha_2+\alpha_3+2d}.
    \end{align}
In addition,
    \begin{align} \label{7200}
        \lim_{K \rightarrow \infty } \sum_{ \substack{z_1,z_2\in \Z^d \\ |z_1| \ge K }  } |z_1|^{\alpha_1}  |z_2|^{\alpha_2}  |z_1-z_2|^{\alpha_3} = 0.
    \end{align}
\end{lemma}
\begin{proof}
Without loss of generality, we   assume that $-2d<\alpha_2+\alpha_3  < -d$.  We first verify \eqref{720}.
By Lemma \ref{lemma basic}, LHS of \eqref{720} is bounded by
\begin{align*}
     \sum_{z_1 \in \Z^d}      |u-z_1|^{\alpha_1} &\sum_{z_2\in \Z^d} |v-z_2|^{\alpha_2}  |z_2-z_1|^{\alpha_3}  \\
     &\le C \sum_{z_1 \in \Z^d}      |u-z_1|^{\alpha_1} |v-z_1|^{\alpha_2+\alpha_3+d} \le C|u-v|^{\alpha_1+\alpha_2+\alpha_3+2d},
\end{align*}
where in the last inequality we used the condition $\alpha_2+\alpha_3+d  > -d$ and $\alpha_1 + \alpha_2+\alpha_3+d < -d.$

In addition, by Lemma \ref{lemma basic} again,   for any $K \in \N,$
\begin{align*}
    \sum_{ \substack{|z_1| \ge K }  }  |z_1|^{\alpha_1}   \sum _{z_2\in \Z^d} |z_2|^{\alpha_2}  |z_1-z_2|^{\alpha_3}  & \le   \sum_{ \substack{|z_1| \ge K }  }  |z_1|^{\alpha_1} \cdot |z_1|^{\alpha_2+\alpha_3+d} \\
    &\le C \sum_{r=K}^\infty r^{d-1} \cdot r^{\alpha_1+\alpha_2+\alpha_3+d}  \le  C K^{\alpha_1+\alpha_2+\alpha_3+2d},
\end{align*}
 concluding the proof of \eqref{7200}.
\end{proof}

\bibliographystyle{plain}
 \bibliography{GFF}

\begin{thebibliography}{10}

\bibitem{triangle1}
Michael Aizenman and Charles~M Newman.
\newblock Tree graph inequalities and critical behavior in percolation models.
\newblock {\em Journal of Statistical Physics}, 36(1-2):107--143, 1984.

\bibitem{AO82}
Shlomo Alexander and Raymond Orbach.
\newblock Density of states on
  fractals:{\guillemotleft}fractons{\guillemotright}.
\newblock {\em Journal de Physique Lettres}, 43(17):625--631, 1982.

\bibitem{lap1}
Nalini Anantharaman.
\newblock Topologie des hypersurfaces nodales de fonctions al{\'e}atoires
  gaussiennes [d'apr{\`e}s nazarov et sodin, gayet et welschinger].
\newblock {\em Ast{\'e}risque}, 2015:1116, 2017.

\bibitem{barlow1}
Martin Barlow.
\newblock Which values of the volume growth and escape time exponent are
  possible for a graph?
\newblock {\em Revista Matem{\'a}tica Iberoamericana}, 20(1):1--31, 2004.

\bibitem{gaussian}
Martin~T Barlow.
\newblock Random walks on supercritical percolation clusters.
\newblock {\em Annals of Probability}, pages 3024--3084, 2004.

\bibitem{barlow}
Martin~T Barlow, Antal~A J{\'a}rai, Takashi Kumagai, and Gordon Slade.
\newblock Random walk on the incipient infinite cluster for oriented
  percolation in high dimensions.
\newblock {\em Communications in mathematical physics}, 278(2):385--431, 2008.

\bibitem{triangle2}
David~J Barsky and Michael Aizenman.
\newblock Percolation critical exponents under the triangle condition.
\newblock {\em The Annals of Probability}, pages 1520--1536, 1991.

\bibitem{bk4}
Christian Borgs, Jennifer~T Chayes, and Dana Randall.
\newblock The van den berg-kesten-reimer inequality: a review.
\newblock {\em Perplexing problems in probability: Festschrift in honor of
  Harry Kesten}, pages 159--173, 1999.

\bibitem{cd}
Zhenhao Cai and Jian Ding.
\newblock One-arm exponent of critical level-set for metric graph gaussian free
  field in high dimensions.
\newblock {\em arXiv preprint arXiv:2307.04434}, 2023.

\bibitem{loop1}
Yinshan Chang and Art{\"e}m Sapozhnikov.
\newblock Phase transition in loop percolation.
\newblock {\em Probability Theory and Related Fields}, 164(3-4):979--1025,
  2016.

\bibitem{ch}
JT~Chayes and L~Chayes.
\newblock On the upper critical dimension of bernoulli percolation.
\newblock {\em Communications in mathematical physics}, 113(1):27--48, 1987.

\bibitem{dw}
Jian Ding and Mateo Wirth.
\newblock Percolation for level-sets of gaussian free fields on metric graphs.
\newblock {\em The Annals of Probability}, 48(3):1411--1435, 2020.

\bibitem{drs}
Alexander Drewitz, Bal{\'a}zs R{\'a}th, and Art{\"e}m Sapozhnikov.
\newblock On chemical distances and shape theorems in percolation models with
  long-range correlations.
\newblock {\em Journal of Mathematical Physics}, 55(8), 2014.

\bibitem{gff5}
Hugo Duminil-Copin, Subhajit Goswami, Pierre-Fran{\c{c}}ois Rodriguez, and
  Franco Severo.
\newblock Equality of critical parameters for percolation of gaussian free
  field level sets.
\newblock {\em Duke Mathematical Journal}, 172(5):839--913, 2023.

\bibitem{markovian}
Patrick~J Fitzsimmons and Jay~S Rosen.
\newblock Markovian loop soups: permanental processes and isomorphism theorems.
\newblock {\em Electron. J. Probab}, 19(60):1--30, 2014.

\bibitem{ganguly}
Shirshendu Ganguly and James~R Lee.
\newblock Chemical subdiffusivity of critical 2d percolation.
\newblock {\em Communications in Mathematical Physics}, pages 1--20, 2020.

\bibitem{grs}
Subhajit Goswami, Pierre-Fran{\c{c}}ois Rodriguez, and Franco Severo.
\newblock On the radius of gaussian free field excursion clusters.
\newblock {\em The Annals of Probability}, 50(5):1675--1724, 2022.

\bibitem{hara1}
Takashi Hara and Gordon Slade.
\newblock Mean-field critical behaviour for percolation in high dimensions.
\newblock {\em Communications in Mathematical Physics}, 128(2):333--391, 1990.

\bibitem{hara2}
Takashi Hara and Gordon Slade.
\newblock The scaling limit of the incipient infinite cluster in
  high-dimensional percolation. i. critical exponents.
\newblock {\em Journal of Statistical Physics}, 99:1075--1168, 2000.

\bibitem{hara3}
Takashi Hara, Gordon Slade, and Remco van~der Hofstad.
\newblock Critical two-point functions and the lace expansion for spread-out
  high-dimensional percolation and related models.
\newblock {\em The Annals of Probability}, 31(1):349--408, 2003.

\bibitem{iic2}
Markus Heydenreich, Remco van~der Hofstad, and Tim Hulshof.
\newblock High-dimensional incipient infinite clusters revisited.
\newblock {\em Journal of Statistical Physics}, 155:966--1025, 2014.

\bibitem{fkg2}
Svante Janson.
\newblock Bounds on the distributions of extremal values of a scanning process.
\newblock {\em Stochastic processes and their applications}, 18(2):313--328,
  1984.

\bibitem{kesten1}
Harry Kesten.
\newblock The incipient infinite cluster in two-dimensional percolation.
\newblock {\em Probability theory and related fields}, 73:369--394, 1986.

\bibitem{kesten}
Harry Kesten.
\newblock Subdiffusive behavior of random walk on a random cluster.
\newblock In {\em Annales de l'IHP Probabilit{\'e}s et statistiques},
  volume~22, pages 425--487, 1986.

\bibitem{kol}
AN~Kolmogorov.
\newblock On the solution of a problem in biology.
\newblock {\em Izv. NII Matem. Mekh. Tomskogo Univ}, 2:7--12, 1938.

\bibitem{kn2}
Gady Kozma and Asaf Nachmias.
\newblock The alexander-orbach conjecture holds in high dimensions.
\newblock {\em Inventiones mathematicae}, 178(3):635--654, 2009.

\bibitem{kn1}
Gady Kozma and Asaf Nachmias.
\newblock Arm exponents in high dimensional percolation.
\newblock {\em Journal of the American Mathematical Society}, 24(2):375--409,
  2011.

\bibitem{rw}
Gregory~F Lawler and Vlada Limic.
\newblock {\em Random walk: a modern introduction}, volume 123.
\newblock Cambridge University Press, 2010.

\bibitem{sle1}
Gregory~F Lawler, Oded Schramm, and Wendelin Werner.
\newblock One-arm exponent for critical 2d percolation.
\newblock {\em Electronic Journal of Probability}, 7:Paper--No, 2002.

\bibitem{le1}
Yves Le~Jan.
\newblock {\em Markov Paths, Loops and Fields: {\'E}cole d'{\'E}t{\'e} de
  Probabilit{\'e}s de Saint-Flour XXXVIII--2008}.
\newblock Springer, 2011.

\bibitem{fkg}
Yves Le~Jan and Sophie Lemaire.
\newblock Markovian loop clusters on graphs.
\newblock {\em Illinois Journal of Mathematics}, 57(2):525--558, 2013.

\bibitem{ls}
Joel~L Lebowitz and Hubert Saleur.
\newblock Percolation in strongly correlated systems.
\newblock {\em Physica A: Statistical Mechanics and its Applications},
  138(1-2):194--205, 1986.

\bibitem{lee}
James~R Lee.
\newblock Relations between scaling exponents in unimodular random graphs.
\newblock {\em Geometric and Functional Analysis}, 33(6):1539--1580, 2023.

\bibitem{markov}
David~A Levin and Yuval Peres.
\newblock {\em Markov chains and mixing times}, volume 107.
\newblock American Mathematical Soc., 2017.

\bibitem{lupu}
Titus Lupu.
\newblock From loop clusters and random interlacements to the free field.
\newblock {\em The Annals of Probability}, 44(3):2117--2146, 2016.

\bibitem{lap2}
Stanislav~A Molchanov and Andrei~K Stepanov.
\newblock Percolation in random fields. i.
\newblock {\em Theoretical and Mathematical Physics}, 55(2):478--484, 1983.

\bibitem{pr}
Serguei Popov and Bal{\'a}zs R{\'a}th.
\newblock On decoupling inequalities and percolation of excursion sets of the
  gaussian free field.
\newblock {\em Journal of Statistical Physics}, 159(2):312--320, 2015.

\bibitem{pt}
Serguei Popov and Augusto Teixeira.
\newblock Soft local times and decoupling of random interlacements.
\newblock {\em Journal of the European Mathematical Society},
  17(10):2545--2593, 2015.

\bibitem{bk3}
David Reimer.
\newblock Proof of the van den berg--kesten conjecture.
\newblock {\em Combinatorics, Probability and Computing}, 9(1):27--32, 2000.

\bibitem{lap3}
Peter Sarnak.
\newblock Topologies of the zero sets of random real projective hyper-surfaces
  and of monochromatic waves.
\newblock In {\em Talk delivered at Random geometries/Random topologies
  conference, slides available online at https://www. math. ethz.
  ch/fim/conferences/past-conferences/2017/random-geometries-topologies/talks.
  html}, 2017.

\bibitem{sle2}
Stanislav Smirnov.
\newblock Critical percolation in the plane: conformal invariance, cardy's
  formula, scaling limits.
\newblock {\em Comptes Rendus de l'Acad{\'e}mie des Sciences-Series
  I-Mathematics}, 333(3):239--244, 2001.

\bibitem{bk2}
J~Van~den Berg and U~Fiebig.
\newblock On a combinatorial conjecture concerning disjoint occurrences of
  events.
\newblock {\em The Annals of Probability}, pages 354--374, 1987.

\bibitem{bk1}
Jacob Van Den~Berg and Harry Kesten.
\newblock Inequalities with applications to percolation and reliability.
\newblock {\em Journal of applied probability}, 22(3):556--569, 1985.

\bibitem{iic}
Remco Van~der Hofstad and Antal~A J{\'a}rai.
\newblock The incipient infinite cluster for high-dimensional unoriented
  percolation.
\newblock {\em Journal of statistical physics}, 114(3-4):625--663, 2004.

\bibitem{werner}
Wendelin Werner.
\newblock On clusters of brownian loops in d dimensions.
\newblock {\em In and Out of Equilibrium 3: Celebrating Vladas Sidoravicius},
  pages 797--817, 2021.

\end{thebibliography}

\end{document}